\documentclass[12pt]{article}

\usepackage{graphicx,amsmath,amscd,amsfonts,amsthm,amssymb,verbatim,stmaryrd,fullpage}
\newtheorem{theorem}{Theorem}
\newtheorem{lemma}[theorem]{Lemma}
\newtheorem{prop}[theorem]{Proposition}

\theoremstyle{remark}

\newcommand{\R}{\mathbb R}
\newcommand{\C}{\mathbb C}
\newcommand{\N}{\mathbb N}
\newcommand{\Z}{\mathbb Z}
\newcommand{\Q}{\mathbb Q}
\newcommand{\F}{\mathbb F}

\newcommand{\p}{\mathfrak p}

\newcommand{\Aa}{\mathfrak a}
\newcommand{\Bb}{\mathfrak b}

\newcommand{\OO}{\mathcal O}

\newcommand{\HH}{\mathbb H}

\newcommand{\Ss}{\mathcal S}
\newcommand{\Tt}{\mathcal T}
\newcommand{\Mm}{\mathcal M}
\newcommand{\Pp}{\mathcal P}

\newcommand{\tr}{\text{tr}}

\newcommand{\ord}{\text{ord}}
\newcommand{\Sym}{\text{Sym}}
\newcommand{\sym}{\text{sym}}
\newcommand{\Hom}{\text{Hom}}
\newcommand{\Res}{\text{Res}}

\begin{document}

\title{Mass Equidistribution for Automorphic Forms of Cohomological Type on $GL_2$}

\author{Simon Marshall}

\maketitle

\begin{abstract}
We extend Holowinsky and Soundararajan's proof of quantum unique ergodicity for holomorphic Hecke modular forms on $SL(2,\Z)$, by establishing it for automorphic forms of cohomological type on $GL_2$ over an arbitrary number field which satisfy the Ramanujan bounds.  In particular, we have uncondtional theorems over totally real and imaginary quadratic fields.  In the totally real case we show that our result implies the equidistribution of the zero divisors of holomorphic Hecke modular forms, generalising a result of Rudnick over $\Q$.
\end{abstract}

\section{Introduction}
\label{weightintro}

One of the central problems in the subject of quantum chaos is to understand the behaviour of high energy Laplace eigenfunctions on a Riemannian manifold $M$.  There is an important conjecture of Rudnick and Sarnak \cite{RS} which predicts one aspect of this behaviour in the case when $M$ is compact and negatively curved, namely that the microlocal lifts of eigenfunctions tend weakly to Liouville measure on the unit tangent bundle.  This is known as the quantum unique ergodicity conjecture, and has as a corollary that the $L^2$ mass of eigenfunctions becomes weakly equidistributed on $M$.  We refer the reader to \cite{Li1, Li2, RS, Sn, SV, Ze1, Ze2} for many illuminating discussions and interesting results related to this conjecture.

In this paper we shall deal with a variant of Rudnick and Sarnak's conjecture which replaces Laplace eigenfunctions with certain modular forms.  This may be described most easily in the case of the modular surface $X = SL(2,\Z) \backslash \HH^2$, where the objects we shall consider are holomorphic modular forms of large weight, or equivalently sections of high tensor powers of the line bundle of holomorphic differentials on $X$.  If $f$ is a holomorphic modular cusp form of weight $k$, the analogue of the $L^2$ mass of $f$ is the Petersson measure

\begin{equation*}
\mu_f = y^k |f(z)|^2 dv,
\end{equation*}

where $dv$ denotes the hyperbolic volume.  The measure $\mu_f$ is invariant under $SL(2,\Z)$, and we may suppose that $f$ has been normalised so that it descends to a probability measure on $X$.  The analogue of the quantum unique ergodicity conjecture for holomorphic forms is then to show that the measures $\mu_f$ tend weakly to the hyperbolic volume as the weight of $f$ tends to infinity.  This is very much in the spirit of the original conjectures, with the Cauchy-Riemann equations replacing the Laplace operator and the weight $k$ playing the role of the eigenvalue, and was considered in \cite{LS, Sr1}.

There are two main differences between this conjecture and the classical form of QUE.  The first is that no microlocal lift is known for holomorphic forms, so we are restricted to considering equidistribution on $X$ rather than its unit tangent bundle, and ergodic methods may not presently be applied to this problem.  The second is that the literal analogue of the conjecture fails because the space of cusp forms is large, and contains elements like $\Delta^k$ (where $\Delta$ is Ramanujan's cusp form) whose mass is not equidistributing.  From a number theoretic point of view it is natural to deal with this multiplicity issue by requiring $f$ to be a Hecke eigenform, which gives a refinement of the conjecture known as arithmetic QUE.  This is a natural condition to impose, as Watson's triple product formula \cite{W} illustrates that the generalised Riemann hypothesis would imply QUE for holomorphic Hecke eigenforms with the optimal rate of equidistribution.  The first unconditional results on this conjecture were obtained by Sarnak \cite{Sr1}, who showed that it was true for dihedral forms, and Luo and Sarnak \cite{LS}, who showed that it was true for almost all eigenforms of weight at most $k$.

In \cite{Ho, HS, So2}, Holowinsky and Soundararajan established QUE for all holomorphic Hecke eigenforms on the modular surface $X$, or more generally any noncompact congruence hyperbolic surface.  Their proof is a combination of two different approaches, one based on bounding the $L$ value appearing in Watson's triple product formula and the other on bounding shifted convolution sums, and which complement each other in a remarkable way to produce the full result.  In this paper we extend Holowinsky and Soundararajan's methods to prove QUE for holomorphic Hecke eigenforms on $GL_2$ over a totally real number field, or more generally for automorphic forms of cohomological type on $GL_2$ over an arbitrary number field and which satisfy the Ramanujan bounds.  For simplicity, we assume our fields to have narrow class number one throughout the paper, but this is not essential.

We shall give a simple outline of our results here, before describing them more fully once we have introduced the required notation.  First let us assume that the field $F$ over which we are working is totally real with narrow class number one.  Let $\OO$ be the ring of integers of $F$, and let $\Gamma = GL^+(2,\OO)$ be the subgroup of $GL(2,\OO)$ of elements with totally positive determinant.  Fix $\nu > 0$, and let $\{ f_n \}$ be a sequence of holomorphic Hecke modular forms for $\Gamma$ whose weights $k_n = (k_{i,n})$ satisfy $k_{i,n} \ge k_{j,n}^\nu$ for all $i$ and $j$.  Our result is:

\begin{theorem}
\label{hilbertQUE}
The normalised Petersson probability measures $\mu_n = y^{k_n} |f_n(z)|^2 dv$ tend weakly to the uniform measure on $\Gamma \backslash (\HH^2)^n$ as $k \rightarrow \infty$.
\end{theorem}

As a consequence of theorem \ref{hilbertQUE}, we prove that if $k$ is a fixed positive weight and $\{ f_N \}$ a sequence of holomorphic Hecke forms of weight $Nk$, then the zero divisors $Z_N$ of $f_N$ become equidistributed on $\Gamma \backslash (\HH^2)^n$, either as Lelong $(1,1)$ currents or as measures defined by integration over $Z_N$ with respect to the volume form of the induced Riemannian metric.  This generalises a result of Rudnick \cite{Ru} on the equidistribution of zeros of Hecke modular forms on $SL(2,\Z)$.

The statement of the mixed case of our result is a little more involved, and for now we will give it only in the case of a Bianchi manifold $Y = \Gamma \backslash \HH^3$, where $\OO$ is the ring of integers in an imaginary quadratic field $F$ and $\Gamma = SL(2,\OO)$.  Let $E_d$ be the representation $\Sym^d \otimes \overline{ \Sym}^d$ of $SL(2,\C)$, and let $V_d$ be the associated local system on $Y$ which we equip with a certain canonical positive definite norm.  The objects whose equidistribution we shall now consider may be thought of either as 1-forms in $A^1(Y, V_{d})$ which are harmonic with respect to the norm on $V_d$ and are eigenforms of the Hecke operators, or as the lowest $K$-types in the corresponding automorphic representations of cohomological type on $\Gamma \backslash SL(2,\C)$.

We may define analogues of the Petersson mass using either of these viewpoints.  A harmonic Hecke form $\omega \in A^1(Y, V_{d})$ is a section of $T^*Y \otimes V_{d}$ to which we may associate the measure $\mu_\omega = \| \omega \|^2 dv$, where $\| \cdot \|$ is the tensor product of the norms on $T^*Y$ and $V_d$ and $dv$ is the hyperbolic volume.  Alternatively, if $\phi \in \pi$ is a vector of lowest $K$-type we may push the measure $| \phi |^2 dg$ from $\Gamma \backslash SL(2,\C)$ down to $Y$ to obtain one differing from $\mu_\omega$ by a constant multiple.  With this notation, we may state our result:

\begin{theorem}
The measures $\mu_\omega$ tend weaky to the hyperbolic measure on $Y$ as $d \rightarrow \infty$.
\end{theorem}

\subsection{Structure of the Paper}

We introduce the manifolds and automorphic forms with which we shall work in section \ref{weightprelims}, before giving the full statements of our results in section \ref{weightresults}.  We describe the structure of the proof in section \ref{weightreview}.  As our proof is a direct generalisation of the methods used by Holowinsky and Soundararajan over $\Q$, we do this by first giving an overview of their proof before explaining the modifications which must be made to extend it to a number field.  Sections \ref{sievereal} to \ref{holosieve} contain the generalisation of Holowinsky's method of shifted convolution sums, and section \ref{someproof} contains the extension of Soundararajan's approach of triple product identities and weak subconvexity.  In section \ref{mixedconclusion} we combine these two approaches to establish our main result, and in section \ref{currents} we prove the generalisation of Rudnick's theorem on the equidistribution of zero divisors of holomorphic forms.  Section \ref{appendix} is an appendix which contains various computations which are needed in the course of the proofs.

{\bf Acknowledgements}: We would like to thank our adviser Peter Sarnak for suggesting this problem as part of our thesis, and providing much guidance and encouragement in the course of our work.

\section{Definitions and Notation}
\label{weightprelims}

\subsection{Arithmetic Manifolds}
\label{weightprelims1}

We begin by introducing the manifolds on which we shall work.  Let $F$ be a number field of narrow class number one with degree $n$ and $r$ infinite places, of which $r_1$ are real and $r_2$ are complex.  Let $\F = F \otimes_{\Q} \R$, and $\F^+$ be the subset of totally positive elements.  If $\OO$ is the ring of integers of $F$, let $\OO^+ = \OO \cap \F^+$.  Define $\mu_+$ to be the group of totally positive roots of unity in $F$, which is the ordinary unit group if $F$ is totally complex and trivial otherwise, and set $\omega_+ = |\mu_+|$.  Let $G_i = GL^+(2,\R)$ for $i \le r_1$ and $GL(2,\C)$ otherwise, and $G = G_1 \times \ldots \times G_r = GL^+(2,\F)$.  $Z_i$ will denote the centre of $G_i$, and $\overline{G_i} = G_i/Z_i$.  $N$ will denote the usual unipotent subgroup of $G$ and $\overline{G}$, and $A$ and $M$ the maximal split and compact diagonal subgroups with lower entry equal to 1.  $K = K_1 \times \ldots \times K_r$ will be the maximal compact.   Let $\Gamma = GL^+ (2, \OO)$ be the integral matrices with totally positive determinant, and define $\Gamma_\infty = \Gamma \cap B$ and $\Gamma_U = \Gamma \cap U$.

Let $\HH_F = \overline{G} / K$ be identified with $(\HH^2)^{r_1} \times (\HH^3)^{r_2}$, and introduce on it the following co-ordinates:

\begin{eqnarray*}
z & = & (z_1, \ldots, z_r), \\
z_i & = & x_i + i y_i, \quad x_i, y_i \in \R \text{ for } i \le r_1, \\
z_i & = & x_i + j y_i, \quad x_i \in \C, \; y_i \in \R \text{ for } i > r_1, \\
x & = & (x_1, \ldots, x_r), \quad y = (y_1, \ldots, y_r).
\end{eqnarray*}

We let

\begin{equation*}
dv = \bigwedge_{i \le r_1} y_i^{-2} dx_i dy_i \wedge \bigwedge_{i > r_1} \frac{y_i^{-3}}{2i} dx_i d\overline{x}_i dy_i
\end{equation*}

be the product of standard hyperbolic measures on $\HH_F$.  We define $X = \Gamma \backslash \overline{G}$ and $Y = \Gamma \backslash \HH_F$, so that automorphic forms on $GL_2/F$ of full level are equivalent to Hecke eigenforms on $X$.

Throughout the paper, we will use a multi-index notation for co-ordinates on $\HH_F$ and the weights of automorphic forms; for instance, if $y$ is the co-ordinate on $\HH_F$ introduced above and $k$ is an $r$-tupe of integers, the expression $y^k$ will denote $\prod y_i^{k_i}$.  If $\delta_i$ is defined to be 1 for $i \le r_1$ and 2 otherwise, for any $r$-tuple $x$ we denote $\prod x_i^{\delta_i}$ by $Nx$, and the maximum of $|x_i|$ by $\| x \|$.

\subsection{Eisenstein Series}
\label{weightprelims2}

In addition to the usual complete Eisenstein series, we will work with two kinds of incomplete Eisenstein series which we term `pure incomplete Eisenstein series' and `unipotent Eisenstein series'.  To define them, we must introduce the multiplicative characters of the group $\F_+^\times / \OO_+^\times$ following Hecke.  Let $\epsilon_j = ( \epsilon_j^1, \ldots, \epsilon_j^r )$, $j = 1, \ldots, r-1$ be generators of $\OO^\times_+$, and define $A$ as

\begin{equation*}
A = \left( \begin{array}{cccc} 1/n & \log | \epsilon_1^1 | & \ldots & \log | \epsilon_{r-1}^1 | \\
 \vdots & & & \\
 1/n & \log | \epsilon_1^r | & \ldots & \log | \epsilon_{r-1}^r | 
 \end{array} \right)
\end{equation*}

with inverse

\begin{equation*}
A^{-1} = \left( \begin{array}{cccc} 1 & 1 & \ldots & 2 \\
 e_1^1 & e_2^1 & \ldots & e_n^1 \\
 \vdots & & & \\
 e_1^{n-1} & e_2^{n-1} & \ldots & e_n^{n-1} 
 \end{array} \right).
\end{equation*}

(Here the first row of $A^{-1}$ contains $r_1$ 1's and $r_2$ 2's.)  We may now define the characters $\lambda_m(y)$ for $m \in \Z^{r-1}$ by the following formula:

\begin{eqnarray*}
\lambda_m(y) & = & \prod_{p=1}^n \prod_{q=1}^{n-1} |y_p|^{2 \pi i m_q e_p^q } \\
 & = & \exp \left( \sum_{p=1}^r \beta(m,p) \log |y_p| \right),
\end{eqnarray*}

\begin{equation}
\label{beta}
\text{where } \beta(m,p) = 2 \pi i \sum_{q=1}^{r-1} m_q e_p^q.
\end{equation}

As $\lambda_m$ is invariant under the action of $\OO_+^\times$ on $\F^+$, it may be extended to a Hecke character on $F$ via the isomorphism $\F / \OO^\times \simeq \F^+ / \OO_+^\times$.

Having defined $\lambda_m$, we may let $E(z, s, m)$ denote the usual Eisenstein series associated to the character $Ny^s \lambda_m(y)$ of the cusp of $X$.  The pure incomplete Eisenstein series are formed by automorphising a function on $\Gamma_\infty \backslash \HH_F$ which is invariant under $U$ and transforms according to $\lambda_m$ under the norm one elements of the diagonal.  They are determined by an index $m \in \Z^{r-1}$ and a function $\psi \in C^\infty_0 (\R^+)$, and are defined as

\begin{equation*}
E( \psi, m | z) = \sum_{\gamma \in \Gamma_\infty \backslash \Gamma } \psi( Ny ( \gamma z) ) \lambda_m( y( \gamma z ) ).
\end{equation*}

The unipotent Eisenstein series are formed by automorphising a function on $\HH_F$ which is only invariant under $U$.  They are determined by a function $g \in C^\infty_0 ( \R^r_+)$, and defined as

\begin{equation*}
E( g | z) = \sum_{\gamma \in \Gamma_U \backslash \Gamma } g( y( \gamma z) ).
\end{equation*}

We note that it is less standard to form Eisenstein series by symmetrising a function over $\Gamma_U$ in this way, and while these series do not play a major part in the proof, their appearance is related to the key fact that the correct way in which to generalise Holowinsky's methods is by unfolding over the unipotent, as will be discussed in section \ref{holoextend}.

\subsection{Representation Theory of $SL(2,\C)$}

For $m \in \N$, let $\rho_m$ denote the irreducible $m+1$ dimensional representation of $SU(2) \subset SL(2,\C)$ with Hermitian inner product $\langle \: , \: \rangle$, and let $\cdot^*$ denote the associated conjugate linear isomorphism between $\rho_m$ and $\rho^*_m$.  We choose an orthonormal basis $\{ v_t \}$ ($t = m, m-2, \ldots, -m$) for $\rho_m$ and dual basis $\{ v_t^* \}$ for $\rho^*_m$, consisting of eigenvectors of $M$ satisfying

\begin{equation*}
\left( \begin{array}{cc} e^{i\theta} & 0 \\ 0 & e^{-i\theta} \end{array} \right) v_t = e^{i t\theta} v_t, \quad \left( \begin{array}{cc} e^{i\theta} & 0 \\ 0 & e^{-i\theta} \end{array} \right) v_t^* = e^{-i t\theta} v_t^*.
\end{equation*}

If $r \in \C$ and $k \in \Z$, let  $I_{(k,r)}$ be the representation of $SL(2,\C)$ unitarily induced from the character

\begin{equation*}
\chi: \left( \begin{array}{cc} z & x \\ 0 & z^{-1} \end{array} \right) \mapsto (z / |z| )^k |z|^{2ir}.
\end{equation*}

These are unitarisable for $(k,r)$ in the set

\begin{equation*}
U = \{ (k,r) | r \in \R \} \cup \{ (k,r) | k = 0, r \in i(-1,1) \},
\end{equation*}

and two such representations $I_{(k,r)}$, $I_{(k',r')}$ are equivalent iff $(k,r) = \pm (k',r')$.  Furthermore, these are all the irreducible unitary representations of $SL(2,\C)$ other then the trivial representation.  We choose a set $U' \subset U$ representing every equivalence class in $U$ to be

\begin{equation*}
U' = \{ (k,r) | r \in (0,\infty) \} \cup \{ (k,r) | r = 0, k \ge 0 \} \cup \{ (k,r) | k = 0, r \in i(0,1) \}.
\end{equation*}

Given $\pi \in \widehat{SL(2,\C)}$ nontrivial, we shall say $\pi$ has weight $k$ and spectral parameter $r$ if it is isomorphic to $I_{(k,r)}$ with $(k,r) \in U'$.  As we shall work on $GL_2$ with trivial central character, to describe the Archimedean components of our automorphic representations it will suffice to describe their restrictions to $SL_2$.  At complex places we shall use the parameters just introduced, and at real places we shall use the customary weight and spectral parameter.

\subsection{Automorphic Forms}
\label{weightprelims3}

We shall consider QUE for automorphic forms $\pi$ on $GL_2 / F$ of full level, trivial central character and cohomological type.  This means that their local factors at real places are holomorphic discrete series of even weight, and the factors at complex places have spectral parameter 0.  In the notation of section \ref{weightprelims1}, these correspond to automorphic forms on $X$ of the prescribed Archimedean type and which are eigenfunctions of the Hecke operators.  We denote the weight of $\pi$ by an $r$-tuple $k = (k_i)$, and its normalised Hecke eigenvalues by $\lambda_\pi(\p)$.  Define $\rho_k$ to be the representation

\begin{equation*}
\rho_k = \bigotimes_{i \le r_1} \chi_{k_i} \otimes \bigotimes_{i > r_1} \rho_{k_i}
\end{equation*}

of $K$, noting that in the presence of complex places $K$ will be nonabelian and $\rho_k$ will have dimension greater than one for most choices of weight.  As $\rho_k$ occurs as a $K$-type in the Archimedean component of $\pi$, there is an embedding $R_\pi$ in $\Hom_K( \rho_k, L^2(X) )$ corresponding to $\pi$.  We may associate to $R_\pi$ a section $F_k$ of the principal bundle $X \times_K \rho_k^*$ on $Y$, where we recall that for a representation $\tau$ of $K$, $X \times_K \tau$ is the quotient of $X \times \tau$ by the right $K$-action

\begin{equation*}
(x,v)k = (xk, \tau(k)^{-1} v)
\end{equation*}

so that sections of $X \times_K \tau$ may be thought of as sections of $X \times \tau$ satisfying

\begin{equation*}
\tau(k) v(xk) = v(x).
\end{equation*}

$F_k$ may be defined by the relation $R_\pi(v)(x) = (s(x), v)$ for $v \in \rho_k$ and $x \in X$, which may be unwound to give

\begin{eqnarray*}
F_k(x) & = & \prod_{i > r_1} (k_i+1)^{-1/2} \sum_t R_\pi(v_t)(x) v_t^*, \\
|F_k(x)|^2 & = & \prod_{i > r_1} (k_i+1)^{-1} \sum_t |R_\pi(v_t)(x)|^2,
\end{eqnarray*}

where $\{ v_t \}$ is a basis of $M$-eigenvectors for $\rho_k$.  Note that $| F_k(x) |^2$ descends to a function on $Y$.  Alternatively, we may define $E_k$ to be the restriction to $\Gamma$ of the representation

\begin{equation*}
\left( \bigotimes_{i \le r_1} \Sym^{k_i-2} \right) \otimes \left( \bigotimes_{i > r_1} \Sym^{k_i/2-1} \otimes \overline{\Sym}^{k_i/2-1} \right)
\end{equation*}

of $G$, and let $V_k$ the associated local system on $Y$, which we equip with a certain canonical positive definite norm.  Then $F_k$ may be thought of as a harmonic 1-form which represents a cohomology class in $H^1( Y, V_k )$ (this is why $\pi$ is referred to as being of cohomological type).  However, we will not use this point of view in this paper, and shall only refer the reader to the book of Borel and Wallach \cite{BW} where correspondences of this kind are described explicitly.  

We wish to establish the equidistribution of the probability measures $|F_k|^2 dv$ on $Y$, in generalisation of holomorhic QUE over $\Q$.  Because the $K$-integrals of $|R_\pi(v_t)|^2$ are independent of $t$, we may let $v_k \in \rho_k$ be the vector of highest weight and think of the measure $|F_k|^2 dv$ as the pushforward of $|R_\pi(v_k)|^2 dx$ from $X$.  In the case where $F$ is totally real, the reader may instead let $f$ be a holomorphic Hecke eigenform with associated representation $\pi$, and let $F_k$ be the mass function $F_k = y^{k/2} f$.  In particular, the results stated in the next section may all be read with this simpler definition in mind.

To simplify the transition from Fourier expansions to shifted convolution sums in the next chapter, we will express the Fourier expansions of all our automorphic forms by sums over the ring of integers $\OO$ rather than the inverse different $\OO^*$ as follows:

\begin{equation*}
\phi(z) = \sum_{ \xi \in \OO} a_\xi(y) e( \tr( \xi \kappa x ) ),
\end{equation*}

where $\kappa$ will denote a fixed totally positive generator of $\OO^*$ throughout.  As the $F_k$ are vector valued, it turns out that they may be expanded in Fourier series more simply by enlarging their domain $\HH_F$, in a manner which we now describe.  We identify $\HH_F$ with the subgroup $NA$ of $G$ in the standard Iwasawa factorisation, and let $\HH'_F$ be the subgroup $NAM$.  We then have an inclusion of $\HH_F$ in $\HH_F'$, and we extend our hyperbolic co-ordinate system to $\HH_F'$ by allowing $y_i$ to take complex values for $i > r_1$.  The $K$-covariance of $F_k$ means that it is determined by its values on $\HH_F'$, and these determine the embedding $R_\pi$ by the formula

\begin{equation*}
R_\pi(v)(g) = ( \rho(k)v, F_k(z) ),
\end{equation*}

where $g = zk$ is the Iwasawa factorisation of $g$.  On $\HH_F'$, we may expand $F_k$ in a Fourier series as

\begin{equation*}
F_k(z) = \sum_{\xi > 0} a_f(\xi) {\bf K}_k( \xi \kappa y) e( \tr( \xi \kappa x) ),
\end{equation*}

where ${\bf K}_k(y) = \otimes_{i=1}^r {\bf K}_i(y_i)$ and the ${\bf K}_i(y_i)$ are defined by

\begin{eqnarray}
\label{whittaker1}
{\bf K}_i(y_i) & = & ( y_i)^{k_i/2} \exp( -2\pi y_i) \quad \text{for } i \le r_1,\\
\label{whittaker2}
{\bf K}_i(y_i) & = & |y_i|^{k_i/2+1} \sum_{j=0}^{k_i} \binom{ k_i }{j}^{1/2} K_{k_i/2 - j}( 4 \pi |y_i| ) e^{ (k_i -2j)i \theta_i/2 } v_{k_i-2j}^*, \quad i > r_1,
\end{eqnarray}

and $\theta_i$ is the argument of $y_i$.  The formula for the Whittaker functions ${\bf K}_i$ at complex places is taken fron Jacquet-Langlands \cite{JL}.  The coefficients $a_f(\xi)$ are proportional to the Hecke eigenvalues $\lambda_\pi(\xi)$,

\begin{equation*}
a_f(\xi ) = \lambda_\pi(\xi) N\xi^{-1/2} a_f( 1 ),
\end{equation*}

and the first Fourier coefficient is determined by the $L^2$ normalisation of $F_k$ to be

\begin{equation}
\label{firstfourier}
| a_f( 1 ) |^2 = \prod_{i \le r_1} \frac{ (4\pi)^{k_i} }{ \Gamma(k_i) } \prod_{i > r_1} \frac{ (2\pi)^{k_i} }{ \Gamma(k_i/2 + 1)^2 } \frac{ 2^{7r_2-1} \pi^{r_1 + 3r_2} }{ |D| L(1, \sym^2 \pi ) }.
\end{equation}

(See section \ref{appnorms} for this calculation.)

\section{Statement of Results}
\label{weightresults}

Our main result is theorem \ref{weightmain}, which establishes QUE for the sections $F_k$ under the assumption that the associated cohomological representations $\pi$ satisfy the Ramanujan bound; this is known when $F$ is totally real or imaginary quadratic, as discussed below.

\begin{theorem}
\label{weightmain}
If $\phi$ is a Hecke-Maass cusp form, we have

\begin{equation*}
| \langle \phi F_k, F_k \rangle | \ll_{\phi, \epsilon, \nu} ( \log \| k \| )^{-1/30 + \epsilon}.
\end{equation*}

If $\phi$ is a pure incomplete Eisenstein series, we have

\begin{equation*}
\langle \phi F_k, F_k \rangle = \frac{1}{ Vol(Y) } \langle \phi, 1 \rangle + O_{\phi, \epsilon, \nu} ( (\log \| k \|)^{-2/15 + \epsilon}  )
\end{equation*}

\end{theorem}

Theorem \ref{weightmain} is proven by combining the following two results, which summarise the extensions of Holowinsky and Soundararajan's respective approaches to proving the equidistribution of $F_k$.  Their statements are almost identical to those of the original theorems over $\Q$, which are recalled in section \ref{Qproof}, with the only significant difference being that in the statement of theorem \ref{Home} we must impose a mild condition that all weights tend to infinity in a uniform way.

\begin{theorem}
\label{Home}
Fix an automorphic form $\phi$, and suppose that there exists a $\nu > 0$ such that $k_i > \| k \|^\nu$ for all $i$.  Define

\begin{equation*}
M_k(\pi) = \frac{1}{ ( \log \| k \| )^2 L( 1, \sym^2 \pi ) } \prod_{N\p \le \| k \|} \left( 1 + \frac{ 2 |\lambda_\pi (\p)| }{ N\p } \right).
\end{equation*}

If $\phi$ is a Hecke-Maass cusp form, then

\begin{equation}
\label{maassbound}
\langle \phi F_k, F_k \rangle \ll_{\phi, \epsilon, \nu} ( \log \| k \| )^\epsilon M_k(\pi)^{1/2}
\end{equation}

for any $\epsilon > 0$.  If $\phi$ is a pure incomplete Eisenstein series then

\begin{equation}
\label{eisbound}
\langle \phi F_k, F_k \rangle = \frac{1}{ Vol(Y) } \langle \phi, 1 \rangle + O_{\phi, \epsilon, \nu} ( (\log \| k \|)^\epsilon M_k(\pi)^{1/2} (1 + R_k(f)) )
\end{equation}

for any $\epsilon > 0$, where

\begin{equation*}
R_k(f) = \frac{1}{\sqrt{Nk} L( 1, \sym^2 \pi ) } \sum_m \int_{-\infty}^{+\infty} \frac{ |L( 1/2+it, \sym^2 \pi \otimes \lambda_{-m} )| }{ ( |t| + \| m \|+1)^A } dt.
\end{equation*}

\end{theorem}

\begin{theorem}
\label{Some}
If $\phi$ is a Hecke-Maass cusp form, we have

\begin{equation}
\label{Some1}
| \langle \phi F_k, F_k \rangle | \ll_{\phi, \epsilon} \frac{ ( \log \| k \| )^{-1/2 + \epsilon} }{ L( 1, \sym^2 \pi ) }.
\end{equation}

If $E(\tfrac{1}{2} + it, m, \cdot )$ is a unitary Eisenstein series, we have

\begin{equation}
\label{Some2}
| \langle E(\tfrac{1}{2} + it, m, \cdot ) F_k, F_k \rangle | \ll_\epsilon (1 + |t| + \| m \|)^{2n} \frac{ ( \log \| k \| )^{-1 + \epsilon} }{ L( 1, \sym^2 \pi ) }.
\end{equation}

\end{theorem}

We shall prove theorem \ref{Home} in sections \ref{sievereal} to \ref{holosieve} and theorem \ref{Some} in section \ref{someproof}, before combining them to give our main result in section \ref{mixedconclusion}.  The presence of these two components and the way in which they interact makes the overall proof somewhat elaborate, and so we begin by reviewing its basic outline in the case of $SL(2,\Z)$ and giving an overview of our modifications in section \ref{weightreview}.  Our assumption that $\pi$ satisfies the Ramanujan bound is needed in the proofs of both theorem \ref{Some} and \ref{Home}, in the first case to establish the weak form of Ramanujan required by Soundararajan's weak subconvexity theorem, and in the second as an ingredient in bounding shifted convolution sums.  It is known when $F$ is totally real or imaginary quadratic, and so we have an unconditional theorem in these cases.  In the totally real case this is derived from Deligne's theorem by Blasuis in \cite{Bl}, while in the imaginary quadratic case this relies on deep work of Harris, Soudry, Taylor, Berger, Harcos et al \cite{BH, HST} and requires the construction of a theta lift from $GL_2/F$ to $GSp_4/\Q$, where complex geometry is available.  The generalisation of their results to other fields with complex places is not yet established, and consequently we have no unconditional result outside totally real and imaginary quadratic fields.  On the other hand, Ramanujan will hold for forms lifted from totally real subfields and so our theorem becomes unconditional if the family of cohomological forms of fixed level has the structure suggested by the results of \cite{FGT} and \cite{Ma1}, i.e. if base change and CM constructions account for all but finitely many forms.

The assumption we have made on the uniform growth of the weight is a purely technical one, and by combining the triple product identities in section \ref{weighttripprod} with the Lindel\"of hypothesis we see that the result should still be true without it.  The reason we have adopted it is so that when we come to the point in the generalisation of Holowinsky's theorem at which we apply the sieve, it will ensure that we are sieving over a rounded subset of the ring of integers rather than a narrow box.

Theorem \ref{weightmain} establishes QUE for any sequence of sections $F_k$ over a totally real or imaginary quadratic field whose weights tend to infinity with the required uniformity.  However, we should ask whether such a sequence exists for these fields.  When $F$ is totally real, Riemann-Roch ensures that the dimension of the space $S_k$ of holomorphic cusp forms of weight $k$ is $\sim Nk$, with an exact formula established by Shimizu in \cite{Sh}.  Over a general field, base change from $\Q$ is expected to provide $\sim k$ forms of parallel weight $k$ on a sufficiently deep congruence subgroup of $\Gamma$, where the term `parallel' means that the weights at all places are equal as in the totally real case.  In particular, for $F$ imaginary quadratic it has been proven by Finis, Grunewald and Tirao \cite{FGT} that base change produces forms of full level and so our result is not vacuous for the Bianchi manifolds.  The proof may be easily modified to allow nontrivial level in any case, so by restricting to forms base changed from $\Q$ (or another totally real subfield) which are known to satisfy Ramanujan, we may view it as having content over any solvable field $F$.

\subsection{Equidistribution of Zero Currents}
\label{currentQUE1}

One consequence of QUE for holomorphic modular forms over $\Q$ is that the zeros of a sequence of forms become equidistributed with respect to hyperbolic measure as $k \rightarrow \infty$, as was proven by Shiffman and Zelditch \cite{SZ} for compact hyperbolic surfaces and extended to $SL(2,\Z) \backslash \HH^2$ by Rudnick \cite{Ru}.  Using their methods, we have derived the analogous statement about the equidistribution of the zero divisors of holomorphic modular forms from our proof of holomorphic QUE.  We may prove this equidistribution either in the sense of measures of integration over the (smooth parts of the) zero divisors, or in the more refined sense of Lelong (1,1)-currents, which we now describe.

We now use $\HH^n$ to denote the product of $n$ copies of the upper half plane, so that the holomorphic forms $f$ we consider live on $Y = \Gamma \backslash \HH^n$.  In higher dimensions we may replace the sum of delta measures at the zeros of $f$ by the current of integration over its zero divisor $Z_f$, which is a distribution on differential forms of bidegree $(n-1, n-1)$.  If $Z_f = \sum_i \ord_{V_i} (f) V_i$ is the expression of $Z_f$ as the sum of irreducible subvarieties, then

\begin{equation}
\label{zerocurrent}
( Z_f, \phi) = \sum_i \ord_{V_i}(f) \int_{V_i} \phi
\end{equation}

for all smooth, compactly supported forms $\phi$ on $\Gamma \backslash \HH^n$.  To define these notions in the presence of torsion in $\Gamma$, we use the standard procedure of choosing $\Gamma' \subset \Gamma$ finite index and torsion free, and defining forms, subvarieties etc. on $\Gamma \backslash \HH^n$ to be those on $\Gamma' \backslash \HH^n$ which are invariant under $\Gamma$.  Integrals such as (\ref{zerocurrent}) are defined to be the lifted integral on $\Gamma' \backslash \HH^n$ divided by $| \Gamma' : \Gamma |$.  We shall use $\overset{w^*}{\longrightarrow}$ to denote weak$^*$ convergence of currents.  With these notions in mind, we may state our result.

\begin{theorem}
\label{currentQUE}

Fix a weight $k = (k_i)$, $k_i > 0$, and let $\{ f_N \}$ be a sequence of holomorphic Hecke modular forms of weight $Nk$.  Define

\begin{eqnarray*}
\omega & = & \frac{-i}{2\pi} \partial \overline{\partial} \log y^k \\
 & = & \frac{1}{4 \pi} \sum k_i y_i^{-2} dx_i \wedge dy_i.
\end{eqnarray*}

If $Z_N$ are the zero divisors of $f_N$, then $\tfrac{1}{N} Z_N \overset{w^*}{\longrightarrow} \omega$, i.e.

\begin{equation*}
\underset{N \rightarrow \infty}{\lim} \left( \tfrac{1}{N} Z_N, \phi \right) = \int_Y \omega \wedge \phi
\end{equation*}

for all continuous, compactly supported $(n-1,n-1)$-forms $\phi$.  In particular, if $k = (2, \ldots, 2)$ then $\tfrac{1}{N} Z_N \overset{w^*}{\longrightarrow} \omega_0$, the K\"ahler form of $Y$ with the product hyperbolic metric.

\end{theorem}

This theorem is based on ideas from complex potential theory as developed for problems in quantum chaos in \cite{NV, Ru, SZ}.  It may be loosely interpreted as saying that not only do the (smooth parts of the) submanifolds $Z_N$ become equidistributed as measures of integration with respect to the induced Riemannian volume, but that the directions in which their tangent subspaces lie are also becoming equidistributed.  We prove theorem \ref{currentQUE} in section \ref{currents}.

\section{Outline of the Proof}
\label{weightreview}

\subsection{The Proof Over $\Q$}
\label{Qproof}

We begin by giving an outline of Holowinsky and Soundararajan's proof over $\Q$, as our proof over a number field runs on the same lines as theirs.  Suppose $f$ is a holomorphic Hecke eigenform of weight $k$ on $Y = SL(2,\Z) \backslash \HH^2$, with associated mass function $F_k = y^{k/2} f$.  We wish to show that the normalised probability measure $\mu_f = |F_k|^2 y^{-2} dx dy$ tends weakly to hyperbolic measure $\tfrac{3}{\pi} y^{-2} dx dy$ as $k$ tends to infinity, i.e. that for all $h \in C^\infty_0(X)$

\begin{equation*}
\mu_f(h) = \int_Y  h |F_k|^2 y^{-2} dx dy \rightarrow \frac{3}{\pi} \langle h, 1 \rangle.
\end{equation*}

In \cite{Ho, HS, So1}, Holowinsky and Soundararajan have established this by decomposing $h$ in two different bases for smooth functions on $X$, the first a complete set of eigenfunctions for the Laplacian and the second the incomplete Poincare series $P_m$, defined by

\begin{equation*}
P_m( \psi | z) = \sum_{\gamma \in \Gamma_\infty \backslash \Gamma } e( m x(\gamma z) ) \psi( y ( \gamma z) )
\end{equation*}

for $m \in \Z$ and $\psi \in C^\infty_0(\R^+)$.  The chosen basis of Laplace eigenfunctions consists of the constant function, Hecke-Maass cusp forms $\phi$ and unitary Eisenstein series $E(\tfrac{1}{2} + it, \cdot )$, and the corresponding integrals which must be estimated are $\langle \phi F_k, F_k \rangle$ and $\langle E(\tfrac{1}{2} + it, \cdot ) F_k, F_k \rangle$.  These integrals may be expressed in terms of central $L$-values, using the classical Rankin-Selberg formula in the first case and Watson's formula in the second, and so one may hope that the theory of $L$ functions would provide nontrivial upper bounds for them.  The convex bound just fails to be of use here, however by strengthening the convex bound by a factor of $(\log C)^{-1 + \epsilon}$ where $C$ is the analytic conductor Soundararajan obtains the following result:

\begin{theorem}
\label{So}
If $\phi$ is a Hecke-Maass cusp form, we have

\begin{equation*}
| \langle \phi F_k, F_k \rangle | \ll_{\phi, \epsilon} \frac{ ( \log k )^{-1/2 + \epsilon} }{ L( 1, \sym^2 f ) }.
\end{equation*}

If $E(\tfrac{1}{2} + it, \cdot )$ is a unitary Eisenstein series, we have

\begin{equation*}
| \langle E(\tfrac{1}{2} + it, \cdot ) F_k, F_k \rangle | \ll_\epsilon (1 + |t|)^2 \frac{ ( \log k )^{-1 + \epsilon} }{ L( 1, \sym^2 f ) }.
\end{equation*}

\end{theorem}

The equidistribution of $\mu_f$ would follow from theorem \ref{So} if one knew that $L(1, \sym^2 f) \gg (\log k)^{-1/2 + \delta}$ for some $\delta > 0$.  This is certainly expected, as it follows from the generalised Riemann hypothesis that $L(1, \sym^2 f)$ is bounded below by a power of $\ln \ln k$.  The best unconditional bound in this direction is due to Hoffstein and Lockhart \cite{HL}, and Goldfeld, Hoffstein and Lockhart \cite{GHL}, who prove that $L(1, \sym^2 f) \gg (\log k)^{-1}$; this is a deep result analogous to proving that there is no Siegel zero.  The bound $L(1, \sym^2 f) \gg (\log k)^{-1/2 + \delta}$ is known unconditionally for all but $K^\epsilon$ eigenforms of weight $\le K$ by a zero density argument, however one cannot rule out those forms with small values of $L( 1, \sym^2 f )$ for which Soundararajan's approach is insufficient.

Holowinsky's approach is to test $\mu_f$ against incomplete Poincare and Eisenstein series.  This is equivalent to testing $\mu_f$ against Hecke-Maass cusp forms and incomplete Eisenstein series, and evaluating the inner products $\langle \phi F_k, F_k \rangle$ by regularising them with a second incomplete Eisenstein series and then unfolding.  In doing this one is led to estimating the shifted convolution sums

\begin{equation*}
\sum_{n \sim k} \lambda_f(n) \lambda_f(n+l)
\end{equation*}

for fixed $l$ as $k \rightarrow \infty$, where $\lambda_f$ are the automorphically normalised Hecke eigenvalues of $f$, and quite strikingly one is able to obtain useful bounds for these by taking absolute values of the terms and forgoing any additive cancellation.  The idea behind this is that the eigenvalues $\lambda_f(p)$ not only satisfy the Ramanujan bound $|\lambda_f(p)| \le 2$, but are distributed in the interval $[-2,2]$ according to Sato-Tate measure and so on average $|\lambda_f(p)|$ will be significantly smaller than 2 (we do not need to consider dihedral forms as we are working at full level).  Moreover, as a typical $\lambda_f(n)$ is a product of many $\lambda_f(p)$'s this leads to a gain on average over the bound $|\lambda_f(n)| \le \tau(n)$.  This phenomenon may also be seen in the work of Elliot, Moreno and Shahidi \cite{EMS} where they prove the bound

\begin{equation*}
\sum_{n \le x} | \tau(n) | \ll x^{13/2} (\log x)^{-1/18},
\end{equation*}

where $\tau$ here denotes Ramanujan's $\tau$-function.  Holowinsky uses this idea, combined with a large sieve to show that $n$ and $n+l$ seldom both have small prime factors, to prove the following:

\begin{theorem}
\label{Ho}
If $\lambda_f$ are the normalised Hecke eigenvalues as above, define

\begin{equation*}
M_k(f) = \frac{1}{ (\log k)^2 L( 1, \sym^2 f ) } \prod_{p \le k} \left( 1 + \frac{ 2 | \lambda_f(p) | }{p} \right).
\end{equation*}

If $\phi$ is a Hecke-Maass cusp form, we have

\begin{equation*}
| \langle \phi F_k, F_k \rangle | \ll_{\phi, \epsilon} ( \log k )^\epsilon M_k(f)^{1/2}.
\end{equation*}

If $E( \psi | \, \cdot \, )$ is an incomplete Eisenstein series, we have

\begin{equation*}
| \langle E( \psi | \, \cdot \, ) F_k, F_k \rangle - \tfrac{3}{\pi} \langle E( \psi | \, \cdot \, ), 1 \rangle | \ll_{\psi,\epsilon}  ( \log k )^\epsilon M_k(f)^{1/2} ( 1 + R_k(f) ),
\end{equation*}

where

\begin{equation*}
R_k(f) = \frac{1}{ k^{1/2} L( 1, \sym^2 f ) } \int_{-\infty}^\infty \frac{ | L( 1/2 + it, \sym^2 f ) | }{ ( 1 + |t| )^{10} } dt.
\end{equation*}

\end{theorem}

One can see the appeal to Sato-Tate in the quantity $M_k(f)$ appearing in theorem \ref{Ho}; if we only apply the bound $|\lambda_f(p)| \le 2$ to this, one finds that $M_k(f) \ll (\ln k)^2 L(1, \sym^2 f)^{-1}$ which is of no use.  However, under certain natural assumptions about the distribution of $\lambda_f(p)$ it may be shown that $M_k(f)$ is small - more precisely, in \cite{Ho2} Holowinsky shows that if neither $L(1, \sym^2 f)$ or $L(1, \sym^4 f)$ are small then we have $M_k(f) \ll (\ln k)^{-\delta}$ for some $\delta > 0$.  As with Soundararajan's theorem, these assumptions may also be shown to hold for almost all eigenforms using zero density estimates.

Surprisingly, while both of these approaches may fail it can be shown that together they cover all cases completely.  Intuitively speaking, if $L( 1, \sym^2 f ) < (\log k)^{-1/2 + \delta}$ is small then we should have $\lambda_f(p^2) \sim -1$ for most primes $p \le k$ (a Siegel zero type phenomenon).  However, $M_k(f)$ is proven in \cite{HS} to satisfy the upper bound

\begin{equation}
\label{Mk}
M_k(f) \ll (\log k)^\epsilon \exp \left( - \sum_{p \le k} \frac{ ( |\lambda_f(p)| - 1)^2 }{ p } \right),
\end{equation}

and if $\lambda_f(p^2) \sim -1$ then $\lambda_f(p)^2 - 1 \sim -1$, so that $\lambda_f(p) \sim 0$ for most $p \le k$ and the right hand side of (\ref{Mk}) should be small.  The precise bound Holowinsky and Soundararajan prove based on this argument is

\begin{equation*}
M_k(f) \ll (\log k)^{1/6} ( \log \log k )^{9/2} L( 1, \sym^2 f )^{1/2}.
\end{equation*}

This inequality may be used to show that if $\phi$ is a cusp form and $L( 1, \sym^2 f ) < ( \log k)^{-1/3 - \delta}$ for some $\delta > 0$ then $M_k(f)$, and hence $\langle \phi F_k, F_k \rangle$, is small.  However, if $L( 1, \sym^2 f ) > ( \log k)^{-1/3 - \delta} > ( \log k)^{-1/2 + \delta}$ then theorem \ref{So} shows that $\langle \phi F_k, F_k \rangle$ is small.  This shows how theorems \ref{So} and \ref{Ho} complement each other in the cusp form case, and a similar relationship holds between them in the incomplete Eisenstein case.

\subsection{Extension to a Number Field}
\label{holoextend}

We now describe the the steps that must be made to generalise the method of section \ref{Qproof} to a number field.  Soundararajan's approach is the easier of the two to extend, as one has the triple product formula of Ichino \cite{I} available to generalise Watson's formula, and Soundararajan's weak subconvexity theorem is sufficiently general to also be applicable to the central $L$ value which appears there.  The only technical difficulty is in making Ichino's formula sufficiently quantitative, which requires estimating certain Archimedean integrals.  The necessary computation at complex places was carried out in \cite{Ma2} using a result of Michel and Venkatesh appearing in \cite{MV}, while at real places it may be obtained by comparison with Watson's formula.  Applying weak subconvexity is then straightforward, with the only consideration being that Soundararajan's theorem is stated for $L$ functions over $\Q$ rather than a number field.  However, it is easy to show that our $L$ functions still satisfy the required hypotheses when viewed as Euler products over $\Q$.  These steps are carried out in section \ref{someproof}.

The modifications that must be made in the case of Holowinsky's method are more involved, and we shall now describe his method in more detail before illustrating how we have adapted it in the simple case of a real quadratic field.  Holowinsky's approach for $SL(2,\Z)$ is similar to calculating the integral of $|F_k|^2$ against a Poincare series in terms of shifted convolution sums.  For a Hecke-Maass form or incomplete Eisenstein series $\phi$, he defines a regularised unfolding of $\langle \phi F_k, F_k \rangle$ in terms of a fixed positive $g \in C^\infty_0( \R^+)$ and a slowly growing parameter $T$ by

\begin{equation}
\label{HIphi}
I_\phi(T) =  \int_{\Gamma_\infty \backslash \HH^2} g(Ty) \phi(z) |F_k(z)|^2 d\mu.
\end{equation}

This behaves like the integral of $\phi |F_k|^2$ over $T$ copies of a fundamental domain for $SL(2,\Z)$, which may be seen by taking the Mellin transform $G$ of $g$ and expressing (\ref{HIphi}) in terms of Eisenstein series as

\begin{equation*}
I_\phi(T) =  \frac{1}{2\pi i} \int_{(\sigma)} G(-s) T^s \int_{Y} E(s,z) \phi(z) |F_k(z)|^2 d\mu.
\end{equation*}

Shifting the contour to $\sigma = 1/2$ then gives

\begin{eqnarray*}
I_\phi(T) & = &  cT \langle \phi F_k, F_k \rangle + O(T^{1/2}), \\
\text{where} \quad c & = & \frac{3}{\pi} \langle E(g|z), 1 \rangle.
\end{eqnarray*}

Holowinsky then calculates $I_\phi(T)$ in a second way using the Fourier expansions of $\phi$ and $f$,

\begin{eqnarray*}
\phi(z) & = & \sum_l a_l(y) \exp( 2\pi i l x), \\
f(z) & = & \sum_{n \ge 1} a_f(n) \exp( 2 \pi i n z).
\end{eqnarray*}

Only those $l$ with $|l| \ll T^{1+\epsilon}$ make a significant contribution, and for those $l$ Holowinsky considers

\begin{eqnarray}
\notag
S_l(T) & = & \int_{\Gamma_\infty \backslash \HH^2} g(Ty) a_l(y) \exp( 2\pi l x) |F_k(z)|^2 d\mu \\
\label{HSl}
& \ll & |a_l(T^{-1}) | \sum_{n \ge 1} |a_f(n) a_f(n+l) | \left( \int_0^\infty g(Ty) y^{k-2} e^{ -2\pi (2n+l) y } dy \right)
\end{eqnarray}

so that

\begin{equation*}
I_\phi(T) = \sum_{|l| \ll T^{1+\epsilon} } S_l(T) + O(T^{1/2}).
\end{equation*}

When $l \neq 0$, the regularising factor $g(Ty)$ effectively truncates the sum in (\ref{HSl}) to $n \ll Tk$, and we end up with an upper bound for $S_l(T)$ of

\begin{equation*}
S_l(T) \ll \frac{ |a_l(T^{-1}) }{ k L(1, \sym^2 f) } \sum_{n \le Tk} | \lambda_f(n) \lambda_f(n+l) |.
\end{equation*}

The expected main term $\tfrac{3}{\pi} \langle \phi, 1 \rangle$ appears in $S_0(T)$, and so to prove that $\tfrac{3}{\pi} \langle \phi, 1 \rangle$ and $\langle \phi F_k, F_k \rangle$ are close one needs to bound the off diagonal terms $S_l(T)$ and hence $\sum_{n \le x} |\lambda_f(n) \lambda_f(n+l)|$.  Having given up additive cancellation in this sum, Holowinsky instead proceeds by using the ideas discussed in section \ref{Qproof} to show that $|\lambda_f(n) \lambda_f(n+l)|$ is small on average.

We have extended this method to work over an arbitrary number field $F$, with the key innovation being the way the unfolding is carried out in the presence of units.  For simplicity, we will briefly describe the method in the case of a real quadratic field $F = \Q( \sqrt{d} )$, and $f$ a holomorphic Hecke modular form of parallel weight $(k,k)$ with associated automorphic representation $\pi$. Let $\phi$ be a Hecke-Maass cusp form, and write the Fourier expansions of $f$ and $\phi$ as

\begin{eqnarray*}
f(z) & = & \sum_{\eta > 0} a_f(\eta) \exp( 2\pi i \tr( \eta \kappa z) ), \\
\phi(z) & = & \sum_{\xi \neq 0} a_\xi(y) \exp( 2\pi i \tr( \eta \kappa x) ).
\end{eqnarray*}

The totally positive units $\OO_+^\times$ of $\OO$ act on the terms of these expansions, and when unfolding we must do so in a way which breaks this symmetry so that the resulting shifted convolution sums are over well rounded sets in $\OO$.  The correct approach is to unfold to $\Gamma_U \backslash \HH^2 \times \HH^2 \simeq \R_+^2 \times ( \R^2 / \OO)$ and localise in a set of the form $B_T \times ( \R^2 / \OO)$, where $B$ is a ball in $\R_+^2$ and we multiply it by $T^{-1}$ in each co-ordinate to get $B_T$.  This lets us largely ignore the units, and when we form the analogues of $S_l(T)$ it will allow us to truncate the resulting shifted convolution sum over $\OO$ at each place seperately.  We therefore define $I_\phi(T)$ as the integral

\begin{equation}
\label{Iphisimple}
I_\phi(T) = \int_{ \Gamma_U \backslash \HH^2 \times \HH^2} g(Ty) \phi(z) |F_k(z)|^2 dv,
\end{equation}

where now we let $h \in C^\infty_0(\R^+)$ be a positive function and $g \in C^\infty_0(\R_+^2)$ be its square.  We extract a main term $c T^2 \langle \phi F_k, F_k \rangle$ from this as before, by forming the symmetrised function

\begin{equation*}
\widetilde{g}(y) = \sum_{u \in \OO_+^\times} g(uy)
\end{equation*}

and expanding it in the multiplicative characters of $\R_+^2 / \OO_+^\times$ to express $I_\phi(T)$ in terms of integrals against Eisenstein series.  When we calculate $I_\phi(T)$ in terms of the Fourier expansion of $\phi$ it may again be shown that only those $\xi$ with $\| \xi \| \ll T^{1+\epsilon}$ contribute, and for these we define

\begin{equation*}
S_\xi(T) = \int_{ \Gamma_U \backslash \HH^2 \times \HH^2} g(Ty) a_\xi(y) \exp( 2 \pi i \tr( \xi \kappa x) ) |F_k(z)|^2 dv.
\end{equation*}

The analogue of the upper bound on $S_\xi(T)$ for $\xi \neq 0$ in terms of shifted convolutions sums is

\begin{equation*}
S_\xi(T) \ll |a_\xi(T^{-1})| \sum_{\eta > 0} |a_f(\eta) a_f(\eta + \xi)| \int_{\R_+^2} g(Ty) y^{k-2} \exp( 2 \pi \tr( ( 2 \eta + \xi) \kappa y) ) dy.
\end{equation*}

The key feature of the integral appearing here is that it factorises over the places of $\Q( \sqrt{d} )$, and each factor depends only on the image of $2\eta + \xi$ at that place, which lets us truncate the sum to the ball of radius $k$ in $\OO$ and leaves us with bounding $\sum_{0 < \eta < k} | \lambda_f(\eta) \lambda_f(\eta + \xi)|$.  This round set is well suited to the application of the large sieve for lattices in $\R^n$, and we may carry out Holowinsky's sieving approach as before by translating congruences modulo primes $\p$ of $\OO$ to sieve conditions in $\OO / p\OO$ without significant interference from the units.

We carry this method out in detail in sections \ref{sievereal} to \ref{holosieve}.  The proof splits into two parts, the first of which is reducing bounds on $\langle \phi F_k, F_k \rangle$ to ones on shifted convolution sums, and the second of which is bounding these sums using the large sieve.  The bulk of the work lies in the first step, and we have divided it into the case of totally real fields, carried out in section \ref{sievereal}, and the modifications which are needed in the presence of complex places which are described in section \ref{sievemixed}.  The application of the large sieve is carried out in section \ref{holosieve}.

\section{Sieving for Mass Equidistribution: The Totally Real Case}
\label{sievereal}

In this section we prove proposition \ref{Holo1} below, which reduces the problem of bounding $\langle \phi F_k, F_k \rangle$ to one of bounding shifted convolution sums.  We shall assume for simplicity in this section that $F$ is totally real, so that the key modifications in the unfolding argument can be seen more clearly, and leave the treatment of complex places for section \ref{sievemixed}.  We will work with holomorphic forms rather than vector valued ones, and so let $f$ be a $L^2$ normalised holomorphic Hecke eigenform of weight $k$ with associated automorphic representation $\pi$.  We assume there exists $\nu > 0$ such that $k_i \ge \| k \|^\nu$ for all $i$.

\begin{prop}
\label{Holo1}
Let $T \ge 1$ and $\epsilon > 0$.  Fix $h \in C^\infty_0(\R^+)$ positive and let $g \in C^\infty_0 ( \F^+)$ be its n-fold product, and define $C_g = \langle E(g | z), 1 \rangle / Vol(Y)$.  Fix an automorphic form $\phi$ with Fourier expansion

\begin{equation*}
\phi(z) = \sum_{ \xi \in \OO} a_\xi(y) e( \tr( \xi \kappa x ) ).
\end{equation*}

If $\phi$ is a Hecke-Maass cusp form, then

\begin{equation}
\label{Maassint}
\langle \phi F_k, F_k \rangle = C_g^{-1}T^{-n} \sum_{0 < \| \xi \| < T^{1+\epsilon} } S_\xi(T) + O(T^{-n/2}).
\end{equation}

If $\phi$ is a pure incomplete Eisenstein series, then

\begin{equation}
\label{eisint}
\langle \phi F_k, F_k \rangle = \frac{1}{ Vol(Y) } \langle \phi, 1 \rangle + C_g^{-1}T^{-n} \sum_{0 < \| \xi \| < T^{1+\epsilon} } S_\xi(T) + O\left( \frac{1 + R_k(f)}{T^{n/2}} \right)
\end{equation}

with

\begin{equation}
\label{Rbound}
R_k(f) = \frac{1}{\sqrt{Nk} L( 1, \sym^2 \pi ) } \sum_m \int_{-\infty}^{+\infty} \frac{ |L( 1/2+it, \sym^2 \pi \otimes \lambda_{-m} )| }{ (|t| + \| m \| + 1)^A } dt.
\end{equation}

Furthermore, we have the bound

\begin{multline}
\label{unfold}
S_\xi(T) \ll \frac{ | a_\xi(T^{-1}) |  }{ Nk L( 1, \sym^2 \pi ) } \biggl( \sum_{ \eta > 0} | \lambda_f(\eta) \lambda_f( \eta + \xi ) | \prod_{i=1}^n h \left( \frac{ T(k_i-1) }{ 4\pi (\eta_i + \xi_i/2) } \right) \\
+ O( Nk \| k \|^{-\nu +\epsilon} T^{n+\epsilon} ) \biggr).
\end{multline}

\end{prop}

The bound we shall apply to the shifted convolution sums appearing in proposition \ref{Holo1} is given below; it will be proven in section \ref{holosieve} following Holowinsky, although it should be noted that this result may also be derived from the works of Nair \cite{Na} and Nair-Tenenbaum \cite{NT}.

\begin{prop}
\label{shiftbound}

Let $\lambda_1$ and $\lambda_2$ be multiplicative functions on $\OO^+$ satisfying $| \lambda_i(\eta) | \le \tau_m(\eta)$ for some $m$.  For any $x = (x_i)$ sufficiently large with respect to $\epsilon$ and satisfying $x_i \ge \| x \|^{\nu}$, and any fixed $\xi$ satisfying $0 < \| \xi \| \le \| x \|^\nu$ we have

\begin{equation}
\sum_{0 < \eta < x} | \lambda_1(\eta) \lambda_2( \eta + \xi ) | \ll \frac{ \tau(\xi) Nx}{ (\log |x|)^{2-\epsilon}} \prod_{N\p \le z} \left( 1 + \frac{| \lambda_1(\p)| + | \lambda_2(\p)|}{N\p} \right).
\end{equation}

\end{prop}

To deduce theorem \ref{Home} in the totally real case from propositions \ref{Holo1} and \ref{shiftbound}, first apply proposition \ref{shiftbound} with $\lambda_1 = \lambda_2 = \lambda_\pi$ and $x = Tk$ to obtain

\begin{equation}
\sum_{ \eta > 0} | \lambda_f(\eta) \lambda_f( \eta + \xi ) | \prod_{i=1}^n h \left( \frac{ T(k_i-1) }{ 4\pi (\eta_i + \xi_i/2) } \right) \ll \frac{ \tau(\xi) T^n Nk}{ (\log \| k \|)^{2-\epsilon}} \prod_{N\p \le \| k \| } \left( 1 + \frac{2 | \lambda_\pi (\p) | }{N\p} \right).
\end{equation}

In the case of $\phi$ a Maass form, we substitute this into (\ref{unfold}) and bound $|a_\xi(T^{-1})|$ by $|\rho(\xi)| T^{-n/2+\epsilon}$ using lemma \ref{coeffbound} from section \ref{holofourier} below.  As we shall choose $T$ so that it is bounded above by any positive power of $\| k \|$, (\ref{unfold}) then becomes

\begin{equation*}
S_\xi(T) \ll \frac{ |\rho(\xi)| \tau(\xi) T^{n/2 + \epsilon} }{ L(1, \sym^2 \pi) ( \log \| k \| )^{2-\epsilon} } \prod_{N\p \le \| k \| } \left( 1 + \frac{2 | \lambda_\pi (\p) | }{N\p} \right) + O( \| k \|^{-\nu + \epsilon} ).
\end{equation*}

Substituting this into (\ref{Maassint}) and applying the Ramanujan bound on average (\ref{avramanujan}), we obtain 

\begin{equation*}
\langle \phi F_k, F_k \rangle \ll \frac{ T^{n/2} (T \log \| k \| )^\epsilon }{ (\log \| k \|)^2 L( 1, \sym^2 \pi ) } \prod_{N\p \le \| k \|} \left( 1 + \frac{2 | \lambda_\pi (\p) | }{N\p} \right) + O(T^{-n/2})
\end{equation*}

which gives (\ref{maassbound}) on choosing $T^n = M_k(\pi)^{-1}$.  The derivation of (\ref{eisbound}) in the pure incomplete Eisenstein series case is similar.

The organisation of this section is as follows.  In section \ref{holofourier} we prove some results we shall need on the Fourier coefficients of $\phi$ and $f$, and in section \ref{holounfold} we introduce the regularised unfolding integral which is the heart of our proof before using it to relate $\langle \phi F_k, F_k \rangle$ to shifted convolution sums in section \ref{holoshift}.

\subsection{Fourier Coefficient Calculations}
\label{holofourier}

In this section we present some bounds and normalisations we shall need for the Fourier coefficients of $\phi$ and $f$.  If $\phi$ is an automorphic form on $\Gamma \backslash (\HH^2)^n$, we may expand it in a Fourier series as

\begin{equation*}
\phi(z) = a_0(y) + \sum_{ \xi } a_\xi(y) e( \tr( \xi \kappa x ) )
\end{equation*}

with $a_0(y) = 0$ if $\phi$ is a cusp form.  If $\phi$ is a fixed Maass cusp form with spectral parameter $r = (r_i)$ then we have the expansion

\begin{equation*}
\phi(z) = \sqrt{Ny} \sum_{ \xi \neq 0 } \rho(\xi) \prod_{p = 1}^n K_{ir_p}(2\pi |\xi_p| \kappa_p y_p )  e( \tr( \xi \kappa x ) ),
\end{equation*}

where the $\rho(\xi)$ satisfy the Ramanujan bound on average, i.e.

\begin{equation}
\label{avramanujan}
\sum_{ \| \xi \| \le T } | \rho(\xi) | \ll T^n.
\end{equation}

If $\phi$ is a pure incomplete Eisenstein series $E( \psi, m | z)$, we may determine its Fourier coefficients in terms of the coefficients of the complete Eisenstein series $E(s, m, z)$.  The Fourier expansion of these series was calculated by Efrat \cite{Ef} to be

\begin{multline*}
E(s, m, z) = Ny^s \lambda_m(y) + \phi(s,m) Ny^{1-s} \lambda_{-m}(y) + \frac{2^n \pi^{ns} }{ \sqrt{D} } Ny^{1/2} \times \\
\sum_{ \xi \neq 0 } N(\xi \kappa)^{s-1/2} \lambda_m(\xi \kappa) \prod_{p = 1}^n \frac{ K_{ s + \beta(m,p) - 1/2 } ( 2\pi |\xi_p| \kappa_p y_p ) }{ \Gamma(s + \beta(m,p)) } \frac{ \sigma_{1-2s, -2m}(\xi \kappa) } {\zeta(2s,\lambda_{-2m} )} e( \tr(\xi \kappa x)),
\end{multline*}

where $\beta(m,p)$ is as in (\ref{beta}) and

\begin{eqnarray*}
\phi(s) & = & \frac{ \pi^{n/2} }{ \sqrt{D} } \prod_{p = 1}^n \frac{ \Gamma( s + \beta(m,p) -1/2) }{ \Gamma( s + \beta(m,p) ) } \frac{ \zeta( 2s-1, \lambda_{-2m} ) }{ \zeta(2s,\lambda_{-2m} ) } = \frac{ \theta(s-1/2) }{ \theta(s) }, \\
\theta(s) & = & \pi^{-ns} D^s \prod_{p = 1}^n \Gamma( s + \beta(m,p) ) \zeta(2s,\lambda_{-2m} ), \\
\sigma_{1-2s,-2m}(\xi \kappa) & = & \sum_{\substack{ ( c ) \\ \xi/c \in \OO } } \frac{ \lambda_{-2m}(c) }{ |Nc|^{2s-1} }.
\end{eqnarray*}

We have

\begin{equation*}
E(\psi, m | z) = \frac{1}{2\pi i} \int_{(2)} \Psi(-s) E(s,m,z) ds,
\end{equation*}

where $\Psi$ is the Mellin transform of $\psi$.  From this formula we may calculate the Fourier coefficients of $E(\psi, m | z)$, obtaining the expression

\begin{eqnarray*}
a_0(y) & = & \frac{1}{2\pi i} \int_{(2)} \Psi(-s)( Ny^s \lambda_m(y) + \phi(s,m) Ny^{1-s} \lambda_{-m}(y) ) ds \\
& = & \psi( Ny ) \lambda_m(y) + O( Ny^{-1} ).
\end{eqnarray*}

Moving the line of integration to $\sigma = 1/2$ we obtain

\begin{equation*}
a_0(y) = \frac{1}{ \text{Vol}(Y) } \langle E(\psi, m | z), 1 \rangle + O( Ny^{1/2} ),
\end{equation*}

where the main term is only nonzero for $m=0$.  Doing the same for nonzero $\xi$ we obtain

\begin{multline*}
a_\xi(y) = \frac{ 2^n \pi^{n/2} Ny^{1/2} }{2\pi i \sqrt{D} } \int_{-\infty}^\infty \pi^{nit} \Psi( -1/2 - it) 
N(\xi \kappa)^{it} \lambda_m(\xi \kappa) \\
\prod_{p=1}^n \frac{ K_{ it + \beta(m,p) } ( 2\pi |\xi_p| \kappa_p y_p ) }{ \Gamma(1/2 + it + \beta(m,p)) } \frac{ \sigma_{-2it, -2m}(\xi \kappa)}{\zeta(2s, \lambda_{-2m} )} dt.
\end{multline*}

We may apply the bound

\begin{equation*}
K_{ir}(y) \ll | \Gamma( 1/2 + ir) | \left( \frac{ 1 + |r|}{ y } \right)^A \left( 1 + \frac{1+|r|}{y} \right)^\epsilon,
\end{equation*}

valid for any integer $A \ge 0$ and $\epsilon > 0$, to this to obtain

\begin{equation*}
a_\xi(y) \ll \tau(\xi) Ny^{1/2} \| \xi y \|^{-A} \prod_{i=1}^n \left( 1 + \frac{1}{\xi_i y_i} \right)^\epsilon
\end{equation*}

Similar bounds are valid for $\phi$ a Maass cusp form.  The bounds for both varieties of form are summarised in the following lemma:

\begin{lemma}
\label{coeffbound}
Let $\phi$ be an automorphic form on $Y$ with Fourier series expansion 

\begin{equation*}
\phi(z) = a_0(y) + \sum_{ \xi \neq 0 } a_\xi(y) e( \tr( \xi \kappa x ) ).
\end{equation*}

If $\phi$ is a Maass cusp form, then $a_0(y) = 0$ and for $\xi \neq 0$ we have

\begin{equation*}
a_\xi(y) \ll |\rho(\xi)| Ny^{1/2} \| \xi y \|^{-A} \prod_{i=1}^n \left( 1 + \frac{1}{ \xi_i y_i } \right)^\epsilon
\end{equation*}

for any integer $A \ge 0$ and any $\epsilon > 0$.  If $\phi$ is a pure incomplete Eisenstein series, then

\begin{equation*}
a_0(y) = \frac{1}{ \text{Vol}(Y) } \langle \phi, 1 \rangle + O( Ny^{1/2} )
\end{equation*}

and for $\xi \neq 0$ we have

\begin{equation*}
a_\xi(y) \ll \tau(\xi) Ny^{1/2} \| \xi y \|^{-A} \prod_{i=1}^n \left( 1 + \frac{1}{ \xi_i y_i } \right)^\epsilon
\end{equation*}

for any integer $A \ge 0$ and any $\epsilon > 0$.

\end{lemma}

As we shall work with the fourier expansion of $f$ rather than $F_k$, the $L^2$ normalisation of $a_f(1)$ differs slightly from the one given in secton \ref{weightprelims3}.  $f$ has the expansion

\begin{equation*}
f(z) = \sum_{\eta > 0} a_f(\eta) e( \tr( \eta \kappa z) )
\end{equation*}

with $a_f(\eta)$ satisfying

\begin{equation*}
a_f(\eta ) = \lambda_\pi(\eta) a_f( 1 ) \eta^{(k-1)/2},
\end{equation*}

and the bound $|\lambda_\pi(\xi)| \le \tau(\xi)$ is known by the work of Blasius \cite{Bl} and Deligne.  The correct normalisation of $a_f(1)$ so that $\langle F_k, F_k \rangle = 1$ is

\begin{equation}
\label{holonorm}
| a_f( 1 ) |^2 = \kappa^k \prod_{i=1}^n \frac{ (4\pi)^{k_i} }{ \Gamma(k_i) } \frac{ \pi^n /2 }{ D L(1, \sym^2 \pi ) }.
\end{equation}

\subsection{The Regularised Unfolding Integral $I_\phi(T)$}
\label{holounfold}

In this section we construct our main object $I_\phi(T)$.  By computing it asymptotically in two ways, by contour shift and then unfolding, we will obtain a link between inner products and shifted convolution sums which will prove proposition \ref{Holo1}.  Choose a positive function $h \in C^\infty_0(\R^+)$, and let $g \in C^\infty_0( \F^+ )$ be its $n$-fold product.  Define $C_g = \langle E(g|z), 1 \rangle / \text{Vol}(Y)$.  Let

\begin{equation*}
\widetilde{g} = \sum_{u \in \OO^\times_+} g( uy )
\end{equation*}

be the symmetrisation of $g$ under the action of $\OO^\times_+$, and let

\begin{equation}
\label{cuspmellin}
G(s,m) = \int_{\F^+} g(y) Ny^{s-1} \lambda_m(y) dy
\end{equation}

be the Mellin transform of $\widetilde{g}$ thought of as a function on $\F^+ / \OO^\times_+$.  If $\F_+^1$ denotes the multiplicative subgroup of norm 1 elements, we may use the formula of Efrat \cite{Ef} for the volume of $\F_+^1 / \OO_+^\times$ to invert this, obtaining

\begin{equation}
\label{gtilde}
\widetilde{g}(y) = \frac{1}{ 2^n \pi i R} \sum_m \int_{(\sigma)} G(-s,-m) Ny^s \lambda_m(y) ds.
\end{equation}

Let $T \ge 1$ and consider the integral

\begin{equation}
\label{Iphi}
I_\phi(T) = \int_{\F^+} g(Ty) Ny^{-2} \left( \int_{ \F / \OO } \phi(z) |F_k(z)|^2 dx \right) dy,
\end{equation}

which may be rewritten by substituting (\ref{gtilde}) and refolding the Eisenstein series as

\begin{equation}
\label{Iphi2}
I_\phi(T) = \frac{1}{ 2^n \pi i R} \sum_m \int_{(\sigma)} G(-s,-m)T^{ns} \int_Y E(s,m,z) \phi(z) |F_k(z)|^2 dv ds.
\end{equation}

We first use a contour shift to relate $I_\phi(T)$ to the inner product $\langle \phi F_k, F_k \rangle$.

\begin{lemma}
\label{Iphimain}
For $\phi$ a fixed Hecke-Maass cusp form or pure incomplete Eisenstein series we have

\begin{equation*}
I_\phi(T) = C_g \langle \phi F_k, F_k \rangle T^n + O( T^{n/2} ).
\end{equation*}

\end{lemma}

\begin{proof}
Starting with equation (\ref{Iphi2}) and moving the contour of integration to the line $\text{Re}(s) = 1/2$, we write

\begin{equation*}
I_\phi(T) = C_g \langle \phi F_k, F_k \rangle T^n + R_\phi(T)
\end{equation*}

with $C_g$ coming from the pole of the Eisenstein series at $s=1$ (see section \ref{appvol} for this calculation).  $R_\phi(T)$ is the remaining integral along $\text{Re}(s) = 1/2$,

\begin{equation*}
R_\phi(T) = \int_Y p(z) \phi(z) |F_k(z)|^2 dv,
\end{equation*}

with

\begin{equation*}
p(z) = \frac{1}{ 2^n \pi i R} \sum_m \int_{(1/2)} G(-s,-m)T^{ns} E(s,m,z) ds.
\end{equation*}

From the Fourier series expansion of $E(s,m,z)$ and the bound for $K_{ir}$ we have 

\begin{equation*}
E(s,m,z) \ll Ny^{1/2} + Ny^{-n-1/2}(|s| + \| m \|)^{n+2} (1 + ( |s| + \| m \| )Ny^{-1/n} )^\epsilon,
\end{equation*}

so that $p(z) \ll \sqrt{Ny} T^{n/2}$ if $Ny \gg 1$.  It follows from this and the rapid decay of $\phi(z) |F_k(z)|^2$ that $R_\phi(T) \ll_{\phi,g} T^{n/2}$.

\end{proof}

Restating this with $I_\phi(T)$ expressed in the form (\ref{Iphi}) gives

\begin{equation}
\label{comparison}
C_g \langle \phi F_k, F_k \rangle T^n + O( T^{n/2} ) = \int_{\F^+} g(Ty) Ny^{-2} \left( \int_{ \F / \OO } \phi(z) |F_k(z)|^2 dx \right) dy,
\end{equation}

and we shall extract shifted convolution sums from the expression on the RHS after truncating our fixed form $\phi$.  Recall that this had a Fourier expansion

\begin{equation}
\label{fourier}
\phi(z) = \sum_{ \xi \in \OO} a_\xi(y) e( \tr( \xi \kappa x ) )
\end{equation}

with the $a_\xi(y)$ bounded as in lemma \ref{coeffbound}.  If $\phi$ is a pure incomplete Eisenstein series, then we find that the contribution to $I_\phi(T)$ from the tail of (\ref{fourier}) with $\| \xi \| \ge T^{1+\epsilon}$ for any $\epsilon > 0$ is bounded by

\begin{equation*}
I_1(T) T^{-n/2+A+\epsilon'} \sum_{ \| \xi \| \ge T^{1+\epsilon} } \tau(\xi) \| \xi \|^{-A + \epsilon'} \ll T^{3n/2 + \epsilon(n+1-A)}
\end{equation*}

by the support of $g$ and lemma \ref{coeffbound} (which is the source of the $\epsilon'$).  Here $I_1(T)$ is our main integral with $\phi$ chosen to be the constant function.  As a result, the contribution of these terms to $I_\phi(T)$ is $\ll T^{n/2}$ after choosing $A$ sufficiently large with respect to $\epsilon$, and a similar argument works when $\phi$ is a fixed cusp form.  If we define $\phi^*$ to be the truncated function

\begin{equation*}
\phi^* (z) = \sum_{ \| \xi \| < T^{1+\epsilon} } a_\xi(y) e( \tr( \xi \kappa x ) ),
\end{equation*}

we therefore have

\begin{equation}
\label{truncate}
C_g \langle \phi F_k, F_k \rangle T^n = \int_{\F^+} g(Ty) Ny^{-2} \left( \int_{ \F / \OO } \phi^*(z) |F_k(z)|^2 dx \right) dy + O( T^{n/2} ).
\end{equation}

\subsection{Extracting Shifted Convolution Sums}
\label{holoshift}

In this section we shall expand the RHS of (\ref{truncate}) using the Fourier expansion of $\phi^*$, writing

\begin{equation*}
C_g \langle \phi F_k, F_k \rangle T^n = S_0(T) + \sum_{0 < \| \xi \| < T^{1+\epsilon} } S_\xi(T) + O( T^{n/2} )
\end{equation*}

where for any $\xi \in \OO$ we define

\begin{equation*}
S_\xi(T) = \int_{\F^+} g(Ty) Ny^{-2} \left( \int_{ \F / \OO } a_\xi(y) e( \tr(\xi \kappa x ) ) |F_k(z)|^2 dx \right) dy.
\end{equation*}

Note that this definition gives us (\ref{Maassint}) of proposition \ref{Holo1}.  The aim of this section is to analyse the objects $S_\xi(T)$ so that when we divide through by $C_g T^n$ we have the remaining equations and bounds of proposition \ref{Holo1}.  We first note that $S_0(T) = 0$ for $\phi$ a cusp form and by lemma \ref{coeffbound} we have

\begin{equation}
\label{S0}
S_0(T) = \left( \frac{ \langle \phi, 1 \rangle }{ \text{Vol}(Y) } + O(T^{-n/2}) \right) I_1(T)
\end{equation}

for $\phi$ a pure incomplete Eisenstein series.  We shall treat $I_1(T)$ and $S_\xi(T)$ for $\xi \neq 0$ seperately, beginning with $\xi \neq 0$.  Squaring out $|F_k(z)|^2$ and integrating in $x$ gives

\begin{equation*}
S_\xi(T) = \sqrt{D} \sum_{ \eta > 0} \overline{ a_f(\eta) } a_f( \eta + \xi ) \left( \int_{\F^+} g(Ty) a_\xi(y) y^{k-2} e^{-2\pi \tr( (2\eta + \xi) \kappa y)} dy \right).
\end{equation*}

As the exponentials and $g$ are positive, this satisfies

\begin{equation*}
S_\xi(T) \ll | a_\xi(T^{-1}) |  \sum_{ \eta > 0} | a_f(\eta) a_f( \eta + \xi ) | \left( \int_{\F^+} g(Ty) y^{k-2} e^{-2\pi \tr( (2\eta + \xi) \kappa y)} dy \right).
\end{equation*}

Appealing to the Mellin transform $H$ of $h$ and applying the normalisations of $a_f(\eta)$ and $a_f(\eta + \xi)$, we may integrate in $y$ to obtain 

\begin{multline}
\label{unfold1}
S_\xi(T) \ll \frac{ | a_\xi(T^{-1}) | }{ Nk L( 1, \sym^2 \pi ) }  \sum_{ \eta > 0} | \lambda_\pi(\eta) \lambda_\pi( \eta + \xi ) | \\
\times \prod_{i=1}^n \left( \frac{ \sqrt{ \eta_i(\eta_i + \xi_i) } }{\eta_i + \xi_i/2} \right)^{k_i-1} \frac{1}{2\pi i} \int_{ (\sigma) } H(-s) \left( \frac{T}{4\pi \kappa_i ( \eta_i + \xi_i/2 ) } \right)^s \frac{ \Gamma(s+k_i-1) }{ \Gamma(k_i-1) } ds.
\end{multline}

Note that $\sqrt{ \eta_i (\eta_i + \xi_i) } \le \eta_i + \xi_i/2$, so that these factors may be omitted.  We may simplify this expression using a lemma seen in the work of Luo and Sarnak.  By \cite{LS}, we have

\begin{equation}
\label{luosarnak}
\frac{ \Gamma( s + k_i - 1 ) }{ \Gamma( k_i - 1) } = (k_i - 1)^s ( 1 + O_{a,b}(( |s| + 1 )^2 k_i^{-1} ) ),
\end{equation}

which holds by Stirling's formula for any vertical strip $0 < a \le \text{Re}(s) \le b$.  If we apply this to (\ref{unfold1}) we may invert the Mellin transform of $h$ to obtain

\begin{multline}
\label{unfold3}
S_\xi(T) \ll \frac{ | a_\xi(T^{-1}) | }{ Nk L( 1, \sym^2 \pi ) }  \sum_{ \eta > 0} | \lambda_\pi(\eta) \lambda_\pi( \eta + \xi ) | \\
\times \prod_{i=1}^n \left( h \left( \frac{ T(k_i-1) }{ 4\pi (\eta_i + \xi_i/2) } \right) + O \left( k_i^\epsilon \left( \frac{T}{\eta_i + \xi_i/2} \right)^{1+\epsilon} \right) \right).
\end{multline}

The final step in proving (\ref{unfold}) from this is showing that when this product is expanded out, the total contribution from all the error terms is $\ll Nk \| k \|^{-\nu +\epsilon} T^{n+\epsilon}$.  It is enough to consider one such term which contains `main term' factors at the first $t$ places and error term factors at the last $n-t$.  As the factors of $h$ provide a truncation at the first $t$ places, the contribution from this term is bounded by

\begin{equation}
\label{unfoldexpand}
\ll \sum_{ \substack{ \eta \in \OO \\ | \eta_i + \xi_i/2 | \ll Tk_i, \: i \le t } } | \lambda_\pi(\eta) \lambda_\pi( \eta + \xi ) | \prod_{i>t} k_i^\epsilon \left( \frac{T}{\eta_i + \xi_i/2} \right)^{1+\epsilon}.
\end{equation}

If we let $\tau = \eta + \xi/2$, then $\tau \in \tfrac{1}{2} \OO^+$ (because we may assume $\eta_i$ and $\eta_i + \xi_i$ are positive), and because $|\xi_i| \ll T^{1+\epsilon}$ we have $\tau_i + T^2 \gg \max( \eta_i, \eta_i + \xi_i)$ for all $i$.  Therefore by Deligne's bound,

\begin{equation*}
\lambda_\pi(\eta) \lambda_\pi( \eta + \xi ) \ll N\eta^\epsilon N(\eta+\xi)^\epsilon \ll \| \tau \|^\epsilon + T^\epsilon
\end{equation*}

and the expression (\ref{unfoldexpand}) may be simplified to

\begin{equation*}
\ll \sum_{ \tau_i \ll Tk_i, \: i \le t } ( \| \tau \|^\epsilon + T^\epsilon ) \prod_{i>t} k_i^\epsilon T^{1+\epsilon} \tau_i^{-1-\epsilon}.
\end{equation*}

Because $\tau_i \ll Tk_i$ for $i \le t$, this may be further reduced to

\begin{equation*}
\ll Nk^\epsilon T^{n-t+\epsilon} \sum_{ \substack{ \tau_i \ll Tk_i, \\ i \le t } } \prod_{i>t} \tau_i^{-1-\epsilon}.
\end{equation*}

If we project the set $\{ \tau \in \tfrac{1}{2} \OO^+ : \tau_i \ll Tk_i, i \le t \}$ onto the last $n-t$ real places, we obtain a set $\OO' \in \R^{n-t}$, any two of whose elements are a distance $\gg \delta = ( T^t \prod_{i \le t} k_i )^{-1/(n-t)}$ from each other and the origin.  The sum above may therefore be bounded by

\begin{eqnarray*}
& \ll & T^t \prod_{i \le t} k_i \int_{x_i \gg \delta} \prod_{i > t} x_i^{-1-\epsilon} dx_i \\
& \ll & T^{t+\epsilon} \prod_{i \le t} k_i^{1+\epsilon},
\end{eqnarray*}

so that the total contribution of our error term is $\ll Nk^\epsilon T^{n+\epsilon} \prod_{i \le t} k_i^{1+\epsilon}$.  As $t < n$, we are omitting a factor of size at least $\| k \|^\nu$ from $Nk$, so this is $\ll Nk \| k \|^{-\nu +\epsilon} T^{n+\epsilon}$ as required.  Therefore

\begin{multline*}
S_\xi(T) \ll \frac{ 1}{ Nk L( 1, \sym^2 \pi ) } | a_\xi(T^{-1}) |  \biggl( \sum_{ \eta > 0} | \lambda_\pi(\eta) \lambda_\pi( \eta + \xi ) | \prod_{i=1}^n h \left( \frac{ T(k_i-1) }{ 4\pi (\eta_i + \xi_i/2) } \right) \\
+ O( Nk |k|^{-\nu +\epsilon} T^{n+\epsilon} ) \biggr),
\end{multline*}

which is the bound (\ref{unfold}).\\

We now deal with the case $\xi = 0$.  Squaring out $|F_k(z)|^2$ and integrating in $x$ gives

\begin{equation*}
I_1(T) = \sqrt{D} \sum_{ \eta > 0} |a_f(\eta)|^2 \left( \int_{\F^+} g(Ty) y^{k-2} e^{-4\pi \tr(\eta \kappa y)} dy \right).
\end{equation*}

Expressing $a_f$ in terms of $\lambda_\pi$ and symmetrising by the action of $\OO_+^\times$, this becomes

\begin{eqnarray*}
I_1(T) & = & \sqrt{D} |a_f(1)|^2 \kappa^{1-k} \sum_{ \eta > 0} |\lambda_\pi(\eta)|^2 \left( \int_{\F^+} g(Ty) (\eta \kappa y)^{k-1} e^{-4\pi \tr(\eta \kappa y)} dy^\times \right) \\
 & = & \sqrt{D} |a_f(1)|^2 \kappa^{1-k} \sum_{ ( \eta ) > 0} |\lambda_\pi(\eta)|^2 \left( \int_{\F^+} g(Ty) \widetilde{ \psi }(\eta \kappa y) dy^\times \right) \\
  & = & \sqrt{D} |a_f(1)|^2 \kappa^{1-k} \sum_{ ( \eta ) > 0} |\lambda_\pi(\eta)|^2 \left( \int_{\F^+ / \OO^\times_+} \widetilde{g}(Ty) \widetilde{ \psi }(\eta \kappa y) dy^\times \right),
\end{eqnarray*}

where

\begin{equation*}
\widetilde{g}(y) = \sum_{u \in \OO^\times_+} g( uy ), \quad \widetilde{ \psi }(y) = \sum_{u \in \OO^\times_+} (uy)^{k-1} \exp( -4 \pi \tr( uy )).
\end{equation*}

If we let $G(s,m)$ be the Mellin transform of $\widetilde{g}$ as in (\ref{cuspmellin}), then by the Mellin inversion formula we have

\begin{multline*}
I_1(T)  = \frac{ \sqrt{D} }{ 2^{n-1} R } |a_f(1)|^2 \kappa^{1-k} \sum_{ ( \eta ) } |\lambda_\pi(\eta)|^2 \frac{1}{2\pi i} \sum_m \int_{(\sigma)} G(-s,-m) \left( \frac{T}{ 4\pi } \right)^{ns} \\
N(\eta \kappa)^{-s} \lambda_{-m}(\eta \kappa) \prod_{i=1}^n (4\pi)^{-k_i + 1} \Gamma( s + \beta(m,i) + k_i - 1 ) ds.
\end{multline*}

Forming the $L$-function from the sum over $\eta$, this becomes

\begin{multline*}
I_1(T)  = \frac{ \sqrt{D} }{ 2^{n-1} R } |a_f(1)|^2  \kappa^{1-k-s} \lambda_{-m}(\kappa) \frac{1}{2\pi i} \sum_m \int_{(\sigma)} G(-s,-m) \left( \frac{T}{ 4\pi } \right)^{ns} \\
L( s, \sym^2 \pi \otimes \lambda_{-m} ) \frac{ L( s, \lambda_{-m} ) }{ L( 2s, \lambda_{-2m} ) } \prod_{i=1}^n (4\pi)^{-k_i + 1} \Gamma( s + \beta(m,i) + k_i - 1 ) ds.
\end{multline*}

When we substitute the value of $|a_f(1)|^2$ and shift the line of integration to $\sigma = 1/2$, we pick up a main term from the pole at $s=1$ which is

\begin{equation*}
\frac{ \pi^n G(-1,0) T^n }{ 2 D \zeta_F(2) },
\end{equation*}

and in section \ref{appvol} it is shown that this agrees with the expected main term $C_g T^n$.  We therefore have

\begin{equation}
\label{I1}
I_1(T) = C_g T^n + E_{1/2}(T),
\end{equation}

with

\begin{multline}
\label{mainbound}
E_{1/2}(T) = \frac{ (2\pi^2)^n N\kappa^{1/2} \lambda_{-m}(\kappa) }{ R \sqrt{D} L( 1, \sym^2 \pi ) } \frac{1}{2\pi i}  \sum_m \int_{(1/2)} G(-s,-m) \left( \frac{T}{ 4\pi } \right)^{ns} \\
L( s, \sym^2 \pi \otimes \lambda_{-m} ) \frac{ L( s, \lambda_{-m} ) }{ L( 2s, \lambda_{-2m} ) } \prod_{i=1}^n \frac{ \Gamma( s + \beta(m,i) + k_i - 1 ) }{ (k_i-1) \Gamma( k_i - 1) } ds.
\end{multline}

If we apply the Luo-Sarnak lemma to \ref{mainbound} in the form

\begin{equation*}
\left| \frac{ \Gamma( s + \beta(m,i) + k_i - 1 ) }{ \Gamma( k_i - 1) } \right| \ll (k_i - 1)^{\text{Re}(s)} (1 + |s| + \| m \|)^2,
\end{equation*}

together with the rapid decay of $G(s,m)$, the convex bound for $L( \lambda_{-m}, s)$ and any lower bound of the form $L( \lambda_{-2m}, 1 + it) \gg ( |t| + \| m \| + 1)^{-A}$, we obtain

\begin{equation}
\label{E}
E_{1/2}(T) \ll \left( \frac{T^n}{ Nk } \right)^{1/2} \frac{ 1 }{ L( 1, \sym^2 \pi ) } \sum_m \int_{-\infty}^{+\infty} \frac{ |L( 1/2+it, \sym^2 \pi \otimes \lambda_{-m} )| }{ (|t| + \| m \| + 1)^A } ds.
\end{equation}

The asymptotic (\ref{eisint}) and bound (\ref{Rbound}) for the error now follow by combining (\ref{E}), (\ref{I1}) and (\ref{S0}), which completes the proof of proposition \ref{Holo1}.

\section{Sieving for Mass Equidistribution: The Mixed Case}
\label{sievemixed}

In this section we generalise proposition \ref{Holo1} to allow complex places of $F$.  As we may no longer talk about holomorphic forms we now let $F_k$ be a vector valued cohomological form with associated automorphic representation $\pi$, and as before assume the existence of a $\nu > 0$ such that $k_i \ge \| k \|^\nu$ for all $i$.  Our bound for $\langle \phi F_k, F_k \rangle$ in terms of shifted convolution sums is as follows.

\begin{prop}
\label{Holomixed}
Let $T \ge 1$ and $\epsilon > 0$.  Fix $h \in C^\infty_0(\R^+)$ positive and let $g \in C^\infty_0 ( \R_+^r)$ be its r-fold product, and define $C_g = \langle E(g | z), 1 \rangle / Vol(Y)$.  Let $J$ be the set of $r_2$-tuples $J = \{ (j_{r_1+1}, \ldots, j_{r}) | 0 \le j_i \le k_i \}$.  Fix an automorphic form $\phi$ with Fourier expansion

\begin{equation*}
\phi(z) = \sum_{ \xi \in \OO} a_\xi(y) e( \tr( \xi \kappa x ) ).
\end{equation*}

If $\phi$ is a Hecke-Maass cusp form, then

\begin{equation}
\label{Maassintmixed}
\langle \phi F_k, F_k \rangle = c^{-1}T^{-n} \sum_{0 < \| \xi \| < T^{1+\epsilon} } S_\xi(T) + O(T^{-n/2}).
\end{equation}

If $\phi$ is a pure incomplete Eisenstein series, then

\begin{equation}
\label{eisintmixed}
\langle \phi F_k, F_k \rangle = \frac{1}{ Vol(Y) } \langle \phi, 1 \rangle + c^{-1}T^{-n} \sum_{0 < \| \xi \| < T^{1+\epsilon} } S_\xi(T) + O\left( \frac{1 + R_k(f)}{T^{n/2}} \right)
\end{equation}

with

\begin{equation}
\label{Rboundmixed}
R_k(f) = \frac{1}{\sqrt{Nk} L( 1, \sym^2 \pi ) } \sum_m \int_{-\infty}^{+\infty} \frac{ |L( 1/2+it, \sym^2 \pi \otimes \lambda_{-m} )| }{ (|t| + \| m \| + 1)^A } dt.
\end{equation}

Furthermore, we have the bound

\begin{equation}
\label{unfoldmixed}
S_\xi(T) \ll \frac{ | a_\xi(T^{-1}) |  }{ Nk L( 1, \sym^2 \pi ) } \left( \sum_{j \in J} A_{\xi,j} + O( Nk \| k \|^{-\nu +\epsilon} T^{n+\epsilon} ) \right),
\end{equation}

where

\begin{multline}
\label{unfoldj}
A_{\xi,j}(T) \ll \Biggl( \sum_{\eta > 0} | \lambda_\pi(\eta) \lambda_\pi( \eta + \xi )|  \prod_{i \le r_1} \frac{ 1 }{ k_i } h\left( \frac{T (k_i-1) }{ 4\pi |\eta_i \kappa_i| } \right) \\
\times \prod_{i > r_1} \frac{1}{ k_i (k_i - j_i) j_i }  h\left( \frac{T \sqrt{ j_i(k_i-j_i) } }{ 2\pi |\eta_i \kappa_i| } \right) \Biggr).
\end{multline}

\end{prop}

Theorem \ref{Home} follows from combining this with proposition \ref{shiftbound} as in the totally real case.  Most of the added difficulty in the proof of proposition \ref{Holomixed} comes from the changes to the Fourier expansion of $F_k$ in the presence of complex places.  The first difference is that $F_k$ is vector valued, which is the source of the summation over $J$ in (\ref{unfoldmixed}).  The second is that its Fourier coefficients contain Bessel functions, and so multiple Bessel integrals appear when we bound $S_\xi(T)$ in terms of shifted convolution sums.  Section \ref{mixedrevise} below contains the Fourier expansions and $L^2$ normalisations of the automorphic forms on $Y$ we shall work with, as well as the revised definition of $I_\phi(T)$ and its expression in terms of the shifted convolution integrals $S_\xi(T)$.  The remainder of the proof of proposition \ref{Holomixed} will then lie in analysing $S_\xi(T)$, which we do in section \ref{mixedshift1} in the case $\xi \neq 0$ and in section \ref{mixedshift2} in the case $\xi = 0$.

\subsection{Revision of Basic Definitions}
\label{mixedrevise}

The bounds we shall use for the Fourier coefficients of Maass forms and Eisenstein series are essentially unchanged from the totally real case.  If $\phi$ is a Maass cusp form with spectral parameter $r = (r_i)$ then we have the expansion

\begin{equation*}
\phi(z) = \sqrt{Ny} \sum_{ \xi \neq 0 } \rho(\xi) \prod_{p=1}^r K_{ir_p}(2\pi \delta_p |\xi_p \kappa_p| y_p )  e( \tr( \xi \kappa x ) )
\end{equation*}

where the $\rho(\xi)$ satisfy the Ramanujan bound on average, i.e.

\begin{equation}
\label{avramanujan1}
\sum_{ \| \xi \| \le T } | \rho(\xi) | \ll T^n.
\end{equation}

The coefficients of a pure incomplete Eisenstein series may again be expressed in terms of those of complete Eisenstein series $E(s, m, z )$.  The Fourier expansion of these is computed in section \ref{appeisenstein}, following Efrat in the totally real case \cite{Ef}, and is

\begin{multline*}
E(s, m, z ) = Ny^s \lambda_m(y) + \phi(s,m) Ny^{1-s} \lambda_{-m}(y) + \frac{2^r \pi^{ns - r_2} }{ \sqrt{|D|} } \sqrt{Ny} \times \\
\sum_{ \xi \neq 0 } N(\delta \xi \kappa )^{s - 1/2} \lambda_m(\delta \xi \kappa ) \prod_{p=1}^r \frac{ K_{ \delta_p (s-1/2) + \beta(m,p) } ( 2\pi \delta_p |\xi_p \kappa_p| y_p ) }{ \Gamma(\delta_p s + \beta(m,p)) } \frac{ \sigma_{1-2s, -2m}(\xi \kappa) } {\zeta(2s, \lambda_{-2m} )} e( \tr(\xi \kappa x)),
\end{multline*}

where $\beta(m,p)$ is as in (\ref{beta}) and

\begin{eqnarray*}
\phi(s) & = & \frac{ \pi^{n/2} }{ \sqrt{|D|} } \prod_{p \le r_1} \frac{ \Gamma( s + \beta(m,p) -1/2) }{ \Gamma( s + \beta(m,p) ) } \prod_{p > r_1} \frac{2}{ 2s + \beta(m,p) -1} \frac{ \zeta( 2s-1, \lambda_{-2m} ) }{ \zeta( 2s, \lambda_{-2m} ) } \\
& = & \frac{ \theta(s-1/2) }{ \theta(s) }, \\
\theta(s) & = & |D|^s \pi^{-ns} \prod_{p \le r_1} \Gamma( s + \beta(m,p) ) \prod_{p > r_1} 2^{-2s} \Gamma( 2s + \beta(m,p) )\zeta(2s, \lambda_{-2m} ), \\
\sigma_{1-2s,-2m}(\xi \kappa) & = & \sum_{\substack{ ( c ) \\ \xi/c \in \OO } } \frac{ \lambda_{-2m}(c) }{ |Nc|^{2s-1} }.
\end{eqnarray*}

By Mellin inversion we again have the two asymptotics

\begin{eqnarray}
\notag
a_0(y) & = & \psi( Ny ) \lambda_m(y) + O( Ny^{-1} ), \\
\label{constantterm}
& = & \frac{1}{ \text{Vol}(Y) } \langle E(\psi, m | z), 1 \rangle + O( Ny^{1/2} )
\end{eqnarray}

for the zeroth Fourier coefficient of $E(\psi, m | z)$ as $Ny$ tends to 0 and infinity.  The bounds on the nonzero coefficients of Hecke-Maass cusp forms and pure incomplete Eisenstein series are also unchanged, and so lemma \ref{coeffbound} of section \ref{holofourier} continues to hold.  We recall the formula for the Fourier expansion of the vector valued function $F_k$ in $\HH_F'$:

\begin{equation*}
F_k(z) = \sum_{\eta > 0} a_f(\eta) {\bf K}_k( \eta \kappa y) e( \tr( \eta \kappa x) ),
\end{equation*}

where ${\bf K}_k(y)$ are as in (\ref{whittaker1}) and (\ref{whittaker2}).  The coefficients $a_f(\eta)$ satisfy the proportionality relation $a_f(\eta ) = \lambda_\pi(\eta) N\eta^{-1/2} a_f( 1 )$, where $a_f(1)$ is given by (\ref{firstfourier}).  As before, we assume the Ramanujan bound $| \lambda_\pi(\eta) | \le \tau(\eta)$; see the discussion of section \ref{weightresults} for the circumstances under which this is known.

$I_\phi(T)$ is still a regularised unfolding integral over $\Gamma_U$, and when constructing it we bear in mind that $\HH_F \simeq \R_+^r \times \F$, and $\Gamma_U \backslash \HH_F \simeq \R_+^r \times (\F / \OO)$.  To define $I_\phi(T)$ we therefore choose a positive function $h \in C^\infty_0(\R^+)$ and let $g \in C^\infty_0( \R_+^r )$ be its $r$-fold product.  Let

\begin{equation*}
\widetilde{g} = \sum_{u \in \OO^\times_+} g( |u| y )
\end{equation*}

be the symmetrisation of $g$ under the action of $\OO^\times_+$, and let

\begin{equation}
\label{cuspmellin1}
G(s,m) = \int_{\R_+^r} g(y) Ny^s \lambda_m(y) dy^\times
\end{equation}

be the Mellin transform of $\widetilde{g}$ thought of as a function on $\R_+^r / \OO^\times_+$.  After calculating the volume of $\R_+^r / \OO_+^\times$ with the restriction of hyperbolic measure as in Efrat \cite{Ef}, we may invert to obtain

\begin{equation}
\label{cuspmellin2}
\widetilde{g}(y) = \frac{ 1 }{ 2 \pi i V_c } \sum_m \int_{(\sigma)} G(-s,-m) Ny^s \lambda_m(y) ds,
\end{equation}

where $V_c$ is the volume of $\R_+^r / \OO_+^\times$ and is equal to $2^{r_1 - r_2 -1 + \delta_{0 r_1}} R$ (see section \ref{appvol} for this calculation).  We now define $I_\phi(T)$ to be

\begin{equation}
\label{Iphimixed}
I_\phi(T) = \int_{\R_+^r} g(Ty) Ny^{-1} \left( \int_{ \F / \OO } \phi(z) |F_k(z)|^2 dx \right) dy^\times.
\end{equation}

To rewrite this in terms of integrals against Eisenstein series, we symmetrise over $\OO_+^\times$ and substitute (\ref{cuspmellin2}), giving

\begin{eqnarray*}
\notag
I_\phi(T) & = & \frac{ 1 }{ 2 \pi i V_c } \sum_m \int_{(\sigma)} G(-s,-m)T^{ns} \int_{\R_+^r / \OO_+^\times } Ny^{s-1} \lambda_m(y) \\
\notag
& \quad & \quad \quad \int_{ \F / \OO } \phi(z) |F_k(z)|^2 dx dy^\times ds \\
\notag
& = &  \frac{ \omega_+ }{ 2 \pi i V_c } \sum_m \int_{(\sigma)} G(-s,-m)T^{ns} \int_{ \Gamma_\infty \backslash \HH_F } Ny^s \lambda_m(y) \phi(z) |F_k(z)|^2 dv ds \\
\label{Iphimixed2}
 & = &  \frac{ \omega_+ }{ 2 \pi i V_c } \sum_m \int_{(\sigma)} G(-s,-m)T^{ns} \int_Y E(s, m, z ) \phi(z) |F_k(z)|^2 dv ds.
\end{eqnarray*}

On shifting the line of integration to $\sigma = 1/2$, we have the asymptotic

\begin{equation}
\label{Iphimain2}
I_\phi(T) = C_g \langle \phi F_k, F_k \rangle T^n + O( T^{n/2} ).
\end{equation}

(See section \ref{appvol} for the verification that the residue at $s=1$ is correct.)  Comparing this with the form (\ref{Iphimixed}) of $I_\phi(T)$ and truncating the Fourier expansion of $\phi$ to those terms with $\| \xi \| \ll T^{1+\epsilon}$ we arrive at the equation 

\begin{eqnarray}
\label{truncatemixed}
c \langle \phi F_k, F_k \rangle T^n + O( T^{n/2} ) & = & \int_{\R_+^r} g(Ty) Ny^{-1} \left( \int_{ \F / \OO } \phi^*(z) |F_k(z)|^2 dx \right) dy^\times \\
\notag
\text{where} \quad \phi^*(z) & = & \sum_{ \| \xi \| < T^{1+\epsilon} } a_\xi(y) e( \tr( \xi \kappa x ) ),
\end{eqnarray}

which is the starting point for our analysis of Fourier coefficients.

\subsection{Extracting Shifted Convolution Sums: Nonzero Shifts}
\label{mixedshift1}

Define $S_\xi(T)$ to be the contribution of the $\xi$th Fourier coefficient of $\phi$ to (\ref{truncatemixed}) as before.  It remains to estimate $S_\xi(T)$ in terms of shifted convolution sums, which we do first when $\xi \neq 0$.  Squaring out $|F_k(z)|^2$ and integrating in $x$ gives

\begin{equation*}
S_\xi(T) = 2^{-r_2} \sqrt{D} \sum_{ \eta > 0} a_f(\eta) \overline{ a_f( \eta + \xi ) } \left( \int_{\R_+^r} g(Ty) a_\xi(y) \langle {\bf K}_k( \eta \kappa y ), {\bf K}_k( (\eta + \xi) \kappa y ) \rangle Ny^{-1} dy^\times \right).
\end{equation*}

Applying H\"older's inequality, this becomes

\begin{eqnarray*}
S_\xi(T) & \ll & | a_\xi(T^{-1}) | \sum_{ \eta > 0} | a_f(\eta) a_f( \eta + \xi ) |  \Biggl( \int_{\R_+^r} g(Ty) \biggl[ \left( \frac{ N(\eta + \xi)}{ N\eta } \right)^{1/2} | {\bf K}_k( \eta \kappa y) |^2 \\
&& + \left( \frac{ N\eta }{ N(\eta + \xi)} \right)^{1/2} | {\bf K}_k( (\eta + \xi) \kappa y) |^2 \biggr]  Ny^{-1} dy^\times \Biggr). \\
\end{eqnarray*}

The second term in this integral behaves identically to the first, and we ignore it for simplicity.  Applying a change of variable and the normalisation $a_f(\eta) = \lambda_\pi(\eta) N\eta^{-1/2} a_f(1)$, we have

\begin{eqnarray}
\notag
S_\xi(T) & \ll & |a_\xi(T^{-1}) | \sum_{ \eta > 0} | a_f(\eta) a_f( \eta + \xi )| N(\eta (\eta + \xi))^{1/2} \\
\notag
& \quad & \quad \int_{\R_+^r} g(T |\eta \kappa |^{-1} y) | {\bf K}_k(y) |^2 Ny^{-1} dy^\times \\
\label{sxi1}
& \ll & |a_\xi(T^{-1})|  |a_f(1)|^2 \sum_{ \eta > 0} | \lambda_\pi(\eta) \lambda_\pi( \eta + \xi )| \\
\notag
& \quad & \quad \int_{\R_+^r} g(T |\eta \kappa |^{-1} y) | {\bf K}_k(y) |^2 Ny^{-1} dy^\times.
\end{eqnarray}

We have the following formula for $|K_f|^2$ from (\ref{whittaker1}) and (\ref{whittaker2}),

\begin{equation*}
|{\bf K}_k(y)|^2 = \prod_{i \le r_1} y_i^{k_i} \exp( -4\pi y_i) \prod_{i > r_1} y_i^{k_i+2} \sum_{j=0}^{k_i} \binom{k_i}{j} | K_{k_i/2 - j}(4\pi y_i) |^2,
\end{equation*}

where we now take $y_i \in \R$ for all $i$.  For a multi-index $j = (j_i) \in J$ we define $K_{k,j}(y)$ to be the corresponding term in the formula for ${\bf K}_k(y)$ so that

\begin{equation}
\label{whittakerj}
|K_{k,j}(y)|^2 = \prod_{i \le r_1} y_i^{k_i} \exp( -4\pi y_i) \prod_{i > r_1} y_i^{k_i+2} \binom{k_i}{j_i} | K_{k_i/2 - j}(4\pi y_i) |^2,
\end{equation}

and define $S_{\xi,j}(T)$ be the corresponding term in $S_\xi(T)$.  We shall partition $J$ as $J_0 \cup J_1$, where $J_0 = \{ j \; | \; \min( j_i, k_i - j_i ) > k_i^{1/2} \}$.  The reason for seperating the indices in this way is that for $j \in J_0$, the arguments of all the Bessel functions appearing in (\ref{whittakerj}) are bounded away from $\pm k_i/2$.  As a result, when we calculate the Mellin transforms of $|K_{k,j}(y)|^2$ the gamma factors which appear have arguments with large real parts, and so we may approximate them well using the Luo-Sarnak lemma.  For $j \in J_0$ this lets us give good bounds for $S_{\xi,j}(T)$, while a weaker bound will suffice for the remaining terms because $J_1$ is small (in fact $|J_1| \ll \| k \|^{-\nu/2} |J|$).  We begin by deriving this weak bound for all $j$, interchanging the sum and integral in (\ref{sxi1}) to obtain

\begin{equation}
\label{sxi2}
S_{\xi,j}(T) \ll |a_\xi(T^{-1})| |a_f(1)|^2  \int_{\R_+^r} | K_{k,j}(y) |^2 Ny^{-1} \sum_{ \eta > 0} | \lambda_\pi(\eta) \lambda_\pi( \eta + \xi )| g(T |\eta \kappa |^{-1} y) dy^\times.
\end{equation}

The inner function

\begin{equation*}
\sum_{ \eta > 0} | \lambda_\pi(\eta) \lambda_\pi( \eta + \xi )| g(T |\eta \kappa |^{-1} y)
\end{equation*}

is bounded above by the sum over $\eta$ such that $Ty_i \ll |\eta_i| \ll Ty_i$ for all $i$, weighted by $| \lambda_\pi(\eta) \lambda_\pi( \eta + \xi )| \ll \| Ty \|^\epsilon$, from which it follows that 

\begin{equation*}
\sum_{ \eta > 0} | \lambda_\pi(\eta) \lambda_\pi( \eta + \xi )| g(T |\eta \kappa |^{-1} y) \ll T^{n+\epsilon} Ny \prod_{i=1}^r (1+y_i^\epsilon).
\end{equation*}

Applying this to (\ref{sxi2}) gives the upper bound

\begin{eqnarray*}
S_{\xi,j}(T) \ll T^{n+\epsilon} |a_\xi(T^{-1})| |a_f(1)|^2  \int_{\R_+^r} | K_{k,j}(y) |^2 \prod_{i=1}^r (1+y_i^\epsilon) dy^\times.
\end{eqnarray*}

We may factorise this integral as a product over the Archimedean places, and at each place we will bound the product of the local integral and the corresponding terms of $|a_f(1)|^2$.  The factor corresponding to a real place $i \le r_1$ is

\begin{eqnarray*}
\frac{ (4\pi)^{k_i} }{ \Gamma(k_i) } \int_0^\infty y^{k_i} \exp(-4\pi y) (1 + y^\epsilon) dy^\times & \ll & \frac{ 1 }{ \Gamma(k_i) } ( \Gamma(k_i) + \Gamma(k_i +\epsilon) ) \\
 & \ll & k_i^\epsilon.
\end{eqnarray*}

For $i > r_1$, it is

\begin{eqnarray*}
\frac{ (2\pi)^{k_i} }{ \Gamma( k_i/2 + 1)^2 } \binom{k_i}{j_i} \int_0^\infty y^{k_i+2} | K_{k_i/2 - j_i}(4\pi y) |^2 (1 + y^\epsilon) dy^\times,
\end{eqnarray*}

and we may evaluate this using the following formula, taken from \cite{GR}:

\begin{equation}
\label{bessel}
\int_0^\infty y^\lambda K_\mu(y) K_\nu(y) dy = \frac{ 2^{\lambda-2} }{ \Gamma( \lambda+1 ) } \prod_\pm \Gamma \left( \frac{ 1 + \lambda \pm \mu \pm \nu }{2} \right).
\end{equation}

Applying this, we obtain

\begin{multline*}
\frac{1}{ \Gamma( k_i/2 + 1)^2 } \binom{k_i}{j_i} \Biggl( \frac{ \Gamma( k_i/2 + 1)^2 }{ \Gamma(k_i+2) } \Gamma(j_i+1) \Gamma(k_i-j_i+1) \\
+ \frac{ \Gamma( k_i/2 + 1 +\epsilon/2)^2 }{ \Gamma(k_i+2 + \epsilon) } \Gamma(j_i+1 +\epsilon/2) \Gamma(k_i-j_i+1 +\epsilon/2) \Biggr) \ll k_i^{1-\epsilon}.
\end{multline*}

Multiplying these local integrals gives the bound

\begin{equation}
S_{\xi,j}(T) \ll \frac{ |a_\xi(T^{-1})| }{ L( 1, \sym^2 \pi ) } Nk^\epsilon \prod_{i>r_1} k_i^{-1},
\end{equation}

and so the contribution to $S_\xi(T)$ from all $j \in J_1$ is bounded above by

\begin{equation*}
\ll \frac{ |a_\xi(T^{-1})| }{ L( 1, \sym^2 \pi ) } \| k \|^{-\nu/2 + \epsilon}.
\end{equation*}

We shall treat the terms with $j \in J_0$ more carefully, by factorising the inner integral in (\ref{sxi1}) and using Mellin inversion to estimate each factor.  As before, we shall pair each local integral with the corresponding factor from  $|a_f(1)|^2$.  For $i \le r_1$ we need to consider

\begin{equation}
\label{sxiint1}
\frac{ (4\pi)^{k_i} }{ \Gamma(k_i) } \int_{\R^+} h( T |\eta_i \kappa_i|^{-1} y) y^{k_i-1} \exp(-4\pi y) dy^\times,
\end{equation}

which by Mellin inversion is equal to

\begin{eqnarray}
\label{sxiint2}
\frac{ 4\pi }{ k_i -1 } \int_{(\sigma)} H(-s) \left( \frac{T}{ 4\pi |\eta_i \kappa_i| } \right)^s \frac{ \Gamma( s+k_i-1) }{ \Gamma(k_i-1) } ds.
\end{eqnarray}

Applying the Luo-Sarnak lemma to this, we have

\begin{eqnarray*}
(\ref{sxiint1}) & = & \frac{ 4\pi }{ k_i -1 } \int_{(\sigma)} H(-s) \left( \frac{T}{ 4\pi |\eta_i \kappa_i| } \right)^s (k_i-1)^s ( 1 + O( (|s|+1)^2 k_i^{-1} ) ) ds \\
& \ll & \frac{ 1 }{ k_i } \left( h\left( \frac{T (k_i-1) }{ 4\pi |\eta_i \kappa_i| } \right) + \left| \int_{(\sigma)} H(-s) \left( \frac{T}{ 4\pi |\eta_i \kappa_i| } \right)^s (k_i-1)^{s-1}  (|s|+1)^2 ds \right| \right).
\end{eqnarray*}

Choosing $\sigma = 1 +\epsilon$ gives the final bound

\begin{equation}
\label{sxiint3}
(\ref{sxiint1}) \ll \frac{ 1 }{ k_i } \left( h\left( \frac{T (k_i-1) }{ 4\pi |\eta_i \kappa_i| } \right) + O \left( k_i^\epsilon ( T/ |\eta_i|)^{1+\epsilon} \right) \right).
\end{equation}

For $i > r_1$, we must consider the integral 

\begin{equation}
\label{sxiint4}
\frac{ (2\pi)^{k_i} }{ \Gamma( k_i/2 +1)^2 } \binom{k_i}{j_i} \int_{\R^+} h( T |\eta_i \kappa_i|^{-1} y) y^{k_i} | K_{k_i/2 - j_i}(4\pi y) |^2 dy^\times.
\end{equation}

Applying Mellin inversion using (\ref{bessel}), this becomes

\begin{equation*}
\ll \frac{ 1 }{ \Gamma( k_i/2 +1)^2 } \binom{k_i}{j_i} \int_{\sigma} H(-s) \left( \frac{T}{ 2\pi |\eta_i \kappa_i| } \right)^s \frac{ \Gamma( (s+k_i)/2)^2 }{ \Gamma(s+k_i) } \Gamma( s/2 + j_i) \Gamma( s/2 + k_i - j_i) ds.
\end{equation*}

Because $j \in J_0$, $j_i$ and $k_i - j_i$ are $ \ge k_i^{1/2}$ so we may apply (\ref{luosarnak}) and choose $\sigma = 2+\epsilon$ to obtain

\begin{equation}
\label{sxiint5}
(\ref{sxiint4}) \ll \frac{1}{ k_i (k_i - j_i) j_i } h\left( \frac{T \sqrt{ j_i(k_i-j_i) } }{ 2\pi |\eta_i \kappa_i| } \right)   + O \left( k_i^{-3/2 + \epsilon} ( T/ |\eta_i|)^{2+\epsilon} \right).
\end{equation}

Substituting the bounds (\ref{sxiint3}) and (\ref{sxiint5}), equation (\ref{sxi1}) becomes

\begin{multline*}
S_{\xi,j}(T) \ll \frac{ |a_\xi(T^{-1})| }{ L( 1, \sym^2 \pi ) } \sum_{\eta > 0} | \lambda_\pi(\eta) \lambda_\pi( \eta + \xi )|  \prod_{i \le r_1} \frac{ 1 }{ k_i } \left( h\left( \frac{T (k_i-1) }{ 4\pi |\eta_i \kappa_i| } \right) + O \left( k_i^\epsilon ( T/ |\eta_i|)^{1+\epsilon} \right) \right) \\
\times \prod_{i > r_1} \left( \frac{1}{ k_i (k_i - j_i) j_i } h\left( \frac{T \sqrt{ j_i(k_i-j_i) } }{ 2\pi |\eta_i \kappa_i| } \right)   + O \left( k_i^{-3/2 + \epsilon} ( T/ |\eta_i|)^{2+\epsilon} \right) \right).
\end{multline*}

As in the totally real case we may use the bound $| \lambda_\pi(\eta) | \ll N\eta^\epsilon$ to show that the contribution to the sum from all error terms is $O( |J|^{-1} \| k \|^{-\nu/2 + \epsilon} T^{n+\epsilon} )$, so our upper bound may be rewritten

\begin{multline*}
S_{\xi,j}(T) \ll \frac{ |a_\xi(T^{-1})| }{ L( 1, \sym^2 \pi ) } \Biggl( \sum_{\eta > 0} | \lambda_\pi(\eta) \lambda_\pi( \eta + \xi )|  \prod_{i \le r_1} \frac{ 1 }{ k_i } h\left( \frac{T (k_i-1) }{ 4\pi |\eta_i \kappa_i| } \right) \\
\times \prod_{i > r_1} \frac{1}{ k_i (k_i - j_i) j_i }  h\left( \frac{T \sqrt{ j_i(k_i-j_i) } }{ 2\pi |\eta_i \kappa_i| } \right) + O( |J|^{-1} \| k \|^{-\nu/2 + \epsilon} T^{n+\epsilon} ) \Biggr).
\end{multline*}

On summing over $j$ it can be seen that we have proven the inequalities (\ref{unfoldmixed}) and (\ref{unfoldj}), where the terms for $j \in J_1$ are absorbed into the error term.

\subsection{Extracting Shifted Convolution Sums: The Zero Shift}
\label{mixedshift2}

Having dealt with the `error' terms with $\xi \neq 0$, it remains to prove (\ref{eisintmixed}) and (\ref{Rboundmixed}) by considering the `main' term $S_0(T)$, which by (\ref{constantterm}) reduces to studying the integral $I_1(T)$ as in the totally real case.  Squaring out $|F_k(z)|^2$ and integrating in $x$, we obtain

\begin{equation*}
I_1(T) = 2^{-r_2} \sqrt{ |D| } \sum_{\eta>0} |a_f(\eta)|^2 \int_{\R_+^r} g(Ty) | {\bf K}_k(\eta \kappa y)|^2 Ny^{-1} dy^\times.
\end{equation*}

Applying the normalisation of $a_f(\eta)$, we have

\begin{equation*}
I_1(T) = 2^{-r_2} \sqrt{ |D| } |a_f(1)|^2 N\kappa \sum_{\eta>0} |\lambda_\pi(\eta)|^2  \int_{\R_+^r} g(Ty) N(\eta \kappa y)^{-1} | {\bf K}_k(\eta \kappa y)|^2  dy^\times.
\end{equation*}

As with the non-zero shifts, we may expand this into a sum over the multi-indices $j \in J$, and denote the $j$th term by $I_{1,j}(T)$.  If we define the symmetrised functions $\widetilde{g}$ and $\widetilde{\psi}_j$ by

\begin{equation*}
\widetilde{g}(y) = \sum_{u \in \OO^+_\times} g(uy), \quad \widetilde{\psi}_j(y) = Ny^{-1} \sum_{u \in \OO^+_\times} | K_{k,j}(uy) |^2,
\end{equation*}

then $I_{1,j}(T)$ may be expressed as

\begin{eqnarray*}
I_{1,j}(T) & = & 2^{-r_2} \omega_+ \sqrt{ |D| } |a_f(1)|^2 N\kappa \sum_{ (\eta) >0} |\lambda_\pi(\eta)|^2  \int_{\R_+^r} g(Ty) \widetilde{\psi}_j (\eta \kappa y)  dy^\times \\
 & = & 2^{-r_2} \omega_+ \sqrt{ |D| } |a_f(1)|^2 N\kappa \sum_{ (\eta) >0} |\lambda_\pi(\eta)|^2  \int_{\R_+^r / \OO^+_\times } \widetilde{g}(Ty) \widetilde{\psi}_j (\eta \kappa y)  dy^\times.
\end{eqnarray*}

Note that the factor of $\omega_+$ arises because the quotient of $\OO^+$ by $\OO_+^\times$ contains each ideal with this multiplicity.  If we let $G(s,m)$ be the Mellin transform of $\widetilde{g}$ as in (\ref{cuspmellin}) then Mellin inversion gives

\begin{multline*}
\int_{\R_+^r / \OO^+_\times } \widetilde{g}(Ty) \widetilde{\psi}_j (\eta \kappa y)  dy^\times = \\
 \frac{1}{2\pi i V_c} \sum_m \int_{(\sigma)} G(-s,-m) \left( \frac{T}{4\pi} \right)^{ns} N(\eta\kappa)^{-s} \lambda_{-m}(\eta \kappa) \Gamma(k,j,s,m) ds,
\end{multline*}

where $\Gamma(k,j,s,m)$ is the Mellin transform of $\widetilde{\psi}_j$ and is given by

\begin{multline*}
\Gamma(k,j,s,m) = \prod_{i \le r_1} ( 4\pi )^{-k_i + 1} \Gamma( s + \beta(m,i) + k_i -1) \\
 \times \prod_{i > r_1} (2\pi)^{-k_i} \binom{k_i}{j_i} 2^{2s + \beta(m,i)} \frac{ \Gamma( s + ( \beta(m,i) + k_i )/2 )^2 }{ 8 \Gamma( 2s + \beta(m,i) + k_i ) } \\
 \Gamma( s + \beta(m,i)/2 + j_i) \Gamma( s + \beta(m,i)/2 + k_i - j_i).
\end{multline*}

Substituting into $I_{1,j}(T)$ and forming the $L$-function from the sum over $\eta$, we have

\begin{multline*}
I_{1,j}(T) = \frac{ \omega_+ \sqrt{|D|} }{ 2^{r_2} V_c} |a_f(1)|^2 \frac{1}{2\pi i} \sum_m \int_{(\sigma)} G(-s,-m) \left( \frac{T}{4\pi} \right)^{ns} N\kappa^{1-s} \lambda_{-m}(\kappa) \\
L( s, \sym^2 \pi \otimes \lambda_{-m} ) \frac{ L( s, \lambda_{-m} ) }{ L( 2s, \lambda_{-2m} ) } \Gamma(k,j,s,m) ds.
\end{multline*}

We now substitute the value of $|a_f(1)|^2$ and shift the line of integration to $\sigma = 1/2$, giving

\begin{equation}
\label{main1}
I_{1,j}(T) = \frac{C_g T^n}{ |J| } + E_{1/2,j}(T)
\end{equation}

with

\begin{multline*}
E_{1/2,j}(T) \ll \frac{1}{ L( 1, \sym^2 \pi ) } \prod_{i \le r_1} \frac{ (4\pi)^{k_i} }{ \Gamma(k_i) } \prod_{i > r_1} \frac{ (2\pi)^{k_i} }{ \Gamma(k_i/2+1)^2 } \binom{k_i}{j_i} \sum_m \int_{(1/2)} G(-s,-m) \left( \frac{T}{4\pi} \right)^{ns} \\
L( s, \sym^2 \pi \otimes \lambda_{-m} ) \frac{ L( s, \lambda_{-m} ) }{ L( 2s, \lambda_{-2m} ) } \Gamma(k,j,s,m) ds.
\end{multline*}

By Stirling's formula and the rapid decay of $G(s,m)$ this error may be bounded above by

\begin{multline}
\label{Emixed}
E_{1/2,j}(T) \ll \left( \frac{ T^n }{ Nk } \right)^{1/2} \frac{1}{ L( 1, \sym^2 \pi ) } \prod_{i > r_1} \frac{ 1}{ \sqrt{ j_i ( k_i - j_i) } } \\
\sum_m \int_{-\infty}^{+\infty} \frac{ | L( 1/2 + it, \sym^2 \pi \otimes \lambda_{-m} ) | }{ ( |t| + \| m \| + 1 )^A } dt
\end{multline}

for any $A > 0$.  Because $x^{1/2}$ is integrable at $0$,

\begin{equation*}
\sum_{j \in J} \prod_{i > r_1} \frac{ 1}{ \sqrt{ j_i ( k_i - j_i) } }
\end{equation*}

is bounded independently of $k$ so that when we sum (\ref{main1}) and (\ref{Emixed}) over $j$ we obtain

\begin{equation*}
I_1(T) = cT^n + O(T^{n/2} R_k(f)),
\end{equation*}

with

\begin{equation*}
R_k(f) = \frac{1}{ \sqrt{Nk} L( 1, \sym^2 \pi ) } \sum_m \int_{-\infty}^{+\infty} \frac{ | L( 1/2 + it, \sym^2 \pi \otimes \lambda_{-m} ) | }{ ( |t| + \| m \| + 1 )^A } dt.
\end{equation*}

This completes the proof of proposition \ref{Holomixed}.

\section{Application of the Large Sieve}
\label{holosieve}

In this section we complete the proof of theorem \ref{Home} by establishing the bounds of proposition \ref{shiftbound} for the shifted sums

\begin{equation*}
C_{\xi}(x) = \sum_{\eta \le x} | \lambda_1(\eta) \lambda_2(\eta+\xi)|,
\end{equation*}

where $\lambda_i$ are multiplicative functions on $\OO^+$ satisfying $|\lambda_i(\eta)| \le \tau_m(\eta)$ for some $m$ and $x = (x_i)$ satisfies $x_i \ge \|x\|^\nu$ for some $\nu > 0$.  We first rearrange and partition the sums into pieces which may be treated either by elementary methods or by a large sieve.  We assume that $0 < \| \xi \| \le  \| x \|^\nu$, and given $\epsilon > 0$ we will be working throughout with a choice of variables satisfying

\begin{eqnarray}
\label{sievez}
z & = & \| x \|^{1/s} \quad \text{with } s = \epsilon \log \log x, \\
y & = & \| x \|^\epsilon.
\end{eqnarray}

We factorise the ideals $(\eta)$ and $(\eta + \xi)$ as

\begin{equation*}
(\eta) = \Aa \Bb \quad \text{and} \quad (\eta + \xi) = \Aa_\xi \Bb_\xi
\end{equation*}

in such a way that for every prime ideal $\p$ dividing $\eta(\eta+\xi)$,

\begin{equation*}
\p | \Aa \Aa_\xi \Rightarrow N\p \le z \quad \text{and} \quad \p | \Bb \Bb_\xi \Rightarrow N\p > z,
\end{equation*}

and partition the sum $C_\xi(x)$ into parts depending on the norm of $\Aa$ and $\Aa_\xi$.  We denote by $C^y(x)$ the part of $C_\xi(x)$ in which either $N \Aa$ or $N \Aa_\xi$ is greater than $y$,

\begin{equation*}
C^y(x) = \sum_{\substack{ \eta \le x \\ N\Aa > y }} | \lambda_1(\eta) \lambda_2(\eta+\xi)| + \sum_{\substack{ \eta \le x \\ N\Aa_\xi > y }} | \lambda_1(\eta) \lambda_2(\eta+\xi)|,
\end{equation*}

and the part where both $N \Aa$ and $N \Aa_\xi$ are less than or equal to $y$ we denote by $C_y(x)$,

\begin{equation*}
C_y(x) = \sum_{\substack{ \eta \le x \\ N\Aa, N\Aa_\xi \le y }} | \lambda_1(\eta) \lambda_2(\eta+\xi)|,
\end{equation*}

so that $C_\xi(x) = C^y(x) + C_y(x)$.

\subsection{Treating $C^y(x)$ by Elementary Methods}

We first handle the terms with $N\Aa$ or $N \Aa_\xi$ large.  We begin by applying H\"older's inequality and $|\lambda_i(\eta)| \le \tau_m(\eta)$ to get

\begin{equation*}
C^y(x) \ll \left( \sum_{\substack{ \eta \le x \\ N\Aa > y }} 1 \right)^{1/2} \left( \sum_{\eta \le x + \| \xi \| } \tau^4_m(\eta) \right)^{1/2}.
\end{equation*} 

We know that $x + \| \xi \| \le 2x$ by our assumption on $\xi$, and have the bound

\begin{equation*}
\sum_{\eta \le 2x} \tau^4_m(\eta) \ll Nx (\log \| x \| )^A
\end{equation*}

for some $A$.  As all prime factors of $N \Aa$ must be at most $z$, we may use a Rankin's method argument (\cite{MVa}, Thm. 7.6) to bound the number of allowable values of $N \Aa$ up to $t$ by $t(\log t)^{-A}$ for all $A$.  Combined with a bound of $\ll (\log t)^n$ for the number of $\Aa$ with norm $t$, we see that the number of choices for $\Aa$ with $N \Aa \le t$ is $\ll t(\log t)^{-A}$.  Partial summation and our choice of $x, y$ and $z$ then gives the bound

\begin{equation*}
\sum_{\substack{ (\eta): N\eta \le Nx \\ N\Aa > y }} 1 \ll \frac{Nx}{ (\log \| x \| )^A }
\end{equation*}

for any $A$, and the upper bound of $(\log \| x \|)^{n-1}$ for the number of $\eta \le x$ generating a given $(\eta)$ lets us conclude

\begin{equation}
\label{C^y}
C^y(x) \le \frac{Nx}{ (\log \| x \| )^2}.
\end{equation}

\subsection{Treating $C_y(x)$ by the Large Sieve}

From our definition of $C_y(x)$, we are left with evaluating

\begin{equation*}
C_y(x) \ll \sum_{\substack{N \Aa, N \Aa_\xi \le y \\ \p | \Aa \Aa_\xi \Rightarrow N\p \le z}} | \lambda_1( \Aa) \lambda_2( \Aa_\xi)| \sum_{\substack{\eta \le x \\ \eta \equiv 0 \; (\Aa) \\ \eta \equiv -\xi \; (\Aa_\xi) \\ \p | \Bb \Bb_\xi \Rightarrow N\p > z}} | \lambda_1( \Bb) \lambda_2( \Bb_\xi)|.
\end{equation*}

To help deal with certain co-primality conditions which come up during our analysis, we pull out the greatest common divisor $\mathfrak{v}$ of $\Aa$ and $\Aa_\xi$, which we choose to have a normalised positive generator $v$.  Writing $\eta_v = \eta/v$ and $\eta_v + w= (\eta + \xi)/v$, we again factorise $(\eta_v)$ and $(\eta_v+w)$ with $(\Aa,\Aa_\xi) = (\Aa \Aa_\xi, w) = \OO$, so that

\begin{equation}
\label{C_y}
C_y(x) \ll \sum_{\substack{ vw = \xi \\ v \text{ normalised} }} \sum_{\substack{N \Aa, N \Aa_\xi \le y / N \mathfrak{v} \\ \p | \Aa \Aa_\xi \Rightarrow N\p \le z \\ (\Aa,\Aa_\xi) = (\Aa \Aa_\xi, w) = \OO }} | \lambda_1( \mathfrak{v} \Aa) \lambda_2( \mathfrak{v} \Aa_\xi)| \sum_{\substack{\eta_v \ll x/ \| v \| \\ \eta_v \equiv r \; (\Aa \Aa_\xi) \\ \p | \Bb \Bb_\xi \Rightarrow N\p > z}} | \lambda_1( \Bb) \lambda_2( \Bb_\xi)|.
\end{equation}

Here we applied the Chinese remainder theorem so that the residue class $r$ in the innermost sum satisfies $r \equiv 0 \; (a)$ and $r \equiv -w \; (a_\xi)$.  The Ramanujan-Petersson conjecture and our choice of $s$ in (\ref{sievez}) imply that $| \lambda_1( \Bb) \lambda_2( \Bb_\xi)| \ll (\log \| x \|)^{2m \epsilon}$, as we have $|\lambda_1(\p^\alpha)| \le \tau_m(\p^\alpha) \le 2^{\alpha + m -1}$ and $\Bb = \p_1^{\alpha_1} \ldots \p_t^{\alpha_t}$ with $\alpha_1 + \ldots +\alpha_t \le s$.  We may therefore substitute this, and proceed to bound the count

\begin{equation}
\label{sieve1}
\sum_{\substack{\eta_v \ll x/ \| v \| \\ \eta_v \equiv r \; (\Aa \Aa_\xi) \\ \p | \Bb \Bb_\xi \Rightarrow N\p > z}} 1.
\end{equation}

Choose normalised generators $a$ and $a_\xi$ for $\Aa$ and $\Aa_\xi$, which will satisfy $\| a \|, \| a_\xi \| \ll y^{1/n} / \| v \|$.  Writing $\eta_v = a a_\xi m + r$ with $r$ chosen in a negative fundamental domain for $\F / (a a_\xi)$ (so that $a a_\xi m \ge 0$), we note the following equivalences between divisibility conditions for primes with $N\p \le z$ :

\begin{eqnarray*}
\p \! \not| \, b \: \, & \Longleftrightarrow & \p \! \not| \, (a_\xi m + r/a), \\
\p \! \not| \, b_\xi & \Longleftrightarrow & \p \! \not| \, (a m + (r+w)/a_\xi).
\end{eqnarray*}

For fixed normalised $a$ and $a_\xi$ satisfyuing $(a, a_\xi) = (a a_\xi, w) = \OO$ and $\| a \|, \| a_\xi \| \ll y^{1/n} / \| v \|$, we see that the count in (\ref{sieve1}) is bounded by $S = | \Ss( \Mm, \Pp, \Omega ) |$ where we define $\Ss( \Mm, \Pp, \Omega )$ to be the `sifted set,'

\begin{equation*}
\Ss( \Mm, \Pp, \Omega ) = \{ m \in \Mm \; | \; m \; (\text{mod } \p) \notin \Omega_\p \text{ for all } \p \in \Pp \}.
\end{equation*}

Here,

\begin{eqnarray*}
\Mm & = & \{ m \in \OO \; | \; 0 < m v a a_\xi \ll x \}, \\
\Pp & = & \{ \p \; | \; 2 < N\p \le z \},
\end{eqnarray*}

and the set $\Omega = \bigcup_{\p \in \Pp} \Omega_\p$ of residue classes to be `sieved out' is given by

\begin{equation*}
\Omega_\p = \Biggl\{ \begin{array}{ll} \{ r_1 \; (\text{mod } \p) \} & \text{for } \p | a \\
\{ r_2 \; (\text{mod } \p) \} & \text{for } \p | a_\xi \\
\{ r_1, r_2 \; (\text{mod } \p) \} & \text{for } \p \not| aa_\xi, \\
\end{array} 
\end{equation*}

where $r_1 \equiv -\overline{a_\xi} r /a \; (\p)$ and $r_2 \equiv -\overline{a} (r+w) /a_\xi \; (\p)$.  Here the overline means multiplicative inverse mod $\p$.

We now apply a variant of the standard large sieve for the lattice $\Z^n$.  Let $d = (d_i)$ with $d_i > \| d \|^\nu$ for some $\nu > 0$, let $B(d)$ be the box with dimensions $d$ centred at the origin in $\R^n$, and $D(d)$ be the image of $B(d)$ under any rotation.  If $\Pp$ is a set of rational primes, define $\Omega_p$ to be a subset of $L / pL$ of cardinality $\omega(p)$ for each $p \in \Pp$, and define a sifted set $\Ss( \mathcal{L}, \Pp, \Omega )$ by

\begin{eqnarray*}
\Ss( \mathcal{L}, \Pp, \Omega ) & = & \{ m \in \mathcal{L} ; m \; (\text{mod } p) \notin \Omega_\p \text{ for all } p \in \Pp \}, \\
\text{with } \mathcal{L} & = & \Z^n \cap D(d).
\end{eqnarray*}

We than have

\begin{equation*}
| \Ss( \mathcal{L}, \Pp, \Omega ) | \ll_\nu \frac{ Nd + Q^{2n} }{ H }
\end{equation*}

for any $\| d \|^{\nu/2} \ge Q \ge 1$, where

\begin{equation*}
H = \sum_{q \le Q} h(q)
\end{equation*}

and $h(q)$ is the multiplicative function supported on squarefree integers with prime divisors in $\Pp$ such that

\begin{equation*}
h(p) = \frac{ \omega(p) }{ p^n - \omega(p) }.
\end{equation*}

This form of the large sieve may be proven using soft techniques of Poisson summation, described in chapter 7 of \cite{IK}.  To apply this in the number field, identify $\OO$ with $\Z^n$ and for each $p$, construct a set $\Omega_p$ from the $\Omega_\p$ with $\p | p$ using the Chinese remainder theorem.  We then have

\begin{equation*}
\omega(p) \ge (\alpha_p + \beta_p) p^{n-1} + O(p^{n-2}),
\end{equation*}

where $\alpha_p$ is the number of degree 1 primes above $p$ and $\beta_p$ is the number which do not divide $a a_\xi$.  We then have the lower bound

\begin{equation*}
H \gg (\log z)^2 \prod_{\p | a a_\xi} \left( 1 - \frac{1}{N \p} \right),
\end{equation*}

so that the count (\ref{sieve1}) is bounded by

\begin{equation*}
\ll \frac{ Nx}{ (\log z)^2 N(v a a_\xi) } \prod_{\p | a a_\xi} \left( 1 - \frac{1}{N \p} \right)^{-1}.
\end{equation*}

Plugging this back into (\ref{C_y}), we obtain

\begin{equation*}
C^y(x) \ll \frac{ ( \log \|x\| )^{2m \epsilon} Nx }{ (\log z)^2 } \sum_{\substack{ vw = \xi \\ v \text{ normalised} }} \sum_{\substack{N \Aa, N \Aa_\xi \le y / N \mathfrak{v} \\ \p | \Aa \Aa_\xi \Rightarrow N\p \le z \\ (\Aa,\Aa_\xi) = (\Aa \Aa_\xi, w) = \OO }} \frac{ | \lambda_1( \mathfrak{v} \Aa) \lambda_2( \mathfrak{v} \Aa_\xi)| }{N(v a a_\xi) } \prod_{\p | a a_\xi} \left( 1 - \frac{1}{N \p} \right)^{-1}.
\end{equation*}

For each $v$, we may bound the inner sum from above by an Euler product.  If $\p \! \not| \, v$, the corresponding term is

\begin{equation}
\label{Euler1}
1 + \frac{| \lambda_1(\p)| + | \lambda_2(\p)|}{N\p} + O( N\p^{-2+\epsilon} )
\end{equation}

by our bounds on $|\lambda_i(\p)|$, and if $\p | v$ it is

\begin{equation}
\label{Euler2}
\frac{ | \lambda_1(\p) \lambda_2(\p) | }{ N\p } + O( N\p^{-2+\epsilon} ).
\end{equation}

(\ref{Euler2}) is at most 1 for almost all $\p$, and so for any $v$ we may bound the inner sum by

\begin{equation*}
\ll \prod_{N\p \le z} \left( 1 + \frac{| \lambda_1(\p)| + | \lambda_2(\p)|}{N\p} \right).
\end{equation*}

This gives the bound

\begin{equation*}
C^y(x) \ll \frac{ \tau(\xi) Nx}{ (\log \| x \|)^{2-\epsilon} } \prod_{N\p \le z} \left( 1 + \frac{| \lambda_1(\p)| + | \lambda_2(\p)|}{N\p} \right)
\end{equation*}

for $C^y(x)$, and when combined with our partition of $C_\xi(x)$ and the bound (\ref{C^y}) this concludes the proof of proposition \ref{shiftbound}.

\section{Proof of Theorem \ref{Some} }
\label{someproof}

In this section we shall prove theorem \ref{Some} by extending Soundararajan's approach of weak subconvexity to a number field.  We prove the necessary triple product identities in section \ref{weighttripprod}, before showing that the triple product $L$ functions which appear satisfy the hypotheses of Soundararajan's theorem in section \ref{weightsubcon}.

\subsection{Triple Products}
\label{weighttripprod}

Throughout this section, $C$ will denote a constant depending only on $F$ which may vary from equation to equation.  We shall also let $\sigma$ denote the conjugate linear automorphism of $\pi$ corresponding to complex conjugation on $X$, which has the property that $\langle \sigma(u), \sigma(v) \rangle = \overline{ \langle u, v \rangle}$.  We begin with the following triple product identity in the case of $\phi$ a Hecke-Maass cusp form.

\begin{prop}
Let $\phi$ be a Hecke-Maass cusp form with associated automorphic representation $\pi'$.  Then

\begin{eqnarray}
\label{weightprod1}
| \langle \phi F_k, F_k \rangle |^2 & = & C \frac{ \Lambda( \tfrac{1}{2}, \pi \otimes \pi \otimes \pi' ) }{ \Lambda( 1, \sym^2 \pi )^2 \Lambda(1, \sym^2 \pi' ) } \\
\label{weightprod2}
& \sim_\phi & Nk^{-1} \frac{ L( \tfrac{1}{2}, \sym^2 \pi \otimes \pi' ) }{ L( 1, \sym^2 \pi )^2 }
\end{eqnarray}

where $\sim_\phi$ means that the ratio of the two quantities is bounded between two positive constants depending only on $\phi$.
\end{prop}

\begin{proof}
Because $|F_k|^2 dv$ is the pushforward of $|R_\pi(v_k)|^2 dx$, the inner product $\langle \phi F_k, F_k \rangle$ is equal to

\begin{equation*}
\int_X |R_\pi(v_k)|^2 \phi dx = \int_X R_\pi(v_k) R_\pi(\sigma(v_k)) \phi dx,
\end{equation*}

and we may evaluate the RHS of this expression using Ichino's formula.  Let $I = \otimes I_i$ and $I' = \otimes I_i'$ be the products of the Archimedean local factors of $\pi$ and $\pi'$, $k_i$ and $r_i'$ be the relevant parameters of these local factors, and $u \in I'$ be the unit spherical vector.  As all our vectors are unramified and our division algebra is split, the statement of Ichino's formula in this case is

\begin{multline}
\label{weightprod3}
\left| \int_X R_\pi(v_k) R_\pi(\sigma(v_k)) \phi dx \right|^2 = C \prod_{i=1}^r \int_{ \overline{G_i} } \langle I_i(g) v_{k_i}, v_{k_i} \rangle \langle I_i(g) v_{-k_i}, v_{-k_i} \rangle \langle I'_i(g) u_i, u_i \rangle d\overline{g_i}  \\
\frac{ L( \tfrac{1}{2}, \pi \otimes \pi \otimes \pi' ) }{ L( 1, \sym^2 \pi )^2 L(1, \sym^2 \pi' ) },
\end{multline}

where $v_k = \otimes v_{k_i}$ and $u = \otimes u_i$.  If $\nu_i$ is a complex place the $i$th local integral appearing in the product was computed in \cite{Ma2} to be

\begin{equation}
\label{cxgamma}
C \frac{ \Gamma \left( \tfrac{ 1 + k_i \pm ir'_i}{2} \right)^2 \Gamma \left( \tfrac{ 1 \pm ir'_i}{2} \right)^2 }{ \Gamma( 1 + \tfrac{k_i}{2} )^4 \Gamma(1 \pm ir'_i)^2 },
\end{equation}

and up to an absolute constant this is equal to the ratio of the Archimedean factors at the place $\nu_i$ of the $L$ functions appearing in (\ref{weightprod3}).  In the real case the local integral may be determined by comparison with Watson's formula, and is

\begin{equation}
\label{realgamma}
C \frac{ \Gamma_\R ( k_i - 1/2 \pm ir'_i ) \Gamma_\R ( k_i + 1/2 \pm ir'_i ) \Gamma_\R ( 1/2 \pm ir'_i ) \Gamma_\R ( 3/2 \pm ir'_i ) }{ \Gamma_\R ( k_i - 1/2 )^2 \Gamma_\R ( k_i + 1/2 )^2 \Gamma_\R ( 1/2 \pm 2i r'_i) }.
\end{equation}

This is again proportional to the relevant Archimedean factors of the $L$ functions appearing in (\ref{weightprod3}), which gives (\ref{weightprod1}).  Finally, the Archimedean factors (\ref{realgamma}) and (\ref{cxgamma}) have the asymptotic behaviours $k_i^{-1}$ and $k_i^{-2}$ as $k_i \rightarrow \infty$, which gives (\ref{weightprod2}).

\end{proof}

We now treat the inner products $\langle E(s,m, \cdot) F_k, F_k \rangle$ against spherical Eisenstein series by unfolding, to obtain the following formula.

\begin{prop}

\begin{eqnarray}
\label{weightprodeis1}
| \langle E(s, m, \cdot) F_k, F_k \rangle | & = & C \left| \frac{ \Lambda( 1/2 + it, \pi \otimes \pi \otimes \lambda_{-m} ) }{ \Lambda( 1, \sym^2 \pi ) \Lambda( 1 + 2it, \lambda_{-2m} ) } \right| \\
\label{weightprodeis3}
| \langle E(s, m, \cdot) F_k, F_k \rangle | & \ll & \frac{ (1 + |t| + \| m \| )^{n/4 + \epsilon} }{ Nk } \left| \frac{ L( 1/2 + it, \sym^2 \pi\otimes \lambda_{-m} ) }{ L( 1, \sym^2 \pi ) } \right|.
\end{eqnarray}

\end{prop}

\begin{proof}

For $\text{Re}(s) > 1$, by unfolding and substituting the Fourier expansion of $F_k$ we have

\begin{eqnarray*}
\langle E(s,m, \cdot) F_k, F_k \rangle & = & \int_{\Gamma_\infty \backslash \HH_F} Ny^s \lambda_m(y) |F_k(z) |^2 dv \\
& = & |a_f(1)|^2 \int_{\Gamma_\infty \backslash \HH_F} Ny^s \lambda_m(y) \sum_{\xi \in \OO} |\lambda_\pi(\xi) |^2 N\xi^{-1} |{\bf K}_k( \xi \kappa y) |^2 dv.
\end{eqnarray*}

$\Gamma_\infty \backslash \HH_F \simeq \F / \OO \mu_+ \times \R_+^r / \OO_+^\times$, where $\mu_+$ acts on $\F / \OO$ by multiplication, and the volume of $\F / \OO \mu_+$ is $2^{-r_2} \omega_+^{-1} \sqrt{ |D| }$.  We therefore have

\begin{eqnarray*}
\langle E(s,m, \cdot) F_k, F_k \rangle & = & |a_f(1)|^2 2^{-r_2} \omega_+^{-1} \sqrt{|D|} \int_{ \R_+^r / \OO_+^\times } Ny^s \lambda_m(y) \\
& \quad & \quad \sum_{\xi \in \OO} |\lambda_\pi(\xi) |^2 N\xi^{-1} |{\bf K}_k( \xi \kappa y) |^2 Ny^{-1} dy^\times.
\end{eqnarray*}

Making the change of variable $y \mapsto |\xi \kappa|^{-1} y$ and unfolding the integral over $\OO_+^\times$, this becomes

\begin{eqnarray*}
& = & |a_f(1)|^2 2^{-r_2} \sqrt{|D|} N\kappa \sum_{ (\xi) } |\lambda_\pi(\xi) |^2 (N \xi \kappa)^{-s} \lambda_{-m}(\xi \kappa) \\
& \quad & \quad \int_{ \R_+^r} Ny^s \lambda_m(y) |{\bf K}_k( y) |^2 Ny^{-1} dy^\times \\
& = & |a_f(1)|^2 2^{-r_2} \sqrt{|D|} N\kappa^{1-s} \lambda_{-m}(\kappa) \frac{ L( s, \pi \otimes \pi \otimes \lambda_{-m} ) }{ L( 2s, \lambda_{-2m} ) } \\
& \quad & \quad \int_{ \R_+^r} Ny^s \lambda_m(y) |{\bf K}_k( y) |^2 Ny^{-1} dy^\times.
\end{eqnarray*}

(Note that the factor of $\omega_+^{-1}$ vanished because $\OO / \OO_+^\times$ counts every ideal with this multiplicity.)  We factorise the integral ocurring here, and pair each factor with the corresponding term of $|a_f(1)|^2$ so that

\begin{equation}
\label{weightprodeis2}
\langle E(1/2 + it, m, \cdot) F_k, F_k \rangle = C \frac{ L( 1/2+it, \pi \otimes \pi \otimes \lambda_{-m} ) }{ L(1, \sym^2 \pi) L( 1+2it, \lambda_{-2m} ) } \prod_{i=1}^r \Tt_i,
\end{equation}

where for $i \le r_1$ we have

\begin{equation*}
\Tt_i = \frac{ (4\pi)^{k_i} }{ \Gamma(k_i) } \int_0^\infty y^{1/2 + it + \beta(m,i) } |{\bf K}_i(y) |^2 y^{-1} dy^\times,
\end{equation*}

and for $i > r_1$

\begin{equation*}
\Tt_i = \frac{ (2\pi)^{k_i} }{ \Gamma(k_i/2 + 1)^2 } \int_0^\infty y^{1 + 2it + \beta(m,i) } |{\bf K}_i(y) |^2 y^{-2} dy^\times.
\end{equation*}

The integral at real places may be easily calculated to be

\begin{equation*}
\Tt_i = (4\pi)^{1/2 -it -\beta(m,i) } \frac{ \Gamma( k_i - 1/2 + it + \beta(m,i) ) }{ \Gamma(k_i) },
\end{equation*}

and the integral at complex places was calculated in \cite{Ma2} to have absolute value

\begin{equation*}
\Tt_i = \frac{ \Gamma \left( 1/2 \pm (it + \beta(m,i) ) + k_i/2 \right) \Gamma \left( 1/2 \pm (it + \beta(m,i) ) \right) }{ \Gamma( 1 + \tfrac{k}{2} )^2 | \Gamma(1+ 2it + 2\beta(m,i) )| }.
\end{equation*}

Both of these terms agree in absolute value with the ratio of gamma factors at the corresponding infinite place of the $L$ functions appearing in (\ref{weightprodeis2}), which proves formula (\ref{weightprodeis1}).  To prove (\ref{weightprodeis3}), we use Stirling together with the bound $| \Gamma( \sigma + it)| \le \Gamma(\sigma)$ to show that $| \Tt_i | \ll k_i^{-1/2}$ for $\nu_i$ real and $\Tt_i \ll k_i^{-1} (1 + |t + \beta(m,i)| )^{-1/2} \le k_i^{-1}$ for $\nu_i$ complex.  This gives

\begin{eqnarray*}
| \langle E(1/2 + it, m, \cdot) F_k, F_k \rangle | & \ll &  \left| Nk^{-1/2} \frac{ L( 1/2+it, \pi \otimes \pi \otimes \lambda_{-m} ) }{ L(1, \sym^2 \pi) L( 1+2it, \lambda_{-2m} ) } \right| \\
& = & Nk^{-1/2} \left| \frac{ L( 1/2+it, \sym^2 \pi \otimes \lambda_{-m} ) L( 1/2 + it, \lambda_{-m} ) }{ L(1, \sym^2 \pi) L( 1+2it, \lambda_{-2m} ) } \right|,
\end{eqnarray*}

and applying the convex bound $L( 1/2 + it, \lambda_{-m} ) \ll (1 + |t| + \| m \| )^{n/4 + \epsilon}$ and the lower bound $L( 1+2it, \lambda_{-2m} ) \gg (1 + |t| + \| m \| )^{-\epsilon}$ yields (\ref{weightprodeis3}).

\end{proof}

\subsection{Weak Subconvexity}
\label{weightsubcon}

Having expressed the inner products $\langle \phi F_k, F_k \rangle$ for $\phi$ a Hecke-Maass cusp form or Eisenstein series in terms of $L$ values, we now prove theorem \ref{Some} by applying the weak subconvexity of Soundararajan \cite{So2} to these values.  This is a theorem which is valid for any Dirichlet series $L(s,\pi)$ over the rationals satisfying certain conditions, which we now describe.  The first of these is that $L(s,\pi)$ may be given by an Euler product

\begin{equation*}
L(s,\pi) = \sum_{n=1}^\infty \frac{ a_\pi(n) }{n^s} = \prod_p \prod_{j=1}^m \left( 1 - \frac{ \alpha_{j,\pi}(p) }{ p^s } \right)^{-1},
\end{equation*}

and that both the series and product are absolutely convergent for $\text{Re}(s) > 1$ (the notation $L(s,\pi)$ is meant to suggest that $\pi$ corresponds to an automorphic representation, although this is not assumed).  The second is that there is an Archimedean component

\begin{equation*}
L_\infty(s,\pi) = N^{s/2} \prod_{j=1}^m \Gamma_\R( s + \mu_j)
\end{equation*}

for $N \in \Z$ and $\mu_j \in \C$, such that the completed $L$ function $\Lambda(s,\pi) = L_\infty(s,\pi) L(s,\pi)$ has an analytic continuation to the entire complex plane.  Moreover, it should satisfy a functional equation

\begin{equation*}
\Lambda(s, \pi) = \kappa \Lambda(1-s, \widetilde{\pi} ),
\end{equation*}

for $\kappa$ a complex number of absolute value one and where

\begin{equation*}
L(s, \widetilde{\pi}) = \sum_{n=1}^\infty \frac{ \overline{a_\pi(n)} }{n^s}, \quad \text{and} \quad L_\infty(s, \widetilde{\pi}) = N^{s/2} \prod_{j=1}^m \Gamma_\R( s + \overline{\mu_j} ).
\end{equation*}

These conditions are quite general, and hold for all the $L$ functions appearing in our triple product identities.  In addition, we require some bounds towards the Ramanujan-Selberg conjectures for $\pi$, which predicts that $| \alpha_{j,\pi}(p) | \le 1$ and $\text{Re}(\mu_j) \ge 0$.  Write

\begin{equation*}
- \frac{ L'}{ L} (s,\pi) = \sum_{n=1}^\infty \frac{ \lambda_\pi(n) \Lambda(n) }{ n^s },
\end{equation*}

where $\lambda_\pi(n) = 0$ unless $n = p^k$ is a prime power, when it equals $\sum_{j=1}^m \alpha_{j,\pi}(p)^k$.  We require the existence of two constants $A_0, A \ge 1$ such that for all $x \ge 1$ the inequality

\begin{equation}
\label{weakramanujan}
\sum_{x < n < ex} \frac{ | \lambda_\pi(n) |^2 }{n} \Lambda(n) \le A^2 + \frac{A_0}{ \log ex}
\end{equation}

is satisfied; note that the Ramanujan conjecture would imply this with $A = m$ and $A_0 \ll m^2$.  The condition on the parameters $\mu_j$ is that $\text{Re}(\mu_j) \ge -1 + \delta_m$ for some $\delta_m > 0$ and all $j$.  If we define the analytic conductor of $\pi$ to be

\begin{equation*}
C(\pi) = N \prod_{j=1}^m (1 + |\mu_j| ),
\end{equation*}

Soundararajan proves the following.

\begin{theorem}
Under the assumptions on $L$ stated above,

\begin{equation*}
L(1/2, \pi) \ll \frac{ C(\pi)^{1/4} }{ ( \log C(\pi) )^{1-\epsilon} },
\end{equation*}

where the implied constant depends on $m, \epsilon, A_0, A$ and $\delta_m$.
\end{theorem}

We may prove theorem \ref{Some} by applying this result to the $L$ values $L(1/2, \sym^2 \pi \otimes \pi')$ and $L(1/2 + it, \sym^2 \pi \otimes \lambda_{-m} )$ appearing in equations (\ref{weightprod2}) and (\ref{weightprodeis3}), once we have established that the $L$ functions satisfy the necessary hypotheses.  While they are $L$ functions over $F$, they may be considered as being over $\Q$ by formal base change.  We begin with $L(s, \sym^2 \pi \otimes \lambda_{-m} )$; if $L(s, \pi)$ has the Euler product expansion

\begin{equation*}
L(s,\pi) = \prod_\p \left( 1 - \frac{ \alpha_\pi(\p) }{ N\p^s } \right)^{-1} \left( 1 - \frac{ \beta_\pi(\p) }{ N\p^s } \right)^{-1},
\end{equation*}

$L(s, \sym^2 \pi \otimes \lambda_{-m} )$ is given by the Euler product

\begin{equation*}
L(s, \sym^2 \pi \otimes \lambda_{-m} ) = \prod_\p \left( 1 - \frac{ \alpha_\pi(\p)^2 \lambda_{-m}(\p) }{ N\p^s } \right)^{-1} \left( 1 - \frac{ \lambda_{-m}(\p) }{ N\p^s } \right)^{-1} \left( 1 - \frac{ \beta_\pi(\p)^2 \lambda_{-m}(\p) }{ N\p^s } \right)^{-1}.
\end{equation*}

Recall our assumption that $|\alpha_\pi(\p)| = |\beta_\pi(\p)| = 1$.  The Archimedean factor of this function is

\begin{equation*}
L_\infty(s, \sym^2 \pi \otimes \lambda_{-m} ) = \prod_j L_{\infty, j}(s, \sym^2 \pi \otimes \lambda_{-m} ),
\end{equation*}

where

\begin{equation*}
L_{\infty,j}(s, \sym^2 \pi \otimes \lambda_{-m} ) = \Gamma_\R ( s + \beta(m,j) + 1) \Gamma_\R ( s + \beta(m,j) + k_j - 1) \Gamma_\R ( s + \beta(m,j) + k_j)
\end{equation*}

for $\nu_j$ real and

\begin{eqnarray*}
L_{\infty,j}(s, \sym^2 \pi \otimes \lambda_{-m} ) & = & \Gamma_\C ( s + k_j/2 + \beta(m,j) )^2 \Gamma_\C ( s + \beta(m,j) ) \\
& = & \Gamma_\R ( s + k_j/2 + \beta(m,j) )^2 \Gamma_\R ( s + k_j/2 + \beta(m,j) + 1 )^2 \\
& \quad & \quad \Gamma_\R ( s + \beta(m,j) ) \Gamma_\R ( s + \beta(m,j) + 1 )
\end{eqnarray*}

for $\nu_j$ complex.  In particular, it can be seen that all $\mu_j$ satisfy $\text{Re}(\mu_j) \ge 0$.  By work of Shimura, it is known that the completed $L$ function $\Lambda(s, \sym^2 \pi \otimes \lambda_{-m} )$ admits an analytic continuation to the whole complex plane and satisfies the functional equation

\begin{equation*}
\Lambda(s, \sym^2 \pi \otimes \lambda_{-m} ) = \Lambda(1-s, \sym^2 \pi \otimes \lambda_{m} ).
\end{equation*}

Therefore the $L$ function we consider satisfies all the hypotheses of Soundararajan's theorem as an Euler product over $F$, and it will continue to do so when considered as a product over $\Q$ - in particular, it will continue to satisfy the Ramanujan bound.  To apply the theorem to the value $L(1/2 + it, \sym^2 \pi \otimes \lambda_{-m} )$ we replace $L(s, \sym^2 \pi \otimes \lambda_{-m} )$ with the shifted function $L(s + it, \sym^2 \pi \otimes \lambda_{-m} )$, which still satisfies all the hypotheses and whose analytic conductor is now $\ll Nk^2 ( 1 + |t| + \| m \| )^{3n}$, to obtain

\begin{equation*}
L(1/2 + it, \sym^2 \pi \otimes \lambda_{-m} ) \ll \frac{ Nk^{1/2} ( 1 + |t| + \| m \| )^{3n/4} }{ ( \log Nk )^{1-\epsilon} }.
\end{equation*}

Turning now to $L(1/2, \sym^2 \pi \otimes \pi')$, let $L(s, \pi')$ have the Euler product

\begin{equation*}
L(s,\pi') = \prod_\p \left( 1 - \frac{ \alpha'_\pi(\p) }{ N\p^s } \right)^{-1} \left( 1 - \frac{ \beta'_\pi(\p) }{ N\p^s } \right)^{-1}
\end{equation*}

so that $L(s, \sym^2 \pi \otimes \pi')$ has the product expansion

\begin{multline*}
L(s, \sym^2 \pi \otimes \pi' ) = \prod_\p \left( 1 - \frac{ \alpha_\pi(\p)^2 \alpha_\pi'(\p) }{ N\p^s } \right)^{-1} \left( 1 - \frac{ \alpha_\pi'(\p) }{ N\p^s } \right)^{-1} \left( 1 - \frac{ \beta_\pi(\p)^2 \alpha_\pi'(\p) }{ N\p^s } \right)^{-1} \\
\left( 1 - \frac{ \alpha_\pi(\p)^2 \beta_\pi'(\p) }{ N\p^s } \right)^{-1} \left( 1 - \frac{ \beta_\pi'(\p) }{ N\p^s } \right)^{-1} \left( 1 - \frac{ \beta_\pi(\p)^2 \beta_\pi'(\p) }{ N\p^s } \right)^{-1}.
\end{multline*}

This $L$ function does not necessarily satisfy the Ramanujan bound because we are not assuming it for the representation $\pi'$, however because $\pi'$ is fixed the weaker estimate (\ref{weakramanujan}) will still hold by Rankin-Selberg theory applied to $\pi'$.  The Archimedean factor $L_{\infty,j}(s, \sym^2 \pi \otimes \pi' )$ at a real place is

\begin{equation*}
L_{\infty,j}(s, \sym^2 \pi \otimes \pi' ) = \Gamma_\R ( s + k_j - 1 \pm ir_j') \Gamma_\R ( s + k_j \pm ir_j') \Gamma_\R ( s \pm ir_j') \Gamma_\R ( s + 1 \pm ir_j'),
\end{equation*}

and at a complex place is

\begin{eqnarray*}
L_{\infty,j}(s, \sym^2 \pi \otimes \pi' ) & = & \Gamma_\C ( s + k_j/2 \pm ir_j'/2 )^2 \Gamma_\C ( s \pm ir_j'/2 ) \\
& = & \Gamma_\R ( s + k_j/2 \pm ir_j'/2 )^2 \Gamma_\R ( s + k_j/2 +1 \pm ir_j'/2 )^2 \\
& \quad & \quad \Gamma_\R ( s \pm ir_j'/2 ) \Gamma_\R ( s + 1 \pm ir_j'/2 ).
\end{eqnarray*}

The required bound $\text{Re}(\mu_j) \ge -1 + \delta$ now follows from the trivial bounds $\text{Im}(r_j') \le 1/2$ for $\nu_j$ real and $\text{Im}(r_j') \le 1$ for $\nu_j$ complex.  It is known by the work of Garrett \cite{Ga} that the completed $L$ function is entire in $\C$, and its value at $s$ is equal to its value at $1-s$.  It only remains to show that $L(s, \sym^2 \pi \otimes \pi' )$ satisfies the weak Ramanujan bound (\ref{weakramanujan}) as a Dirichlet series over $\Q$.  We have

\begin{eqnarray*}
- \frac{L'}{L} (s, \sym^2 \pi \otimes \pi' ) = \sum_\p \log N\p \sum_{n=1}^\infty \frac{ ( \alpha_\pi'^n(\p) + \beta_\pi'^n(\p) ) ( \alpha_\pi^{2n}(\p) + 1 + \beta_\pi^{2n}(\p) ) }{ N\p^{ns} },
\end{eqnarray*}

and $L$ satisfying the weak Ramanujan bound as a Dirichlet series over $\Q$ is equivalent to the bound

\begin{eqnarray*}
\sum_{\substack{ \p, n \\ x < N\p^n \le ex } } \log N\p \frac{ | ( \alpha_\pi'^n(\p) + \beta_\pi'^n(\p) ) ( \alpha_\pi^{2n}(\p) + 1 + \beta_\pi^{2n}(\p) ) |^2 }{ N\p^n } \le A^2 + \frac{ A_0 }{ \log( ex ) }.
\end{eqnarray*}

Applying the Ramanujan bound $|\alpha_\pi(\p)| = |\beta_\pi(\p)| = 1$, we only need to show that

\begin{eqnarray*}
\sum_{\substack{ \p, n \\ x < N\p^n \le ex } } \log N\p \frac{ | \alpha_\pi'^n(\p) + \beta_\pi'^n(\p) |^2 }{ N\p^n } \le A^2 + \frac{ A_0 }{ \log( ex ) }
\end{eqnarray*}

for all $x \ge 1$, where $A$ and $A_0$ are constants which are allowed to depend on $\pi'$.  This follows from Rankin-Selberg theory for $L(s, \pi' \times \widetilde{\pi}')$, whose logarithmic derivative is

\begin{equation*}
- \frac{L'}{L} (s, \pi' \times \widetilde{\pi}' ) = \sum_\p \sum_{n=1}^\infty \log N\p \frac{ | \alpha_\pi'^n(\p) + \beta_\pi'^n(\p) |^2 }{ N\p^{ns} }.
\end{equation*}

Because $L(s, \pi' \times \widetilde{\pi}')$ has a classical zero-free region $\text{Re}(s) \ge 1 - c_\pi' / \log( 1 + |t| )$, it follows in the same way as the proof of the prime number theorem that

\begin{equation*}
\sum_{\substack{ \p, n \\ x < N\p^n \le ex } } \frac{ \log N\p | \alpha_\pi'^n(\p) + \beta_\pi'^n(\p) |^2 }{ N\p^{ns} } = 1 + O_{\pi'} \left( \frac{1}{ \log(ex) } \right),
\end{equation*}

from which (\ref{weakramanujan}) follows.  As $L(s, \sym^2 \pi \otimes \pi')$ has analytic conductor $\ll Nk^4$, we may now apply the weak subconvex estimate to $L(1/2, \sym^2 \pi \otimes \pi')$ to complete the proof of theorem \ref{Some}.

\section{Conclusion of Proof}
\label{mixedconclusion}

We now conclude the proof of theorem \ref{weightmain}, by presenting the way in which theorems \ref{Some} and \ref{Home} may be combined as in Holowinsky and Soundararajan's paper \cite{HS}.  This relies on a lower bound for $L(1, \sym^2 \pi)$ and a relation between this value and the quantity $M_k(\pi)$ appearing in theorem \ref{Home}.  We first consider the symmetric square $L$ function $L(s, \sym^2 \pi)$, whose definition and basic analytic properites were given in section \ref{weightsubcon}, and collect some important results on it due to work of Gelbart and Jacquet \cite{GJ}, Hoffstein and Lockhart \cite{HL}, and Goldfeld, Hoffstein and Lockhart \cite{GHL}.  The lower bound we shall use for $L(1, \sym^2 \pi)$ is proved using the symmetric square lift of Gelbart and Jacquet \cite{GJ} from $GL(2)$ to $GL(3)$, which shows that $L(s, \sym^2 \pi)$ is the standard $L$ function of a cuspidal automorphic form on $GL(3)$.  Using the Rankin-Selberg convolution for this form, one may then establish a standard zero-free region for $L(s, \sym^2 \pi)$.  For instance, using Theorem 5.42 (or Theorem 5.44) of Iwaniec and Kowalski \cite{IK} one may show that for some constant $c > 0$ the region

\begin{equation*}
\mathcal{R} = \left\{ s = \sigma + it : \sigma \ge 1 - \frac{c}{ \log \| k \| ( 1 + |t|) } \right\}
\end{equation*}

contains no zero of $L(s, \sym^2 \pi)$ other than possibly a simple real zero.  This exceptional zero is ruled out by work of Hoffstein and Lockhart \cite{HL} (see the appendix by Goldfeld, Hoffstein and Lockhart \cite{GHL}), who show that there is an effectively computable choice of $c > 0$ such that $\mathcal{R}$ is totally zero free.  Furthermore, Goldfeld, Hoffstein and Lockhart \cite{GHL} show that

\begin{equation}
\label{sym2lower}
L(1, \sym^2 \pi) \gg \frac{1}{\log \| k \| }.
\end{equation}

The first consequence of this lower bound is the following.

\begin{lemma}
\label{Rkbound}
For any $t \in \R$ and $m \in \Z^{n-1}$, we have

\begin{equation*}
|L(s, \sym^2 \pi \otimes \lambda_{m})| \ll \frac{ Nk^{1/2} ( 1 + |t| + \| m \| )^{3n/4} }{ (\log \| k \| )^{1-\epsilon} }.
\end{equation*}

Therefore the quantity $R_k(\pi)$ appearing in theorem \ref{Home} satisfies

\begin{equation*}
R_k(\pi) \ll \frac{ (\log \| k \|)^\epsilon }{ (\log \| k \|) L(1, \sym^2 \pi) } \ll ( \log \| k \|)^\epsilon.
\end{equation*}
\end{lemma}

\begin{proof}
The first inequality follows from weak subconvexity, and is proven in section \ref{weightsubcon}.  The second bound follows immediately by substituting the first in the formula for $R_k(\pi)$ and applying the lower bound (\ref{sym2lower}) for $L(1, \sym^2 \pi)$.
\end{proof}

The relationship between $M_k(\pi)$ and $L(1,\sym^2 \pi)$ we shall use is based on the following lemma.

\begin{lemma}
\label{sym2lower1}
We have

\begin{equation*}
L(1, \sym^2 \pi) \gg ( \log \log \| k \| )^{-3} \exp \left( \sum_{N\p \le \| k \| } \frac{ \lambda_\pi( \p^2 ) }{ N\p } \right).
\end{equation*}

\end{lemma}

The proof of this over $\Q$ in \cite{HS} may be extended to a number field; the only modification is generalising the asymptotic $\sum_{p \le x} 1/p = \log \log x + O(1)$ to $\sum_{N\p \le x} 1/N\p= \log \log x + O(1)$.  Lemma \ref{sym2lower1} gives us the required estimate for $M_k(\pi)$ below.

\begin{lemma}
\label{Mkbound}
We have

\begin{equation*}
M_k(\pi) \ll (\log \| k \| )^{1/6} ( \log \log \| k \| )^{9/2} L(1, \sym^2 \pi)^{1/2}.
\end{equation*}

\end{lemma}

\begin{proof}
From the inequality $2 |x| \le \tfrac{2}{3} + \tfrac{3}{2} x^2$ and the Hecke relations, we obtain

\begin{equation*}
2 \sum_{N\p \le \|k \| } \frac{ |\lambda_\pi(\p)| }{ N\p} \le \frac{2}{3} \sum_{N\p \le \|k \| } \frac{ 1 }{ N\p} + \frac{3}{2} \sum_{N\p \le \|k \| } \frac{ \lambda_\pi(\p)^2 }{ N\p} = \frac{13}{6} \sum_{N\p \le \|k \| } \frac{ 1 }{ N\p} + \frac{3}{2} \sum_{N\p \le \|k \| } \frac{ \lambda_\pi(\p^2) }{ N\p}.
\end{equation*}

Using lemma \ref{sym2lower1} and $\sum_{N\p \le x} 1/N\p= \log \log x + O(1)$, the lemma follows.
\end{proof}

We may now prove the decay of $\langle \phi F_k, F_k \rangle$ for $\phi$ a Hecke-Maass cusp form or pure incomplete Eisenstein series.  We consider the Maass case first.  If $L(1, \sym^2 \pi) \ge (\log \| k \|)^{-7/15}$, then theorem \ref{Some} gives $\langle \phi F_k, F_k \rangle \ll (\log \| k \|)^{-1/30 + \epsilon}$.  Otherwise, from lemma \ref{Mkbound} we have that $M_k(\pi) \ll (\log \| k \|)^{-1/15 + \epsilon}$, and now theorem \ref{Home} gives $\langle \phi F_k, F_k \rangle \ll (\log \| k \|)^{-1/30 + \epsilon}$.  Therefore the bound of theorem \ref{weightmain} holds in either case.

In the Eisenstein case, we begin by showing how theorem \ref{Some} may be used to treat pure incomplete Eisenstein series in the cases where $L(1, \sym^2 \pi)$ is large.  By Mellin inversion, we may write

\begin{equation*}
E(\psi, m | z) = \frac{1}{2\pi i} \int_{(\sigma)} \Psi(-s) E(s, m, z ) ds,
\end{equation*}

where $\Psi$ is the Mellin transform of $\psi$.  We may move the line of integration to $\sigma = 1/2$ to obtain

\begin{equation*}
E(\psi, m | z) = \frac{1}{\text{Vol}(Y) } \langle E(\psi, m | z), 1 \rangle + \frac{1}{2\pi i} \int_{(1/2)} \Psi(-s) E(s, m, z ) ds,
\end{equation*}

and so

\begin{equation}
\label{piesprod}
\langle E(\psi, m | z) F_k, F_k \rangle = \frac{1}{\text{Vol}(Y) } \langle E(\psi, m | z), 1 \rangle + \frac{1}{2\pi i} \int_{(1/2)} \Psi(-s) \langle E(s, m, z) F_k, F_k \rangle ds.
\end{equation}

We now apply theorem \ref{Some} to obtain the bound

\begin{eqnarray*}
\int_{(1/2)} \Psi(-s) \langle E(s, m, z) F_k, F_k \rangle ds & \ll & \int_{\R} | \Psi( -1/2 - it) | \frac{ ( 1 + |t| + \| m \| )^{2n} }{ (\log \| k \|)^{1-\epsilon} L(1, \sym^2 \pi) } ds \\
& \ll & \frac{ (\log \| k \|)^{-1 +\epsilon} }{ L(1, \sym^2 \pi) }.
\end{eqnarray*}

It follows by substituting this in (\ref{piesprod}) that

\begin{equation*}
\left| \langle E(\psi, m | z) F_k, F_k \rangle - \frac{1}{\text{Vol}(Y) } \langle E(\psi, m | z), 1 \rangle \right| \ll  \frac{ (\log \| k \|)^{-1 +\epsilon} }{ L(1, \sym^2 \pi) }.
\end{equation*}

Therefore if $L(1, \sym^2 \pi) \ge (\log \| k \|)^{-13/15}$, we obtain the bound of theorem \ref{weightmain}.  If $L(1, \sym^2 \pi) < (\log \| k \|)^{-13/15}$, lemma \ref{Mkbound} gives $M_k(\pi) \ll (\log \| k \|)^{-4/15 + \epsilon}$.  Applying proposition \ref{Home} with the bound on $R_k(\pi)$ provided by lemma \ref{Rkbound}, we have

\begin{equation*}
\left| \langle E(\psi, m | z) F_k, F_k \rangle - \frac{1}{\text{Vol}(Y) } \langle E(\psi, m | z), 1 \rangle \right| \ll  ( \log \| k \| )^\epsilon M_k(\pi)^{1/2} \ll (\log \| k \|)^{-2/15 + \epsilon},
\end{equation*}

and so the bound of theorem \ref{weightmain} hold in this case also.

\section{Equidistribution of Zero Currents}
\label{currents}

This section contains the proof of theorem \ref{currentQUE} on the equidistribution of the zero divisors of holomorphic Hecke modular forms.  The proof is based on ideas from complex potential theory, which may be described in the general context of of a compact complex manifold $M$ with a positive holomorphic line bundle $L$.  Suppose that $L$ has been equipped with a Hermitian metric $h$, and let $\omega = c_1(h)$ be the associated K\"ahler form.  If $s_N \in H^0( M, L^N)$ are a sequence of $L^2$ normalised sections of $L^N$ whose mass becomes equidistributed on $M$, potential theory may be used to show that their normalised zero divisors $\tfrac{1}{N} Z_N$ tend weakly to $\omega$ in the sense of currents described in section \ref{currentQUE1}.  This was first discovered by Nonnenmacher and Vorros \cite{NV} in the context of quantum maps on tori, and extended in the generality described here by Schiffman and Zelditch \cite{SZ}.  If we now let $\HH^n$ denote the product of $n$ upper half planes and let $L_k$ be the line bundle of differentials of the form $f(z) \otimes_i dz_i^{k_i/2}$ on $\HH^n$, or its quotient by $\Gamma$, theorem \ref{currentQUE} is thus an extension of the result in \cite{SZ} to the bundle $L_k$ over the noncompact manifold $Y$.  To prove it we shall apply the argument of Schiffman and Zelditch, adding the adjustments of Rudnick \cite{Ru} to deal with the cusp.

We may give $L_k$ the natural Hermitian inner product $\| \otimes dz_i^{k_i/2} \|^2 = y^k$, whose associated K\"ahler form $\omega$ is

\begin{eqnarray*}
\omega & = & \frac{-i}{2\pi} \partial \overline{\partial} \log y^k \\
 & = & \frac{1}{4 \pi} \sum k_i y_i^{-2} dx_i \wedge dy_i.
\end{eqnarray*}

If $f$ is a holomorphic modular form of weight $k$, $f \otimes dz_i^{k_i/2}$ is then a section of $L_k$ with $\| f \otimes dz_i^{k_i/2} \|^2 = |f(z)|^2 y^k$.  We let $Z_f$ be the zero divisor of $f$ on $Y$, and $\widetilde{Z}_f$ its pullback to $\HH^n$.  For $\phi \in A^{n-1, n-1}(\HH^n)$ smooth and compactly supported, let 

\begin{equation*}
F_\phi = \sum_{\gamma \in \Gamma} \gamma^* \phi
\end{equation*}

be its symmetrisation under $\Gamma$.  If $f_N$ are a sequence of modular forms of weight $Nk$ as in theorem \ref{currentQUE}, we shall compare $\tfrac{1}{N} Z_N$ and $\omega$ by testing them against the differential forms $F_\phi$ using the following lemma.

\begin{lemma}
\label{lelong}
If $f$ is a holomorphic modular form of weight $k$ on $Y$,

\begin{equation*}
\int_{Z_f} F_\phi = \int_Y F_\phi \wedge \omega + \frac{i}{\pi} \int_{\HH^n} \log( y^{k/2} |f(z)| ) \partial \overline{\partial} \phi.
\end{equation*}

\end{lemma}

\begin{proof}

By unfolding $F_\phi$, we get

\begin{equation}
\label{currentunfold}
\int_{Z_f} F_\phi = \int_{ \widetilde{Z}_f } \phi.
\end{equation}

As $\widetilde{Z}_f$ is the zero divisor of the global holomorphic function $f$ on $\HH^n$, we may apply the Poincare-Lelong formula to the RHS of (\ref{currentunfold}), obtaining

\begin{eqnarray*}
\int_{Z_f} F_\phi & = & \frac{i}{\pi} \int_{\HH^n} \log |f(z)| \partial \overline{ \partial} \phi \\
 & = & - \frac{i}{\pi} \int_{\HH^n} \log y^{k/2} \partial \overline{ \partial} \phi + \frac{i}{\pi} \int_{\HH^n} \log ( y^{k/2} |f(z)| ) \partial \overline{ \partial} \phi.
\end{eqnarray*}

After integration by parts the first term becomes

\begin{equation*}
- \frac{i}{\pi} \int_{\HH^n} \partial \overline{ \partial} \log y^{k/2} \phi  = \int_{\HH^n} \omega \wedge \phi
\end{equation*}

and may be refolded to $\int_Y \omega \wedge F_\phi$, which completes the proof.

\end{proof}

After applying lemma \ref{lelong} to $\tfrac{1}{N} Z_N$, we are left with proving that $\tfrac{1}{N} \log ( y^{Nk/2} |f_N(z)| ) \overset{w^*}{\longrightarrow} 0$ locally everywhere.  As in \cite{Ru,SZ} this will follow from the plurisubharmonicity of $\log |f_N|$ and the equidistribution result $y^{Nk} |f_N(z)|^2 \overset{w^*}{\longrightarrow} c$, once we know that $\tfrac{1}{N} \log ( y^{Nk/2} |f_N(z)| )$ is is bounded above and has lim sup equal to 0, and that both properties hold locally uniformly.  Both of these are provided by the following lemma and the assumption (which we may clearly make) that the $f_N$ are $L^2$ normalised.

\begin{lemma}
\label{pointwise}

Let $f$ be a Hecke cusp form of weight $k$ for $\Gamma$.  Then uniformly for $z$ in compact subsets of $\HH^n$,

\begin{equation}
\frac{ y^k |f(z)|^2 }{ \| f \|^2 } \ll_\Gamma Nk^{5/2 + \epsilon}.
\end{equation}

\end{lemma}

\begin{proof}

Assume $\| f \|^2 = 1$.  We shall bound $|f|$ using its Fourier expansion

\begin{equation*}
f(z) = \sum_{\xi > 0} a_f(\xi) e( \tr( \xi \kappa z) ),
\end{equation*}

and the proportionality relation $a_f(\xi) = \lambda_\pi(\xi) a_f(1) \xi^{(k-1)/2}$ with $\lambda_\pi(\xi) \ll N\xi^\epsilon$.  Applying these and the normalisation of $a_f(1)$ from (\ref{holonorm}), we have

\begin{eqnarray*}
y^{k/2} |f(z)| & \le & \sum_{\xi > 0 } |a_f(\xi)| y^{k/2} \exp( -2\pi \tr( \xi \kappa y) ) \\
& \ll & \kappa^{k/2} Nk^\epsilon \prod_{i=1}^n \frac{ (4\pi)^{k_i/2} }{ \Gamma(k_i)^{1/2} } \sum_{\xi > 0 } \xi^{ (k-1)/2 + \epsilon} y^{k/2} \exp( -2\pi \tr( \xi \kappa y) ) \\
& = & Nk^\epsilon \prod_{i=1}^n \frac{ 1 }{ \Gamma(k_i)^{1/2} } \sum_{\xi > 0 } N\xi^{-1/2 + \epsilon} (4 \pi \xi \kappa y)^{k/2} \exp( -2\pi \tr( \xi \kappa y) ) \\
& \le & Nk^\epsilon \prod_{i=1}^n \frac{ 1 }{ \Gamma(k_i)^{1/2} } \sum_{\xi > 0 } (4 \pi \xi \kappa y)^{k/2} \exp( -2\pi \tr( \xi \kappa y) ).
\end{eqnarray*}

We define $g_i(x) = x^{k_i/2} e^{-x/2}$ and let $g: \R_+^n \rightarrow \R$ be the product function.  If we define $L$ to be the semi-lattice $4\pi \kappa y \OO \cap \R_+^n$, the upper bound above may be written

\begin{equation}
\label{potentialsum}
y^{k/2} |f(z)| \ll Nk^\epsilon \prod_{i=1}^n \frac{ 1 }{ \Gamma(k_i)^{1/2} } \sum_{x \in L} g(x).
\end{equation}

We now apply a lemma bounding the sum in (\ref{potentialsum}) in terms of various integrals of $g$.  Suppose $g_i$ is increasing for $x < t_i$ and decreasing for $x > t_i$.  Let $\mathcal{P}$ be the set of subsets of $\{ 1, \ldots, n \}$, and for $S \in \mathcal{P} $ define the subspace $H_S \in \R_+^n$ by

\begin{equation*}
H_S = \{ x \in \R_+^n \; | \; x_i = t_i, i \in S \}.
\end{equation*}

We then have the following bound (whose proof we omit) on $\sum_{x \in L} g(x)$.

\begin{lemma}
\label{integraltest}

\begin{eqnarray*}
\sum_{x \in L} g(x) & \ll & \sum_{S \in \mathcal{P} } \int_{H_S} g dv \\
 & = & \prod_{i=1}^n \left( \int_{\R^+} g(t) dt + g(t_i) \right),
\end{eqnarray*}

where the implied constant is bounded in compact families of lattices $L$.
\end{lemma}

Applying this to (\ref{potentialsum}) gives

\begin{eqnarray*}
y^{k/2} |f(z)| & \ll & Nk^\epsilon \prod_{i=1}^n \frac{ 1 }{ \Gamma(k_i)^{1/2} } \left( \int_{\R^+} x^{k_i/2} e^{-x/2} dx + k_i^{k_i/2} e^{-k_i/2} \right) \\
 & = & Nk^\epsilon \prod_{i=1}^n \frac{ 1 }{ \Gamma(k_i)^{1/2} } \left( 2^{k_i/2 + 1} \Gamma( k_i/2 + 1) + k_i^{k_i/2} e^{-k_i/2} \right) \\
 & \ll & Nk^\epsilon \prod_{i=1}^n ( k_i^{5/4} + k_i^{-3/4} ) \\
 & \ll & Nk^{5/4 + \epsilon}.
\end{eqnarray*}

The local uniformity of lemma \ref{integraltest} in $L$ gives the local uniformity of this bound, which completes the proof of lemma \ref{pointwise} and theorem \ref{currentQUE}.

\end{proof}

\section{Appendix}
\label{appendix}

We include here a number of routine calculations that were omitted during the proof of propositions \ref{Holo1} and \ref{Holomixed}.  These are the Fourier expansions of Eisenstein series over mixed number fields, the $L^2$ normalisations of cohomological automorphic forms, the calculation of the volume of $\F_1^\times / \OO_+^\times$ and the verification of the main term picked up in the contour shifts in lemma \ref{Iphimain} and equation (\ref{Iphimain2}).

\subsection{Fourier Expansions of Eisenstein Series}
\label{appeisenstein}

We recall the definition of $E(s, m, z)$ for $s \in \C$ and $m \in \Z^{r-1}$,

\begin{equation*}
E(s, m, z) = \sum_{ \gamma \in \Gamma_\infty \backslash \Gamma } N( y ( \gamma z ) )^s \lambda_m( y (\gamma z ) ).
\end{equation*}

We let $s_i = s + \beta(m, i)/\delta_i$, so that this may be rewritten

\begin{equation*}
E(s, m, z) = \sum_{ \gamma \in \Gamma_\infty \backslash \Gamma } \prod_{i=1}^r y_i( \gamma z )^{\delta_i s_i}.
\end{equation*}

The map sending $\gamma \in \Gamma_\infty \backslash \Gamma$ to its lower two entries is a bijection from $\Gamma_\infty \backslash \Gamma$ to the set of pairs $\{ c, d \}$ of relatively prime elements of $\OO$ modulo $\OO^\times$, and $y_i ( \gamma z )$ may be expressed in terms of this pair as

\begin{eqnarray*}
y_i( \gamma z ) & = & \frac{ y_i }{ | c_i z_i + d_i |^2 }, \quad i \le r_1\\
y_i( \gamma z ) & = & \frac{ y_i }{ |c_i|^2 y_i^2 + | c_i x_i + d_i |^2 }, \quad i > r_1.
\end{eqnarray*}

Therefore if we define $F(s,m,z)$ by

\begin{equation}
\label{Fdefine}
F(s,m,z) = \sum_{ \{ c, d \} } \prod_{i \le r_1} \frac{ y_i^{s_i} }{ | c_i z_i + d_i |^{2s_i} } \prod_{i > r_1} \frac{ y_i^{2s_i} }{ ( |c_i|^2 y_i^2 + | c_i x_i + d_i |^2 )^{2s_i} },
\end{equation}

where the sum is over all pairs $\{ c, d \}$ modulo $\OO^\times$, we have $F(s,m,z) = \zeta( 2s, \lambda_{-2m} ) E(s, m, z)$.  The $\xi$th Fourier coefficient of $F(s,m,z)$ is the integral

\begin{equation}
\label{coeffdefine}
a_\xi(s, m, y) = \frac{ 2^{r_2} }{ \sqrt{|D|} } \int_{\F / \OO} F( x + jy, s, m) e( - \tr( \xi \kappa x) ) dx.
\end{equation}

We begin by collecting the terms in (\ref{Fdefine}) with $c = 0$ to write

\begin{multline*}
F(s,m,z) = Ny^s \lambda_m(y) \zeta( 2s, \lambda_{-2m} ) + \sum_{ (c) } \frac{ Ny^s \lambda_m(y) }{ Nc^{2s} \lambda_{2m}(c) } \sum_{ d \text{ mod } (c) } \sum_{ \alpha \in \OO } \\
\prod_{i \le r_1} \frac{1}{ | z_i + \tfrac{d_i}{c_i} + \alpha_i |^{2s_i} } \prod_{i > r_1} \frac{1}{ ( y_i^2 + |x_i + \tfrac{d_i}{c_i} + \alpha_i|^2 )^{2s_i} }.
\end{multline*}

Substituting this into (\ref{coeffdefine}) and unfolding over $\OO$, we express $a_\xi(s, m, y)$ as

\begin{multline*}
a_\xi(s, m, y) = \delta_{\xi 0} Ny^s \lambda_m(y) \zeta( 2s, \lambda_{-2m} ) + \frac{ 2^{r_2} }{ \sqrt{|D|} } \sum_{ (c) } \frac{ Ny^s \lambda_m(y) }{ Nc^{2s} \lambda_{2m}(c) } \sum_{ d \text{ mod } (c) } \\
\prod_{i \le r_1} \int_\R \frac{ e( - \xi_i \kappa_i x_i ) dx_i }{ | z_i + \tfrac{d_i}{c_i} |^{2s_i} } \prod_{i > r_1} \int_\C \frac{ e( - \tr( \xi_i \kappa_i x_i ) ) dx_i }{ ( y_i^2 + |x_i + \tfrac{d_i}{c_i} |^2 )^{2s_i} }.
\end{multline*}

We first consider the case $\xi = 0$.

\begin{eqnarray*}
a_0(s, m, y) & = & Ny^s \lambda_m(y) \zeta( 2s, \lambda_{-2m} ) + \frac{ 2^{r_2} }{ \sqrt{|D|} } \sum_{ (c) } \frac{ Ny^s \lambda_m(y) }{ Nc^{2s} \lambda_{2m}(c) } \sum_{ d \text{ mod } (c) } \\
& \quad & \quad \quad \prod_{i \le r_1} \int_\R \frac{ dx_i }{ | z_i + \tfrac{d_i}{c_i} |^{2s_i} } \prod_{i > r_1} \int_\C \frac{ dx_i }{ ( y_i^2 + |x_i + \tfrac{d_i}{c_i} |^2 )^{2s_i} } \\
& = & Ny^s \lambda_m(y) \zeta( 2s, \lambda_{-2m} ) + \frac{ 2^{r_2} }{ \sqrt{|D|} } \sum_{ (c) } \frac{ Ny^{1-s} \lambda_{-m}(y) }{ Nc^{2s-1} \lambda_{2m}(c) } \\
& \quad & \quad \quad \prod_{i \le r_1} \int_\R \frac{ dx_i }{ ( 1 + x_i^2 )^{s_i} } \prod_{i > r_1} \int_\C \frac{ dx_i }{ ( 1 + |x_i |^2 )^{2s_i} }. \\
& = & Ny^s \lambda_m(y) \zeta( 2s, \lambda_{-2m} ) + \frac{ \pi^{n/2} }{ \sqrt{|D|} } Ny^{1-s} \lambda_{-m}(y) \zeta( 2s-1, \lambda_{-2m} ) \\
& \quad & \quad \quad \prod_{i \le r_1} \frac{ \Gamma( s + \beta(m,i) -1/2) }{ \Gamma( s + \beta(m,i) ) } \prod_{i > r_1} \frac{2}{ 2s + \beta(m,i) -1}.
\end{eqnarray*}

On dividing through by $\zeta(2s, \lambda_{-2m})$, this agrees with the expression given in section \ref{mixedrevise}.  When $\xi \neq 0$, we have

\begin{eqnarray*}
a_\xi(s, m, y) & = & \frac{ 2^{r_2} }{ \sqrt{|D|} } \sum_{ (c) } \frac{ Ny^s \lambda_m(y) }{ Nc^{2s} \lambda_{2m}(c) } \sum_{ d \text{ mod } (c) } \\
& \quad & \quad \quad \prod_{i \le r_1} \int_\R \frac{ e( - \xi_i \kappa_i x_i ) dx_i }{ | z_i + \tfrac{d_i}{c_i} |^{2s_i} } \prod_{i > r_1} \int_\C \frac{ e( - \tr( \xi_i \kappa_i x_i ) ) dx_i }{ ( y_i^2 + |x_i + \tfrac{d_i}{c_i} |^2 )^{2s_i} } \\
& = & \frac{ 2^{r_2} }{ \sqrt{|D|} } \sum_{ (c) } \frac{ Ny^{1-s} \lambda_{-m}(y) }{ Nc^{2s} \lambda_{2m}(c) } \sum_{ d \text{ mod } (c) } e( \tr( \tfrac{ \xi \kappa d }{c} ) )\\
& \quad & \quad \quad \prod_{i \le r_1} \int_\R \frac{ e( - \xi_i \kappa_i y_i x_i ) dx_i }{ ( 1 + x_i^2 )^{s_i} } \prod_{i > r_1} \int_\C \frac{ e( - \tr( \xi_i \kappa_i y_i x_i ) ) dx_i }{ ( 1 + |x_i |^2 )^{2s_i} }. \\
\end{eqnarray*}

The integral at real places is equal to

\begin{equation*}
\frac{ 2 \pi^{s_i} }{ \Gamma(s_i) } ( \xi_i \kappa_i y_i )^{s_i - 1/2} K_{s_i - 1/2}( 2\pi |\xi_i \kappa_i | y_i ),
\end{equation*}

and the integral at complex places may be calculated in the following way as in \cite{Sr}.

\begin{eqnarray*}
\int_\C \frac{ e( - \tr( \xi_i \kappa_i y_i x_i ) ) dx_i }{ ( 1 + |x_i |^2 )^{2s_i} } & = & \int_0^\infty \int_0^{2\pi} \frac{ e( -2 y_i r |\xi_i \kappa_i| \sin ( \theta + \alpha) ) }{ ( r^2 + 1 )^{2s_i} } r d\theta dr \\
& = & \int_0^\infty \frac{r}{ (r^2 + 1)^{2s_i} } \int_0^{2\pi} e( -2 y_i r |\xi_i \kappa_i| \sin \theta ) d\theta dr \\
& = & \int_0^\infty \frac{J_0( 4 \pi r |\xi_i \kappa_i| y_i ) }{ (r^2 + 1)^{2s_i} } dr \\
& = & \frac{ ( 4\pi |\xi_i \kappa_i| y_i )^{2s_i-1} }{ \Gamma( 2s_i) 2^{2s_i -1} }  K_{2s_i-1}( 4\pi |\xi_i \kappa_i| y_i ).
\end{eqnarray*}

(See \cite{GR} for the evaluation of the final integral.)  It can be seen from this that the final form of $a_\xi(s,m,y)$ is the product of a collection of Bessel functions and Gamma factors which agree with the formula for $E(s, m, z)$ of section \ref{mixedrevise}, together with a constant term and a power of $y$ which are given below

\begin{equation*}
\frac{ 2^{r_2} }{ \sqrt{|D|} } Ny^{1-s} \lambda_{-m}(y) \sum_{ (c) } \sum_{ d \text{ mod } (c) } e( \tr( \tfrac{ \xi \kappa d }{c} ) ) \prod_{i \le r_1} 2 \pi^{s_i} ( \xi_i \kappa_i y_i )^{s_i - 1/2} \prod_{i > r_1} \frac{ ( 4\pi |\xi_i \kappa_i| y_i )^{2s_i-1} }{ 2^{2s_i -1} }.
\end{equation*}

The power of $y$ simplifies to $\sqrt{Ny}$, while the constant may be simplified as

\begin{eqnarray*}
& & \frac{ 2^{r_2} }{ \sqrt{|D|} } \sum_{ (c) } \sum_{ d \text{ mod } (c) } e( \tr( \tfrac{ \xi \kappa d }{c} ) ) \prod_{i \le r_1} 2 \pi^{s_i} ( \xi_i \kappa_i )^{s_i - 1/2} \prod_{i > r_1} ( 2 \pi |\xi_i \kappa_i| )^{2s_i-1} \\
& & \quad \quad = \frac{ 2^r \pi^{ns-r_2} }{ \sqrt{|D|} } \sigma_{1-2s, -2m}(\xi \kappa) N(\xi \kappa)^{s-1/2} \lambda_m( \xi \kappa) \prod_{i > r_1} 2^{2s_i-1} \\
& & \quad \quad = \frac{ 2^r \pi^{ns-r_2} }{ \sqrt{|D|} } \sigma_{1-2s, -2m}(\xi \kappa) N(\delta \xi \kappa)^{s-1/2} \lambda_m(\delta \xi \kappa).
\end{eqnarray*}

After dividing through by $\zeta(2s, \lambda_{-2m})$, both of these terms agree with the expression in section \ref{mixedrevise}.

\subsection{$L^2$ Normalisations}
\label{appnorms}

This section contains the calculation of the $L^2$ normalisations of the Fourier coefficients of our forms $F_k$.  The normalisations are based on the equation

\begin{equation}
\label{eisnormal}
\Res_{s=1} \langle E(s,z) F_k, F_k \rangle = \Res_{s=1} \phi(s) \langle F_k, F_k \rangle,
\end{equation}

where $E(s,z) = E(s,0,z)$ and $\phi(s)$ is the scattering coefficient in the constant term.  $\Res_{s=1} \phi(s)$ is given by

\begin{eqnarray*}
\Res_{s=1} \phi(s) & = & \frac{ \pi^{n/2} }{ \sqrt{|D|} } \frac{ \Res_{s=1} \zeta_F(s) }{ 2 \zeta_F(2) } \prod_{i \le r_1} \frac{ \Gamma( 1/2 ) }{ \Gamma(1) } \prod_{i > r_1} 2 \\
& = & \frac{ 2^{r_2-1} \pi^{(n + r_1)/2} \Res_{s=1} \zeta_F(s) }{ \sqrt{|D|} \zeta_F(2) }.
\end{eqnarray*}

We then calculate $\langle E(s,z) F_k, F_k \rangle$ by unfolding and compare the two sides of (\ref{eisnormal}).

\begin{eqnarray*}
\langle E(s,z) F_k, F_k \rangle & = & \int_{ \Gamma_\infty \backslash \HH_F } Ny^s |F_k|^2 dv \\
& = & |a_f(1)|^2 \frac{ \sqrt{|D|} }{ 2^{r_2} \omega_+ } \int_{ \R_+^r / \OO_+^\times } Ny^{s-1} \sum_{ \eta \in \OO^+ } N\eta^{-1}  |\lambda_\pi(\eta)|^2 | {\bf K}_k(\eta \kappa y) |^2 dy^\times \\
& = & |a_f(1)|^2 \frac{ \sqrt{|D|} }{ 2^{r_2} } \sum_{ (\eta) } N\eta^{-1}  |\lambda_\pi(\eta)|^2  \int_{ \R_+^r } Ny^{s-1} |{\bf K}_k(\eta \kappa y) |^2 dy^\times.
\end{eqnarray*}

Note that the factor of $\omega_+$ vanished because $\OO^+ / \OO_+^\times$ counts each ideal with multiplicity $\omega_+$.

\begin{eqnarray*}
\langle E(s,z) F_k, F_k \rangle & = & |a_f(1)|^2 \frac{ \sqrt{|D|} N\kappa^{1-s} }{ 2^{r_2} } \sum_{ (\eta) } \frac{ |\lambda_\pi(\eta)|^2 }{ N\eta^s } \int_{ \R_+^r } Ny^{s-1} |{\bf K}_k(y) |^2 dy^\times \\
& = & |a_f(1)|^2 \frac{ |D|^{s-1/2} }{ 2^{r_2} } L( s, \sym^2 \pi ) \frac{ \zeta_F(s) }{ \zeta_F(2s) } \int_{ \R_+^r } Ny^{s-1} |{\bf K}_k(y) |^2 dy^\times.
\end{eqnarray*}

We only need the value of the integral at $s=1$, and to calculate it we expand it as a product over the infinite places.  The factor at a real place is

\begin{equation*}
\int_0^\infty y^{k_i} \exp( -4\pi k_i y ) dy^\times = (4\pi)^{-k_i} \Gamma( k_i ).
\end{equation*}

At a complex place, it is

\begin{eqnarray*}
& & \int_0^\infty y^{k_i + 2} \sum_{j=0}^{k_i} \binom{k_i}{j} K_{k_i/2 - j}^2 ( 4\pi y ) dy^\times \\
& & \quad \quad = (4\pi)^{-k_i - 2} \sum_{j=0}^{k_i} \binom{k_i}{j} \int_0^\infty y^{k_i + 2} K_{k_i/2 - j}^2 ( y ) dy^\times \\
& & \quad \quad =  (4\pi)^{-k_i - 2} \frac{ 2^{k_i-1} \Gamma( 1 + k_i/2 )^2 }{ \Gamma( 2 + k_i ) } \sum_{j=0}^{k_i} \binom{k_i}{j} \Gamma( 1 + j) \Gamma( 1 + k_i -j ) \\
& & \quad \quad =  (4\pi)^{-k_i - 2} \frac{ 2^{k_i-1} \Gamma( 1 + k_i/2 )^2 }{ \Gamma( 2 + k_i ) } (k_i+1)! \\
& & \quad \quad =  2^{-5} \pi^{-2} (2\pi)^{-k_i} \Gamma( 1 + k_i/2 )^2.
\end{eqnarray*}

Combining these, we have the following expression for $\Res_{s=1} \langle E(s,z) F_k, F_k \rangle $:

\begin{multline*}
\Res_{s=1} \langle E(s,z) F_k, F_k \rangle = |a_f(1)|^2 \frac{ \sqrt{|D|} }{ 2^{6r_2} \pi^{2r_2} } L( 1, \sym^2 \pi ) \frac{ \Res_{s=1} \zeta_F(s) }{ \zeta_F(2) } \\
\prod_{i \le r_1} (4\pi)^{-k_i} \Gamma( k_i ) \prod_{i > r_1} (2\pi)^{-k_i} \Gamma( 1 + k_i/2 )^2.
\end{multline*}

Dividing by $\Res_{s=1} \phi(s)$ we obtain the required relation between $|a_f(1)|^2$ and $\langle F_k, F_k \rangle$,

\begin{equation*}
\langle F_k, F_k \rangle = |a_f(1)|^2 \frac{ |D| L( 1, \sym^2 \pi ) }{ 2^{7r_2-1} \pi^{r_1 + 3r_2} } \prod_{i \le r_1} (4\pi)^{-k_i} \Gamma( k_i ) \prod_{i > r_1} (2\pi)^{-k_i} \Gamma( 1 + k_i/2 )^2.
\end{equation*}

\subsection{Volume Computations}
\label{appvol}

In this section we compute the volume element in the cusp of $Y$, and use this with our computation of the residue of $E(s,z)$ to calculate the volume of $Y$.  As in Efrat \cite{Ef}, we shall introduce simplified co-ordinates in the cusp, defined using the matrix $A$ from section \ref{weightprelims2}.  We define the co-ordinates $Y_0, \ldots, Y_{r-1}$ by

\begin{equation*}
\left( \begin{array}{c} \ln Y_0 \\ Y_1 \\ \vdots \\ Y_{r-1} \end{array} \right) = 
\left( \begin{array}{cccc} 1 & 1 & \ldots & 2 \\
 e_1^1 & e_2^1 & \ldots & e_n^1 \\
 \vdots & & & \\
 e_1^{n-1} & e_2^{n-1} & \ldots & e_n^{n-1} 
 \end{array} \right)
\left( \begin{array}{c} \ln y_1 \\ \ln y_2 \\ \vdots \\ \ln y_r \end{array} \right),
\end{equation*}

so that

\begin{equation*}
\left( \begin{array}{c} \ln y_1 \\ \ln y_2 \\ \vdots \\ \ln y_r \end{array} \right) =
\left( \begin{array}{cccc} 1/n & \log | \epsilon_1^1 | & \ldots & \log | \epsilon_{r-1}^1 | \\
 \vdots & & & \\
 1/n & \log | \epsilon_1^r | & \ldots & \log | \epsilon_{r-1}^r | 
 \end{array} \right)
\left( \begin{array}{c} \ln Y_0 \\ Y_1 \\ \vdots \\ Y_{r-1} \end{array} \right).
\end{equation*}

We shall compute the volume form of $\HH_F$ with respect to the new system of co-ordinates $Y_0, \ldots, Y_{r-1}, x_1, \ldots, x_r$.  As we are not changing the $x$-coordinates at all we may omit them from our calculations, and only compute the form $\bigwedge_i dy_i/y_i$ with respect to $\{ Y_i \}$.  The Jacobian of the change of co-ordinates is

\begin{equation*}
\left( \begin{array}{c} \partial y_1 \\ \partial y_2 \\ \vdots \\ \partial y_r \end{array} \right) =
\left( \begin{array}{cccc} \tfrac{ y_1}{ nY_0} & y_1 \log | \epsilon_1^1 | & \ldots & y_1 \log | \epsilon_{r-1}^1 | \\
 \vdots & & & \\
 \tfrac{ y_r}{ nY_0} & y_r \log | \epsilon_1^r | & \ldots & y_r \log | \epsilon_{r-1}^r | 
 \end{array} \right)
\left( \begin{array}{c} \partial Y_0 \\ \partial Y_1 \\ \vdots \\ \partial Y_{r-1} \end{array} \right),
\end{equation*}

and we need to calculate its determinant which is

\begin{equation*}
\frac{1}{nY_0} \prod_{i=1}^r y_i \det 
\left( \begin{array}{cccc} 1 & \log | \epsilon_1^1 | & \ldots & \log | \epsilon_{r-1}^1 | \\
 \vdots & & & \\
 1 & \log | \epsilon_1^r | & \ldots & \log | \epsilon_{r-1}^r | 
 \end{array} \right).
\end{equation*}

We shall calculate this determinant by a minor expansion along the first column.  The absolute value of the determinant of the $(1,i)$th minor is the regulator $R^+$ of $\OO_+^\times$ times $1/2$ for every complex place we are expanding over.  The index of $\OO_+^\times$ in $\OO^\times$ is $2^{r_1 - 1 + \delta_{r_1 0} }$ so $R^+ = 2^{r_1 - 1 + \delta_{r_1 0} } R$, and the alternating sum of the minors is $(r_1 + 2r_2) 2^{r_1 -r_2 - 1 + \delta_{r_1 0} } R = n 2^{r_1 -r_2 - 1 + \delta_{r_1 0} } R$.  The expression for $dv$ in terms of our new co-ordinate system is therefore

\begin{equation*}
dv = \frac{ 2^{r_1 -r_2 - 1 + \delta_{r_1 0} } R }{ Y_0^2 } \bigwedge_{i=0}^{r-1} dY_i \wedge dx.
\end{equation*}

We now verify that the main term appearing in $I_\phi(T)$ during the contour shift in lemma \ref{Iphimain} and equation (\ref{Iphimain2}) is in fact $\langle E(z|g), 1 \rangle / \text{Vol}(Y) \langle \phi F_k, F_k \rangle$.  The residue is equal to

\begin{equation*}
V_c^{-1} G(-1,0) \Res_{s=1} \phi(s) \langle \phi F_k, F_k \rangle,
\end{equation*}

and it may easily be seen that $\langle E(z|g), 1 \rangle = 2^{-r_2} \omega_+^{-1} \sqrt{|D|} G(-1,0)$.  Therefore to show that the two expressions are equal we only need to show that $\text{Vol}(Y) = V_c 2^{-r_2} \omega_+^{-1} \sqrt{|D|} ( \Res_{s=1} \phi(s) )^{-1}$.  This follows easily from the standard method of computing the volume of fundamental domains using Eisenstein series, and substituting the value of $\Res_{s=1} \phi(s)$ gives

\begin{equation*}
\text{Vol}(Y) = \frac{ 2^{-4r_2 + 1} |D|^{3/2} \zeta_F(2) }{ \pi^n }.
\end{equation*}

\end{document}